\documentclass[11pt,a4paper,english,reqno]{amsart}
\usepackage{amsmath,amscd,amssymb, appendix, lineno, mathtools}

\usepackage{amsmath}
\usepackage{amsthm}
\usepackage{amssymb}
\usepackage{graphicx}
\usepackage{epsfig}
\usepackage{color}

\numberwithin{equation}{section} 

\newtheorem{Theorem}{Theorem}[section]
\newtheorem{Corollary}[Theorem]{Corollary}
\newtheorem{Lemma}[Theorem]{Lemma}
\newtheorem{Proposition}[Theorem]{Proposition}

\theoremstyle{remark}
\newtheorem{rmk}{Remark}[section]

\theoremstyle{definition}

\renewcommand{\tilde}{\widetilde}
\renewcommand{\hat}{\widehat}

\newcommand{\nn}{\nonumber}
\newcommand{\R}{{\mathbb R}}

\newcommand{\N}{{\mathbb N}}

\newcommand{\Grad}{\nabla_{\!x}}
\newcommand{\del}{\partial}
\newcommand{\dx}{ \, {\rm d} x}
\newcommand{\dy}{ \, {\rm d} y}
\newcommand{\dz}{ \, {\rm d} z}
\newcommand{\dt}{ \, {\rm d} t}
\newcommand{\dv}{ \, {\rm d} v}

\newcommand{\dxi}{\, {\rm d} \xi}
\newcommand{\dtheta}{\, {\rm d} \Vtheta}
\newcommand{\dtau}{\, {\rm d} \tau}

\newcommand{\One}{\boldsymbol{1}}
\newcommand{\La}{\left\langle}
\newcommand{\Ra}{\right\rangle}

\newcommand{\Ss}{{\mathbb{S}}}

\newcommand{\Disp}{\displaystyle}
\newcommand{\Denote}{\stackrel{\triangle}{=}}

\newcommand{\CalC}{{\mathcal{C}}}

\newcommand{\CalF}{{\mathcal{F}}}

\newcommand{\CalJ}{{\mathcal{J}}}
\newcommand{\CalK}{{\mathcal{K}}}
\newcommand{\CalL}{{\mathcal{I}}}

\newcommand{\CalR}{{\mathcal{R}}}
\newcommand{\CalS}{{\mathcal{S}}}

\newcommand{\Vtheta}{{\theta}}

\newcommand{\Jac}{J}

\newcommand{\UJ}{u_{\mathcal{J}}}
\newcommand{\WJ}{w_{\mathcal{J}}}
\newcommand{\WJD}{w_{\mathcal{J}, \delta}}
\newcommand{\WJL}{w_{\mathcal{J}, \lambda}^+}

\newcommand{\abs}[1]{\left\lvert#1\right\rvert}
\newcommand{\norm}[1]{\left\lVert#1\right\rVert}

\newcommand{\vint}[1]{\langle #1 \rangle}
\newcommand{\vpran}[1]{\left( #1 \right)}

\date{\today}
\title[RTE in the forward-peaked regime]{The radiative transfer equation in the forward-peaked regime}
\author{Ricardo Alonso}
\address{R. Alonso -- Departamento de Matem\'{a}tica, PUC-Rio, Rua Marqu\^{e}s de S\~{a}o Vicente 225, Rio de Janeiro, CEP 22451-900, Brazil. }
\email{ralonso@mat.puc-rio.br}
\author{Weiran Sun}
\address{W. Sun -- Department of Mathematics, Simon Fraser University, 8888 University Drive, Burnaby, BC V5A 1S6.}
\email{weirans@sfu.ca}

\begin{document}
\maketitle
\begin{abstract}
In this work we study the radiative transfer equation in the forward-peaked regime in free space. Specifically, it is shown that the equation is well-posed by proving instantaneous regularization of weak solutions for arbitrary initial datum in $L^{1}$.  Classical techniques for hypo-elliptic operators such as averaging lemma are used in the argument.  Among the interesting aspect of the proof are the use of the stereographic projection and the presentation of a rigorous expression for the scattering operator given in terms of a fractional Laplace-Beltrami operator on the sphere, or equivalently, a weighted fractional Laplacian analog in the projected plane.  Such representations may be used for accurate numerical simulations of the model.  As a bonus given by the methodology, we show convergence of Henyey-Greenstein scattering models and vanishing of the solution at time algebraic rate due to scattering diffusion.    
\end{abstract}
\section{Introduction}
\subsection{Radiative transfer equation and the highly forward-peaked regime}
Radiative transfer is the physical phenomenon of energy transfer in the form of electromagnetic radiation. The propagation of radiation through a medium is described by absorption, emission, and scattering processes.  In the case that the medium is free of absorption and emission the radiative transfer equation (RTE) in free space reduces to  
\begin{equation}\label{rte1}
\left\{
\begin{array}{cl}
 \del_t u + \theta\cdot\nabla_{x} u
= \CalL(u) \,, & \text{in}\;\; (0,T)\times\mathbb{R}^{d} \times\mathbb{S}^{d-1}
\\
   u  =  u_o \,, & \text{on}\;\; \{t=0\}\times\mathbb{R}^{d} \times\mathbb{S}^{d-1}\,,
\end{array}\right.
\end{equation}
where $u=u(t,x,\theta)$ is understood as the radiation distribution in the free space $(0,T)\times\mathbb{R}^{d}\times\mathbb{S}^{d-1}$ where $\mathbb{S}^{d-1}$ stands for the unit sphere in $\mathbb{R}^{d}$.  The initial radiation distribution is assumed nonnegative and $u_o\in L^{1}\big(\mathbb{R}^{d}\times\mathbb{S}^{d-1}\big)$.  The scattering operator is global only in the radiation propagation direction $\theta$, specifically, it reads simply as
\begin{equation}\label{rte2}
\mathcal{I}(u):=\mathcal{I}_{b_{s}}(u)=\int_{\mathbb{S}^{d-1}}\big(u(\theta')-u(\theta)\big)b_s(\theta,\theta')\text{d}\theta'.
\end{equation} 
It is commonly assumed that the angular scattering kernel has the symmetry $0\leq b_s(\theta,\theta')=b_{s}(-\theta',-\theta)$ due to micro--irreversibility and has the normalized integrability condition
\begin{equation}\label{skca}
1=\int_{\mathbb{S}^{d-1}}b_s(\theta,\theta')\text{d}\theta'=\int_{\mathbb{S}^{d-1}}b_s(\theta',\theta)\text{d}\theta'.
\end{equation}
For detailed presentations of the mathematical theory of linear transport equation with the classical assumption \eqref{skca} refer to \cite[Chapter XXI]{DL}.  In this work we are interested in a different regime of propagation called \textit{highly forward-peaked regime} commonly found in neutron transport, atmospheric radiative transfer and optical imaging among others.  Refer to \cite{Bal} for a general discussion of the RTE, including the forward-peaked regime, with application to inverse problems.  In this regime and under precise scaling, see below, the angular scattering kernel is formally approximated by
\begin{equation}\label{skA1}
b_{s}(\theta,\theta')=\frac{ b(\theta\cdot\theta') }{ \big(1- \theta\cdot\theta' \big)^{\frac{d-1}{2}+s} },\quad s\in\big(0,\min\{1,\tfrac{d-1}{2}\}\big)
\end{equation}
where the function $b(z)\geq0$ enjoys some smoothness in the vicinity of $z=1$.  More precisely, in the sequel we will consider its decomposition as
\begin{equation}\label{skA2}
   b(z) = b(1) + \tilde{b}(z)\,, \quad\text{where}
\quad 
   h(z) = \frac{\tilde{b}(z)}{(1-z)^{1+s}}\in L^{1}(-1,1)\,.
\end{equation}
For instance, some H\"{o}lder continuity in the vicinity of $z=1$ will suffice for $\tilde{b}(z)$.  This decomposition is commonly used to separate the peaked regime scattering from others such as Rayleigh.  In scattering physics literature it is common to use the Henyey-Greenstein angular scattering kernel (also called phase function), first introduced in \cite{HG}, which for $d=3$ reads 
\begin{equation*}
b^{g}_{HG}(\theta,\theta')=\frac{1-g^{2}}{\big(1+g^{2}-2\,g\,\theta\cdot\theta'\big)^{ \frac{3}{2} } } \,,
\end{equation*}
where the anisotropic factor $g\in(0,1)$ measures the strength of forward-peakedness of the scattering kernel. For example, typical values for this factor in animal tissues are in the range $0.9\leq g \leq 0.99$, in such a case the regime is referred as \textit{highly peaked}.  Therefore, the model \eqref{skA1} can be viewed (but not restricted) as the limit $g\rightarrow1$ of Henyey-Greenstein scattering type (with $s=\frac{1}{2}$) after proper rescaling.  Indeed, assume that $u_{HG}$ is the solution of radiative transfer \eqref{rte1} with initial condition $u^{o}_{HG}$ and with the Henyey-Greenstein phase function.  Define the rescaled function $u^{g}$ as
\begin{equation*}
u^{g}(t,x,\theta) = \tfrac{1}{(1-g)^{d}}\,u_{HG}\big(\tfrac{t}{1-g},\tfrac{x}{1-g},\theta\big)\,,
\end{equation*}
where the time-space variables $(t,x)$ are order one quantities.  Thus, this rescaling is introduced in order to observe large spatial-time dynamics (of the original problem) so that the highly forward-peaked scattering has a visible effect.  It can be interpreted as a diffusive scaling of the type given in \cite{LK} for propagation regimes with a small mean free path.  Note that the factor $1/(1-g)^{d}$ is necessary to conserve the solution's mass.  A simple computation shows that $u^{g}$ solves the radiative transfer equation \eqref{rte1} with phase function given by
\begin{equation*}
b^{g}_{s}(\theta,\theta') = \frac{b^{g}_{HG}(\theta,\theta')}{1-g}=\frac{1+g}{\big(1+g^{2}-2\,g\,\theta\cdot\theta'\big)^{ \frac{3}{2} } } \underset{g\rightarrow1}{\longrightarrow} \frac{1}{\sqrt{2}\,\big(1-\theta\cdot\theta'\big)^{ \frac{3}{2} } }\,,
\end{equation*}
\medskip
and initial condition $u^{g}_o=u^{o}_{HG}$.  Therefore, it is expected that in some suitable sense the asymptotic limit $u=\lim_{g\rightarrow 1}u^{g}$ is given by a radiation distribution $u$ that solves \eqref{rte1} with phase function \eqref{skA1} as long as the rescaled initial condition converges towards $u_o$.  Such asymptotic limits are usually referred as Fokker Planck approximations since the distribution $u$ solves essentially a Fokker-Planck equation, see for example the references \cite{P}, \cite{LL}, \cite{Bal} and \cite{QW} which present instances of this approach.  It was noticed in \cite{P} that the sequence of solutions $u_{HG}$ cannot converge (as $g\rightarrow1$) to the solution of a Fokker-Planck equation, therefore, in principle some diffusion scaling that depends on the propagation regime is necessary for this to happen.  In fact, we will show in this work that the limiting scattering mechanism is not given by a Laplace Beltrami operator in the sphere but rather a \textit{fractional} Laplace Beltrami operator.  Thus, in the case of the classical Henyey-Greenstein scattering, Our work provides a rigorous justification of the asymptotic analysis in~\cite{P}. It also shows that a more precise name for the asymptotic limit would be ``fractional Fokker Planck approximation''.  Independent of the name used for the approximation, the important underlying issue is that using a standard Fokker-Planck equation may not be entirely appropriate for the correct modeling of the highly forward-peaked regime.\\

\noindent
Observe that assumptions \eqref{skA1} and \eqref{skA2} imply that  
\begin{equation*}
     \int_{\Ss^{d-1}} b_s(\theta,\theta') \text{d}\theta' = +\infty 
\qquad
    \text{for any $s > 0$} \,,
\end{equation*}
therefore, the operator $\mathcal{I}$ is not well defined unless the radiation distribution $u$ enjoys sufficient regularity, say having two continuous derivatives in the variable $\theta$.  Such regularity needs to be proven for solutions of the radiative transfer equation \eqref{rte1} in the highly forward-peaked regime.  Consequently, the interaction operator is defined using the weak formulation:  For any sufficiently regular functions $u$ and $\psi$
\begin{align}\label{wf}
\int_{\mathbb{S}^{d-1}}\mathcal{I}(u)(\theta)\,\psi(\theta)\,\text{d}\theta :&= -\tfrac{1}{2}\int_{\mathbb{S}^{d-1}}\int_{\mathbb{S}^{d-1}} \big(u(\theta') - u(\theta)\big) \big(\psi(\theta') - \psi(\theta) \big)\,b_{s}(\theta,\theta') \text{d}\theta' \text{d}\theta\,\nonumber\\
&=\lim_{\epsilon\rightarrow0}\int_{\mathbb{S}^{d-1}}u(\theta)\int_{\{1-\theta\cdot\theta'\geq\epsilon\}} \big(\psi(\theta') - \psi(\theta) \big)\,b_{s}(\theta,\theta') \text{d}\theta' \text{d}\theta\,.
\end{align}
Although, we are not yet precise what the space of test functions is, we observe that equations \eqref{wf} is equivalent to the strong formulation \eqref{rte2} for sufficiently regular $u$.  In Proposition \ref{main:smooth} a explicit expression in terms of the fractional laplacian will be given.
\subsection{Definition of solution, results and organization of the proof} Let the function $u_o\in L^{1}_{x,\theta}$ \footnote{The shorthand $L^{1}_{x,\theta}$ denoting $L^{1}(\mathbb{R}^{d}\times\mathbb{S}^{d-1})$ and its equivalent to other Lebesgue and Sobolev spaces will be used extensively along the paper.} be a nonnegative initial state and $T>0$ be an arbitrary time.  A nonnegative function
\begin{equation*}
u\in L^{\infty}\big([0,T);L^{1}_{x,\theta}\big)\cap \CalC\big([0,T);L^{1}_{loc}\big)
\end{equation*}
is a solution of the RTE in the (highly) forward peaked regime with initial condition $u_o$ provided that
\begin{equation*}
\partial_{t}u\,,\,\nabla_{x}u, \, \CalL_{b_s}(u) \in L^{2}\big([t_o,T); L^{2}_{x,\theta} \big) \;\quad \forall\, t_o>0\,,
\end{equation*}
and $u$ solves the RTE equation a.e.
\begin{equation}\label{DWS}
\left\{
\begin{array}{cl}
    \del_t u + \theta\cdot\nabla_{x} u
= \CalL_{b_s}(u) \,, & \text{in}\;\; (0,T)\times\mathbb{R}^{d} \times\mathbb{S}^{d-1}
\\
   u  =  u_o \,, & \text{on}\;\; \{t=0\}\times\mathbb{R}^{d} \times\mathbb{S}^{d-1}\,.
\end{array}\right.
\end{equation}
Let us state the results in one theorem.  The detailed statement of the results with precise estimates and spaces can be found in Sections 4 and 5.
\begin{Theorem}\label{TMain}
(1) (Stability and existence of solutions) Consider a sufficiently regular nonnegative initial state $u_o\in L^{2}_{x,\theta}$ and let $\{u^{g}\}_{g\geq0}$ a sequence of rescaled solutions of the RTE with Henyey-Greenstein kernels having such initial state.  Then,  $\{u^{g}\}_{g\geq0}$ converges weakly in $L^{2}\big([0,T);L^{2}_{x,\theta}\big)$ as $g\rightarrow1$ to the unique solution $u\geq0$ of the RTE in the forward-peaked regime having initial condition $u_o$.\\
(2) (Existence of solution for general initial state) Consider a nonnegative initial state $u_o \in L^{1}_{x,\theta}$.  Then, the RTE in the forward-peaked regime has a unique smooth solution $u\geq0$.  Furthermore, all higher norms of $u$ are controlled exclusively in terms of $m_o$, the mass of $u_o$, for any positive time.\\
(3) (Time asymptotic vanishing) Consider a nonnegative initial state $u_o \in L^{1}_{x,\theta}$.  Then, the solution $u$ of the RTE in the forward-peaked regime satisfies for any $t>0$
\begin{equation*}
u(t)\leq \frac{C(m_o)}{t^{a}}\,,\quad \text{for some universal }\; a\geq \tfrac{1}{2}\,,
\end{equation*}
and constant $C(m_o)$ depending only on the initial mass $m_o$.
\end{Theorem}
\noindent
Theorem \ref{TMain} is proved in Sections 4 (items (1) and (2)) and Section 5 (item (3)).  The proof is based on Section 3 where all the \textit{a priori} estimates are worked out.  Section 2 is of independent importance and contains the averaging lemmas that propagate regularity from the angular variable to the spatial variable.  More precisely, the proof follows the following argument:  Assume existence of a solution (as defined previously) for the RTE in the peaked regime.  For such solution, the main energy estimate \eqref{MEE2} is valid.  Such estimate essentially points out that higher angular Sobolev regularity is controlled in terms solely of the $L^{2}_{x,\theta}$-norm of the solution.  Although, the control of a higher \textit{spatial} Sobolev norm is not explicit in estimate \eqref{MEE2}, it is possible to propagate a fraction of such angular regularity to the spatial variable using hypoelliptic methods.  In particular, we choose to follow in Section 2 a flexible and powerful technique based on the so-called average lemmas, see \cite{FB} and \cite{FD} for a complete discussion and an extensive list of references in the topic.  This section ends with Corollary \ref{Cor:strong-reg-avg-time} which states precisely this fact.  In Section 3, a classical technique in parabolic PDE theory is used, namely, to show successively improved regularity in the solution starting from the lowest conserved quantity, in this case the $L^{1}_{x,\theta}$-norm, we refer to \cite{CH} to observe such technique in the context of nonlinear integral equations.  Thus, Section 3 starts proving the basic control of the $L^{2}_{x,\theta}$-norm in terms of the $L^{1}_{x,\theta}$-norm.  Such result only requires a standard version of the average lemma given in Proposition \ref{Lemma:VA-Weak}.  Improvement of regularity, involving Sobolev norms in both space and angle, is done in Proposition \ref{prop:infty-reg} by differentiating successively the equation and arguing by induction.  The initial step of the induction is given by the strong form of the average lemma proved in Theorem \ref{st-reg-lemma}.  All the results up to Section 3 are valid assumed the existence of a solution, thus, Section 4 is dedicated to show the existence of such solution.  To this end, the RTE in the peaked regime is approximated using the physical model, namely, the rescaled RTE with Henyey-Greenstein type of scattering.  Of course, it is possible to approximate the RTE in the forward peaked regime in many ways (including simpler ones), we choose the  Henyey-Greenstein type for its physical relevance.  Uniform estimates, in the anisotropic coefficient $g$, for the approximating solutions allow to show that such a sequence of solutions indeed converge to a solution of the RTE in the forward peaked regime, see Proposition \ref{l1EUS}.  Item (1) is proved in Proposition \ref{p1EUS} and item (2) is proved in Theorem \ref{t2EUS}.  Finally, in Section 5 a classical technique in elliptic and parabolic theory to obtain improved regularity by studying the level sets of the solution is used, an excellent reference for this topic is \cite{CafVas}.  Interestingly, such technique is borrowed in the present case to obtain a vanishing algebraic rate of the max-norm of the solution as described in item (4).  This result is proof of the diffusive nature of the scattering in the forward-peaked regime.

\section{Basic Properties of the Scattering Operator and Function Spaces}

In this section we show some basic properties of the scattering operator $\CalL$. These properties are fundamental to the analysis in this paper. They also motivate the function spaces that we will work within. 

\subsection{Stereographic projection and the representation of the projected scattering operator}
The results given in this work can be stated  transparently employing the stereographic projection $\mathcal{S}:\mathbb{S}^{d-1} \rightarrow \mathbb{R}^{d-1}$.  Using subscripts to denote the coordinates of a vector, we can write the stereographic projection as
\begin{equation*}
     \CalS(\theta)_{i} 
  = \frac{\theta_{i} }{ 1- \theta_{d} }\,,
\qquad 1 \leq i \leq d-1 \,. 
\end{equation*}
The stereographic projection is surjective and smooth (except in the north pole) with its inverse $\mathcal{J}:\mathbb{R}^{d-1}\rightarrow\mathbb{S}^{d-1}$ given by
\begin{equation*}
    \CalJ(v)_{i}  
 = \frac{ 2v_{i} }{\vint{v}^2}\,, \quad 1 \leq i \leq d-1\,, 
\quad\text{and} \quad    
     \CalJ(v)_{d}  = \frac{|v|^{2} - 1}{\langle v \rangle^{2}}\,,
\end{equation*}
where $\vint{v}:= \sqrt{1+ |v|^{2}}$.  The Jacobian of such transformations can be computed respectively as
\begin{equation*}
    \dv = \frac{\dtheta}{(1- \theta_{d})^{d-1}}\,, \;\;\text{and}\;\; 
    \dtheta = \frac{2^{d-1}\,\dv}{\vint{v}^{2(d-1)}}  \,.
\end{equation*}
Additionally, using the shorthanded notation $\theta = \CalJ(v)$ and $\theta' = \CalJ(v')$, one can show by simple algebra that
\begin{equation}\label{desp}
     1 - \theta \cdot \theta' 
 = 2\,\frac{\abs{v - v'}^{2}}{\vint{v}^2 \vint{v'}^2} \,.
\end{equation}
\begin{Proposition}\label{main:smooth}
Let $b_{s}$ be a scattering kernel satisfying \eqref{skA1} and \eqref{skA2} and write $\mathcal{I}_{b_s}=\mathcal{I}_{b(1)} + \mathcal{I}_{h}$.  Then, for any sufficiently regular function $u$ in the sphere the stereographic projection of the operator $\mathcal{I}_{b(1)}$ is given by
\begin{align}
     \frac{\big[\CalL_{b(1)}(u)\big]_\CalJ}{\langle \cdot \rangle^{d-1}} 
  &= \frac{2^{\frac{d-1}{2}-s}\,b(1)}{c_{d-1,s}}\langle v \rangle^{2s}
     \Big(-(-\Delta_{v})^{s}w_{\mathcal{J}} 
             + \UJ\,(-\Delta_{v})^{s}\frac{1}{\vint{\cdot}^{d-1 - 2s}} \Big)\nonumber\\
    &=\frac{2^{\frac{d-1}{2}-s}\,b(1)}{c_{d-1,s}}\langle v \rangle^{2s}
     \Big(-(-\Delta_{v})^{s}w_{\mathcal{J}} 
             + c_{d,s}\,\frac{\UJ}{\langle v \rangle^{d-1 + 2s}} \Big)\,,\label{main:smooth1}     
\end{align}
where $\UJ=u\circ\mathcal{J}$ (the projected function) and $w_{\mathcal{J}} := \frac{ u_{\mathcal{J}}}{ \langle \cdot \rangle^{d-1 - 2s} }$.  In particular, one has the formula
\begin{equation}\label{main:smooth2}
\frac{1}{b(1)}\int \mathcal{I}_{b(1)}(u)(\theta)\,\overline{u(\theta)}\,\dtheta = - \,c_{d,s}\,\big\| (-\Delta_{v})^{s/2}w_{\mathcal{J}}\big\|^{2}_{L^{2}(\mathbb{R}^{d-1})} + C_{d,s}\,\| u \|^{2}_{L^{2}(\mathbb{S}^{d-1})}\,,
\end{equation}
for some explicit positive constants $c_{d,s}$ and $C_{d,s}$ depending on $s$ and $d$.  Furthermore, defining the differential operator $(-\Delta_{\theta})^{s}$ acting on functions defined on the sphere by the formula
\begin{equation}\label{smooth8}
\big[(-\Delta_{\theta})^{s}u\big]_{\mathcal{J}} := \langle \cdot \rangle^{d-1 + 2s}\,(-\Delta_{v})^{s} w_{ \mathcal{J} }\,,
\end{equation}
the scattering operator simply writes as the sum of a singular and a $L^{2}_{\theta}$-bounded parts
\begin{equation}\label{smooth9}
\mathcal{I}_{b_s}= -\,D\,(-\Delta_{\theta})^{s} + c_{s,d}\,\textbf{1} + \mathcal{I}_{h}\,,
\end{equation}
where $D=2^{\frac{d-1}{2} - s}\frac{b(1)}{c_{d-1,s}}$ is the diffusion parameter. 
\end{Proposition}
\begin{proof}
Given the decomposition of the scattering kernel $b_{s}$ assumed in \eqref{skA2} one certainly can write the scattering operator as $\mathcal{I} = \mathcal{I}_{b(1)} + \mathcal{I}_{h}$.  The operator $\mathcal{I}_{h}$ is a bounded operator in $L^{2}(\mathbb{S}^{d-1})$.  Indeed, assumption~\eqref{skA2} implies that
\begin{equation*}
   \theta' 
\rightarrow 
  h(\theta\cdot\theta')=\frac{\tilde{b}(\theta \cdot \theta')}{(1-\theta \cdot \theta')^{\frac{d-1}{2}+s}}
   \in L^{1}(\Ss^{d-1}) \,.
\end{equation*}
Then, using Cauchy-Schwarz inequality it follows that
\begin{equation}\label{smooth0.5}
\big\| \mathcal{I}_{h}(u) \big\|_{L^{2}(\mathbb{S}^{d-1})} \leq 2\,\big\|h\big\|_{L^{1}(\mathbb{S}^{d-1})}\|u\|_{L^{2}(\mathbb{S}^{d-1})}\,.
\end{equation}
The details can be found in the appendix, Lemma \ref{app:coercive}.  Let us concentrate on the leading term $\mathcal{I}_{b(1)}$ using the Stereographic projection and \eqref{desp}
\begin{align}\label{smooth1}
\big[\mathcal{I}_{b(1)}(u) \big]_{\mathcal{J}}(v)&=2^{\frac{d-1}{2}-s}\,b(1)\,\langle v \rangle^{d-1+2s}\int_{\mathbb{R}^{d-1}}\frac{\UJ(v') - \UJ(v)}{|v-v'|^{d-1+2s}} \frac{\dv'}{\langle v' \rangle^{d-1-2s}}\nonumber\\
&=2^{\frac{d-1}{2}-s}\,b(1)\,\langle v \rangle^{d-1+2s}\Bigg(\int_{\mathbb{R}^{d-1}}\frac{\WJ(v') - \WJ(v)}{|v-v'|^{d-1+2s}} \dv'\nonumber\\
&\hspace{3cm}+\UJ(v)\int_{\mathbb{R}^{d-1} } \frac{ \frac{1}{\langle v \rangle^{d-1-2s}} - \frac{1}{\langle v' \rangle^{d-1-2s} }}{|v-v'|^{d-1+2s}} \dv' \Bigg)\nonumber\\
&= \frac{2^{\frac{d-1}{2}-s}\,b(1)}{c_{d-1,s}}\,\langle v \rangle^{d-1+2s}\Big( -(-\Delta_{v})^{s}\WJ + \UJ(-\Delta_{v})^{s}\frac{1}{\langle \cdot \rangle^{d-1-2s}} \Big)\nonumber\\
&=\frac{2^{\frac{d-1}{2}-s}\,b(1)}{c_{d-1,s}}\langle v \rangle^{2s}
     \Big(-(-\Delta_{v})^{s}w_{\mathcal{J}} 
             + c_{d,s}\,\frac{\UJ}{\langle v \rangle^{d-1 + 2s}} \Big)\,.
\end{align}
For the last inequality we have used Lemma \ref{app:convex} on Bessel potentials in the appendix to find  that 
\begin{equation*}
(-\Delta_{v})^{s}\frac{1}{\langle\cdot\rangle^{d-1 - 2s}}(v) = \frac{c_{d,s}}{\langle v \rangle^{d-1 + 2s}}\,.
\end{equation*} 
This proves \eqref{main:smooth1} and as a direct consequence,
\begin{align}\label{smooth3}
\int_{\mathbb{S}^{d-1}}\mathcal{I}_{b(1)}&(u)(\theta)\, \overline{u(\theta)}\,\dtheta = 2^{d-1}\int_{\mathbb{S}^{d-1}}\Big[\mathcal{I}_{b(1)}(u)\Big]_{\mathcal{J}}(v)\, \overline{\UJ(v)}\,\frac{\dv}{\langle v \rangle^{2(d-1)}}\nonumber\\
& = 2^{\frac{3(d-1)}{2}-s}\,\frac{b(1)}{ c_{d-1,s} }\,\Big[ - \big\|(-\Delta)^{s/2 }w_{\mathcal{J}} \big\|^{2}_{L^{2}(\mathbb{R}^{d-1})} + \frac{c_{d,s}}{2^{d-1}}\| u \|^{2}_{L^{2}(\mathbb{S}^{d-1})} \Big]\,.
\end{align}
This completes the proof.
\end{proof}
\subsection{Functional spaces, mass conservation and main energy estimate}  Due to Proposition \ref{main:smooth}, it is convenient to introduce the Hilbert space $H^{s}(\mathbb{S}^{d-1})$ (or simply $H^{s}_{\theta}$) defined as
\begin{equation}\label{angleSpace}
H^{s}_{\theta} := \big\{u\in L^{p_s}_{\theta}\;:\,(-\Delta)^{s/2}w_{\mathcal{J}}\in L^{2}_{v}\,\big\}\,,\quad s\in(0,1)\,,
\end{equation}
where $\tfrac{1}{p_s} = \tfrac{1}{2} - \tfrac{s}{d-1}$ and endowed with inner product 
\begin{align}\label{SobSphere}
\langle u,f \rangle_{H^{s}(\mathbb{S}^{d-1})}:=\big\langle (-\Delta)^{s/2}w_{\mathcal{J}},&(-\Delta)^{s/2}g_{\mathcal{J}}\big\rangle_{L^{2}(\mathbb{R}^{d-1})}\,,\nonumber\\
\text{where }\,\, w_{\mathcal{J}} &= \frac{ u_{\mathcal{J}}}{ \langle \cdot \rangle^{d-1- 2s} }\;\;\text{and}\;\; g_{\mathcal{J}} = \frac{ f_{\mathcal{J}}}{ \langle \cdot \rangle^{d-1- 2s} }\,,
\end{align}
as the working space in the angular variable.  That \eqref{SobSphere} is an inner product follows from Hardy-Littlewood-Sobolev (HLS) inequality
\begin{align}\label{smooth10}
\langle u,u &\rangle_{H^{s}(\mathbb{S}^{d-1})}=\int_{\mathbb{R}^{d-1}} \big|(-\Delta)^{s/2}w_{\mathcal{J}}(v)\big|^{2}\dv \nonumber
\\
&\geq C^{-2}_{H}\,\Big(\int_{\mathbb{R}^{d-1}}\big|w_{\mathcal{J}}(v)\big|^{p_s}dv\Big)^{\frac{2}{p_s}} = 2^{-(d-1)\frac{2}{p_s}}\,C^{-2}_{H}\,\Big(\int_{\mathbb{S}^{d-1}}\big|u(\theta)\big|^{p_s}\,\dtheta\Big)^{\frac{2}{p_s}}\,,
\end{align}
where $C_{H}$ is the HLS constant.  Thus, $\langle u,u \rangle_{H^{s}(\mathbb{S}^{d-1})}=0$ if and only if $u\equiv0$.  The condition $u\in L^{p_s}(\mathbb{S}^{d-1})$ is imposed to prevent constants as valid choice for $w_{\mathcal{J}}$ (which may happen for example when $u\in L^{2}_{\theta}$ only).  Now, observe the following useful representation of the inner product norm in $H^{s}(\mathbb{S}^{d-1})$ which follows directly from \eqref{smooth8}, \eqref{smooth9} and the weak representation \eqref{wf} 
\begin{align}\label{smooth10.1}
D_0\,\|u\|^{2}_{H^{s}(\mathbb{S}^{d-1})}:&=D_0\,\langle u,u \rangle_{H^{s}(\mathbb{S}^{d-1})} =D\int_{\mathbb{S}^{d-1}}(-\Delta_{\theta})^{s}u(\theta)\,\overline{u(\theta)}\,\dtheta \nonumber \\
&\hspace{.5 cm}= -\int_{\mathbb{S}^{d-1}}\mathcal{I}_{b(1)}(u)(\theta)\,\overline{u(\theta)}\,\dtheta + c_{d,s}\,b(1)\,\int_{\mathbb{S}^{d-1}}|u(\theta)|^{2}\dtheta  \\
&\hspace{-1cm}= 2^{\frac{d-1}{2} + s}\,b(1)\,\int_{\mathbb{S}^{d-1}}\int_{\mathbb{S}^{d-1}}\frac{(u(\theta') - u(\theta))^{2}}{|\theta' - \theta|^{d-1+ 2s}}\dtheta'\dtheta + c_{d,s}\,b(1)\int_{\mathbb{S}^{d-1}}|u(\theta)|^{2} \dtheta\,,\nonumber
\end{align}
where $D_0=2^{d-1}D$.  In the last expression one simply uses the equality $2(1- \theta\cdot\theta') = |\theta' - \theta|^{2}$ valid for any two unitary vectors.  Equations \eqref{smooth10.1} shown the equivalence of norms in $H^{s}(\mathbb{S}^{d-1})$
\begin{equation}\label{equivnorm}
\|u\|^{2}_{H^{s}(\mathbb{S}^{d-1})} \sim \int_{\mathbb{S}^{d-1}}\int_{\mathbb{S}^{d-1}}\frac{(u(\theta') - u(\theta))^{2}}{|\theta' - \theta|^{d-1+ 2s}}\dtheta'\dtheta + \|u\|^{2}_{L^{2}(\mathbb{S}^{d-1})}\,.
\end{equation}
which can be quite useful in computation\footnote{Expression \eqref{equivnorm} proves that $\langle\cdot,\cdot\rangle_{H^{s}_{\theta}}$ is an inner product when $u\in L^{p_s}_{\theta}$ is relaxed to just $u\in L^{2}_{\theta}$.  The HLS inequality, however, does not hold in general under this relaxed assumption.}.  Note that for functions $u\in L^{p_s}_{x,\theta}$ the Sobolev inequality \eqref{smooth10} is also valid since $u(x,\cdot)\in L^{p_s}_{\theta}$ for a.e $x$.  Therefore, the inequality
\begin{equation}\label{SE}
\int_{\mathbb{R}^{d}}\|u\|^{2}_{H^{s}_{\theta}}\dx\geq \int_{\mathbb{R}^{d}}\|u\|^{2}_{L^{p_s}_\theta}\dx
\end{equation}
is valid in the space
\begin{equation} \label{def:space}
H^{s}_{x,\theta}=\big\{u\in L^{p_s}_{x,\theta}\,: (-\Delta_{v})^{s/2}w_{\mathcal{J}}\in L^{2}_{x,v}\big\}\,,\quad s\in(0,1)\,.
\end{equation}
Finally, a direct integration of the radiative transport equation shows that solutions conserve mass
\begin{equation}\label{ConMass}
\int_{\mathbb{R}^{d}}\int_{\mathbb{S}^{d-1}}u(t,x,\theta)\dtheta\dx = \int_{\mathbb{R}^{d}}\int_{\mathbb{S}^{d-1}}u_0(x,\theta)\dtheta\dx\,,\quad t\geq0\,.
\end{equation} 
They also satisfy the energy estimate
\begin{align}\label{MEE1}
\tfrac{1}{2}\int_{\mathbb{R}^{d}}\int_{\mathbb{S}^{d-1}} \big|u(t,x,\theta)\big|^{2} &\dtheta\dx - \int^{t}_{t'}\int_{\mathbb{R}^{d}}\int_{\mathbb{S}^{d-1}}\mathcal{I}_{b_s}(u)\,u\, \dtheta\dx\dtau \nonumber\\
&= \tfrac{1}{2}\int_{\mathbb{R}^{d}}\int_{\mathbb{S}^{d-1}} \big|u(t',x,\theta)\big|^{2} \dtheta\dx\,,\quad \text{for any}\;\; 0 < t'\leq t<T\,.
\end{align}
Thus, using the equivalence of norms \eqref{smooth10.1} in \eqref{MEE1} one gets
\begin{align}\label{MEE2}
  \tfrac{1}{2} 
   \int_{\mathbb{R}^{d}}\int_{\mathbb{S}^{d-1}} 
     &\big| u(t,x,\theta) \big|^{2} \dtheta\dx 
 + D_0 \int^{t}_{t'}\int_{\mathbb{R}^{d}} 
                 \|u\|^{2}_{H^{s}_\theta} \dx\dtau \nonumber
\\
  &\leq 
       \tfrac{1}{2}
       \int_{\mathbb{R}^{d}}\int_{\mathbb{S}^{d-1}} 
            \big|u(t',x,\theta)\big|^{2} \dtheta\dx 
      + D_1 \int^{t}_{t'}\int_{\mathbb{R}^{d}}\|u\|^{2}_{L^{2}(\mathbb{S}^{d-1})} \dx\dtau\,,
\end{align}
valid for any $0< t'\leq t< T$.  Here $D_0$ depends on $d,s$ and $b(1)$ while $D_1$ depends on $d,s, b(1)$ and the integrable kernel $h$.  Energy estimate \eqref{MEE2} is central and will be used extensively along the proof.

\section{Technical Lemmas: Velocity Averaging Lemmas} 
In this section two versions of the regularization mechanism in the RTE equation are shown: weak and strong versions.  The weak version is the classical velocity averaging lemma where the average of the distribution function $u$ in the angular variable $\theta$ has improved regularity in the spatial variable $x$. The strong version is related to the fact that the actual density function $u$ will enjoy higher regularity in both angular and spatial variables.  Both proofs are quite related and follow the classical framework developed in~\cite{FB} adapted to the fact that $\theta$ lies in the sphere. The reader will note that the result about $L^{1}$ to $L^{2}$ improvement of Section~\ref{section:regularity-L-2} will only need the weak version.  Let us show first the regularity for the averaged density 
\begin{align}\label{def:rho}
   \rho(t,x) = \int_{\Ss^{d-1}} u(t, x, \theta) \dtheta 
          = \int_{\R^{d-1}} \UJ(t, x, v) \Jac(v) \dv \,,
\end{align}
where $\Jac(v) = \frac{2^{d-1}}{\La v \Ra^{2(d-1)}}$ is the Jacobian. 

\subsection{Averaging Lemma} First we show the weak regularization, that is, the averaging lemma for the solution. Throughout this subsection, we use $c_{0, d}$ to denote a constant that only depends on $d$. We also use $c_{0, d,s}$ to denote a constant that only on depends on $d,s$ and $c_{d,s,\delta}$ for any constant that only depends on $d,s, \delta$ where $\delta$ is defined in~
\eqref{def:delta}. These constants may change from line to line.

The main result is
\begin{Proposition}\label{Lemma:VA-Weak}
Suppose 
\begin{align*}
    g_1\in L^2([0, T] \times \R^{d} \times \R^{d-1}) \,,
\qquad
    g_2 \in L^2((0, T) \times \R^d; H^s(\R^{d-1})) 
\end{align*} 
for $s \in (0, 1)$ and $d \geq 3$.
Suppose $\UJ$ 
is a strong solution to the transport equation
\begin{align}\label{eq:transport-orig}
     \del_t \UJ + \theta(v) \cdot \Grad \UJ = g_1 + \La v \Ra^{(d-1)+2s} (-\Delta_v)^{s/2} g \,, 
\qquad
     \UJ \big|_{t=0} = \UJ^{o}(x, v) \,.
\end{align}
Then $\rho$ defined in~\eqref{def:rho} satisfies
\begin{align*}
    (-\Delta_x)^{\beta} \rho \in L^2([0, T] \times \R^{d}) \,.
\end{align*}
for $\beta > 0$ defined in~\eqref{def:regularity-exponent}. Moreover, there exists a constant $c_{d,s,\delta} > 0$ such that
\begin{align} \label{ineq:VA-Weak}
     \|(-\Delta_x)^{\beta} \rho\|_{L^2_{t, x}}^2
\leq
    c_{d,s,\delta} \left(\|\UJ^{o}\|_{L^2_{x, v}}^2
                  + \|\UJ\|_{L^2_{t,x,v}}^2
                  + \|g_1\|_{L^2_{t, x, v}}^2
                  + \left\|g_2 \right\|_{L^2_{t, x, v}}^2\right) \,.
\end{align}
where 
$\delta$ is defined in~\eqref{def:delta}.
\end{Proposition}

\begin{proof}
The proof is an adaption of the method in~\cite{FD}. We will focus on the second term containing $g_2$ since the part corresponding to $g_1$ follows directly from~\cite{FD}. 
Therefore we will check the regularity for $\rho$ where $u$ is a solution to 
\begin{equation}
\begin{aligned}\label{eq:transport-g-2}
     \del_t \UJ + \theta(v) \cdot \Grad \UJ &= \La v \Ra^{(d-1)+2s} (-\Delta_v)^{s/2} g_2 \,, 
\\
     \UJ \big|_{t=0} = \UJ^{o}(x, v)  \,.
\end{aligned}
\end{equation}
Let $\lambda$ be a constant (in $v$) which is to be determined. Rewrite~\eqref{eq:transport-g-2} as 
\begin{equation}
\begin{aligned}\label{eq:transport-g-3}
     \del_t \UJ  + \lambda \UJ + \theta(v) \cdot \Grad \UJ
    & \, = \lambda \UJ + \La v \Ra^{(d-1)+2s} (-\Delta_v)^{s/2} g_2 \,, 
\\
     \UJ \big|_{t=0} & \, = \UJ^{o}(x, v) \,.
\end{aligned}
\end{equation}
Let $\hat\rho(t, \xi)$ be the Fourier transform of $\rho$ in $x$ and take the Fourier transform in $x$ of~\eqref{eq:transport-g-3}. We can then directly solve for $\hat\rho$ and obtain
\begin{equation} \label{def:rho-hat}
\begin{aligned}
    \hat\rho(t, \xi) 
  =\, & e^{-\lambda t} \!
     \int_{\R^{d-1}} \!\! e^{- i \theta \cdot \xi t} \,\, \hat u_{\CalJ}^{o} J(v)\dv
       + \lambda \!
             \int_0^t \!\! \int_{\R^{d-1}} \!\!
                   e^{- (\lambda + i \theta \cdot \xi) \tau}  \,
                   \hat u_{\CalJ}(t - \tau, \xi, v) J(v) \dv \dtau
\\[3pt]
   & + \int_0^t \int_{\R^{d-1}}   
               e^{- (\lambda + i \theta \cdot \xi) \tau}  \,
               J(v) \La v \Ra^{(d-1)+2s} \,    
               (-\Delta_v)^{s/2} \hat g_2(t-\tau, \xi, v) \dv\dtau \,.         
\\[3pt]
   & \hspace{-0.5cm} \Denote I_1 + I_2 + I_3 \,.
\end{aligned}
\end{equation}
We estimate $I_1, I_2, I_3$ respectively. First, note that 
$\sqrt{J} \in L^2(\dv)$ 
\begin{align*}
   \int_{z < \theta(v) \cdot e < z + \epsilon} J(v) \dv
   = \int_{z < \theta \cdot e < z + \epsilon} 1 \dtheta 
   < c_{0,d} \, \epsilon \,,
\end{align*}
for any $z \in \R$, $e \in \Ss^{d-1}$, and $\epsilon > 0$. 
Therefore, the first two terms $I_1, I_2$ are estimated in exactly the same way as in~\cite{FD} which gives
\begin{align} \label{bound:I-1-I-2}
& \quad \,
    \int_0^T \left(|I_1|^2(t, \xi) + |I_2|^2(t, \xi) \right) \dt \nn
\\
&\leq
     c_{0,d} \left(\frac{1}{|\xi|} \int_{\R^{d-1}} |\hat u_{\CalJ}^{o}(\xi, v)|^2 J\dv
                   + \frac{\lambda}{|\xi|} \int_0^T \int_{\R^{d-1}} |\hat u_{\CalJ}(t, \xi, v)|^2 J \dv\dt          
             \right) \,,
\end{align}
In order to estimate $I_3$, we denote 
\begin{align*}
    \tilde g = (-\Delta_v)^{-\frac{1-s}{2}} g_2 \,. 
\end{align*}
Since $g_2(t, \cdot, \cdot) \in L^2_{x, \theta}$, we have  
\begin{align*}
    \tilde g(t,x, \cdot) \in 
        L^{p_1}(\R^{d-1}) \,,  
\qquad p_1 = \frac{2(d-1)}{(d-1) - 2(1-s)} > 2 \,,
\end{align*}
The forcing term in terms of $\tilde g$ has the form
\begin{align*}
      (-\Delta_v)^{s/2} g_2 = - \nabla_v \cdot (\mathcal{R} \tilde g) \,, 
\end{align*}
where $\mathcal{R} = (\mathcal{R}_1, \mathcal{R}_2, \cdots, \mathcal{R}_{d-1})$ is the Riesz transform in $\R^{d-1}$. Note that $\mathcal{R} \tilde g \in L^{p_1}(\R^{d-1})$.
The third term $I_3$ then becomes
\begin{align*}
    I_3 &= 2^{d-1} \int_0^t \int_{\R^{d-1}} e^{-\lambda \tau}    
                \left(e^{- i \theta \cdot \xi \tau}  \,
               \La v \Ra^{-(d-1)+2s} \right) \,    
                \nabla_v \cdot  \widehat{\mathcal{R} \tilde g}(t-\tau, \xi, v) \dv\dtau
\\
    & = 2^{d-1} (i \, \tau \, \xi) \cdot  \int_0^t \int_{\R^{d-1}} e^{-\lambda \tau}    
                e^{- i \theta \cdot \xi \tau}  \,
                \left(\La v \Ra^{-(d-1)+2s} \nabla_v \theta(v) \right)  \,  \cdot  
                 \widehat{\mathcal{R} \tilde g}(t-\tau, \xi, v) \dv\dtau
\\ & \quad \,
            - 2^{d-1} \int_0^t \int_{\R^{d-1}} e^{-\lambda \tau}    
                e^{- i \theta \cdot \xi \tau}  \,
                \nabla_v \left(\La v \Ra^{-(d-1)+2s} \right) \,  \cdot
                 \widehat{\mathcal{R} \tilde g}(t-\tau, \xi, v) \dv\dtau
\\
  & \Denote I_{31} - I_{32} \,.                               
\end{align*}
We will show the estimates for $I_{32}$ in details. The other term $I_{31}$ will be bounded in a similar way. For the ease of notation, let 
\begin{align*}
     \psi(v) = \La v \Ra^{-(d-1)+2s} \,.
\end{align*}
Note that 
\begin{align} \label{est:Grad-psi}
    \nabla_v \psi(v) \leq c_{0,d,s} \La v \Ra^{-d + 2s} \,.
\end{align}
The estimates for $I_{32}$ are as follows. 
\begin{equation}
\begin{aligned} \label{ineq:I-32}
& 
   \int_0^T \!\! \left|I_{32}(t, \xi) \right|^2 \!\! \dt 
 = 2^{2(d-1)} \!\! \int_0^T  \!\!\!
             \left(\int_0^t \!\! \int_{\R^{d-1}}  \!\!\! e^{-\lambda \tau} 
                e^{- i \theta \cdot \xi \tau}  \,
                \nabla_v \psi(v)   \cdot
                 \widehat{\mathcal{R} \tilde g}(t-\tau, \xi, v) \dv\dtau\right)^2 \!\!\! \dt
\\
& \leq \frac{2^{2(d-1)}}{\lambda} \int_0^T 
             \int_0^t e^{-\lambda \tau} 
                  \left|\int_{\R^{d-1}}     
                e^{- i b(v) (|\xi| \tau)}  \,
                \nabla_v \psi(v) \,  \cdot
                 \widehat{\mathcal{R} \tilde g}(t - \tau, \xi, v) \dv \right|^2\dtau \dt
\\
& \leq \frac{2^{2(d-1)}}{\lambda} \int_0^T 
             \int_0^\infty e^{-\lambda \tau} 
                  \left|\int_{\R^{d-1}}     
                e^{- i b(v) (|\xi| \tau)}  \,
                \nabla_v \psi(v) \,  \cdot
                 \widehat{\mathcal{R} \tilde g}(t, \xi, v) \dv \right|^2 \dtau \dt
\\
& = \frac{2^{2(d-1)}}{\lambda |\xi|}\int_0^T 
             \left(\int_0^\infty e^{-\frac{\lambda}{|\xi|} \tau} 
                  \left|\int_{\R^{d-1}}     
                e^{- i b(v) \tau}  \,
                \nabla_v \psi(v) \,  \cdot
                 \widehat{\mathcal{R} \tilde g}(t, \xi, v) \dv \right|^2\dtau\right) \dt \,,
\end{aligned}
\end{equation}
where $b(v) = \theta(v) \cdot \xi/|\xi|$. 
Following~\cite{FD}, we let $\zeta(y) = \One_{y>0} \, e^{-y}$
such that $\Disp\hat\zeta(z) = \frac{1}{1-iz}$.
Define
\begin{align*}
     \phi(y) 
     = \int_{\R^{d-1}} 
           \frac{1}{\gamma} \, \zeta\left(\frac{b(v) - y}{\gamma}\right) 
           \nabla_v \psi(v) \cdot \, \widehat{\mathcal{R} \tilde g}(t, \xi, v) \dv \,.
\end{align*}
Denote $\CalF_v$ as the Fourier transform in $v$. Then
\begin{align*}
    \left| \CalF_v (\phi) \left(z \right)\right| 
    = \frac{1}{\sqrt{1+ \gamma^2 z^2}} 
         \left|\int_{\R^{d-1}} 
                    e^{-i b(v) z} 
                    \nabla_v \phi(v) \cdot \widehat{\mathcal{R} \tilde g}(t, \xi, v) \dv \right| \,.
\end{align*}
Hence by Plancherel's theorem, the integrand in the last term of~\eqref{ineq:I-32} satisfies 
\begin{align*}
& \quad \,
   \int_0^\infty e^{-\frac{\lambda}{|\xi|} \tau} 
                  \left|\int_{\R^{d-1}}     
                e^{- i b(v) \tau}  \,
                \nabla_v \psi(v) \,  \cdot
                 \widehat{\mathcal{R} \tilde g}(t, \xi, v) \dv \right|^2\dtau
\\
& \leq 
   \int_0^\infty 
         \left|\frac{1}{\sqrt{1 + \frac{\lambda^2}{|\xi|^2} \tau^2}}\int_{\R^{d-1}}     
                e^{- i b(v) \tau}  \,
                \nabla_v \psi(v) \,  \cdot
                 \widehat{\mathcal{R} \tilde g}(t, \xi, v) \dv \right|^2\dtau
\\
& \leq 
   \int_{\R} 
     \left|\int_{\R^{d-1}} 
      \frac{1}{\gamma} \, \zeta\left(\frac{b(v) - y}{\gamma}\right) 
        \nabla_v \psi(v) \cdot \, \widehat{\mathcal{R} \tilde g}(t, \xi, v) \dv\right|^2 \dy \,,
\end{align*}
where $\gamma = \frac{\lambda}{|\xi|}$. Using H\"{o}lder's inequality, we have
\begin{equation}
    \label{bound:I-3-2-1}
\begin{aligned}
& \quad \,
   \left| \int_{\R^{d-1}} 
      \frac{1}{\gamma} \, \zeta\left(\frac{b(v) - y}{\gamma}\right) 
        \nabla_v \psi(v) \cdot \, \widehat{\mathcal{R} \tilde g}(t, \xi, v) \dv\right|^2
\\
&\leq
     \frac{1}{\gamma^2}
       \left(\!\int_{\R^{d-1}} \!\!
                  \zeta\left(\frac{b(v) - y}{\gamma}\right)
                  |\nabla_v \psi|^{2\alpha} \dv \! \right) \!
      \left(\!\int_{\R^{d-1}} \!\!
               \zeta^{q/2}\left(\frac{b(v) - y}{\gamma}\right)
               |\nabla_v \psi|^{q(1-\alpha)} \dv \right)^{\frac{2}{q}}
\\[3pt]
& \quad \, \times
      \|\widehat{\mathcal{R} \tilde g}(t, \xi, \cdot)\|_{L^{p_1}(\dv)}^2 \,,
\end{aligned}
\end{equation}
where $\alpha, q$ are chosen such that
\begin{align*}
     q = \frac{d-1}{1-s} \,, 
 \qquad 
    \frac{d-1}{2(d-2s)} < \alpha < 1 \,.
\end{align*}
Note that for $d \geq 3$ and $s \in (0,1)$, we indeed have $\Disp \frac{d-1}{2(d-2s)} < 1$. In this case, 
\begin{align} \label{bound:psi-1}
     \int_{\R^{d-1}} |\nabla_v \psi|^{2\alpha} \dv 
\leq 
    c_{0,d,s} \int_{\R^{d-1}} 
                \La v \Ra^{-2\alpha (d-2s)} \dv
< \infty \,,
\end{align}
Next, 
\begin{align*}
   q(1-\alpha) 
< \frac{d-1}{1-s} \left(1 - \frac{d-1}{2(d-2s)}\right)
= \left(\frac{d-1}{1-s}\right) \left( \frac{d+1-4s}{2(d-2s)} \right) \,.
\end{align*}
Take $\alpha$ close to $\frac{d-1}{2(d-2s)}$ such that 
\begin{align*}
    q(1-\alpha) 
=  \left(\frac{d-1}{1-s}\right) \left( \frac{d+1-4s}{2(d-2s)} \right) 
    - \frac{2\delta}{d-2s} \,,
\end{align*}
where $\delta$ is close to zero which is to be determined. Note that 
\begin{align*}
     q(1-\alpha)(d-2s)
& = 2(d-1) + \left(\left(\frac{d-1}{1-s}\right) \left( \frac{d+1-4s}{2} \right)
                           - 2(d-1)\right) - \delta
\\
& = 2(d-1) + \frac{(d-3)(d-1)}{2(1-s)} - 2\delta                           
\end{align*}
In particular, if $d=3$, then
\begin{align*}
    q(1-\alpha)(d-2s) = 2(d-1) - 2\delta \,.
\end{align*}
Then,
\begin{align*}
& \quad \,
    \int_{z < \theta(v) \cdot e < z + \epsilon}
        |\nabla_v \psi|^{q(1-\alpha)} \dv
\leq
     {c_{0,d,s}} \int_{z < \theta(v) \cdot e < z + \epsilon}
               \La v \Ra^{-2(d-1) + 2\delta} \dv
\\
& \leq 
      c_{0,,d,s} \int_{z < \theta \cdot e < z + \epsilon}
           \frac{1}{(1-\theta_3)^\delta} \dtheta
\leq 
      c_{d,s,\delta} \, \epsilon^{1-\delta} \,,
\end{align*}
for any $\delta \in (0, 1)$, $z \in \R$, and $e \in \Ss^{d-1}$. 
If $d \geq 4$, then we can have
\begin{align*}
    q(1-\alpha)(d-2s) > 2(d-1)
\end{align*}
by choosing $\alpha$ close enough to $\frac{d-1}{2(d-2s)}$. In this case, $\delta$ will be chosen as zero.
Therefore, 
\begin{align*}
& \quad \,
    \int_{z < \theta(v) \cdot e < z + \epsilon}
        |\nabla_v \psi|^{q(1-\alpha)} \dv
\leq
     {c_{0,d,s}} \int_{w < \theta(v) \cdot e < w + \epsilon}
               \La v \Ra^{-q(1-\alpha)} \dv
\leq c_{0,d,s} \, \epsilon \,,
\end{align*}
for any $z \in \R$, $e \in \Ss^{d-1}$. 
By the proof of Lemma 2.4 and Remark 2.5 in~\cite{FD}, we have
\begin{align} \label{bound:psi-2}
      \left(\int_{\R^{d-1}}
               \zeta^{q/2}\left(\frac{b(v) - y}{\gamma}\right)
               |\nabla_v \psi|^{q(1-\alpha)} \dv \right)^{2/q}
\leq 
    c_{d,s,\delta} \, \gamma^{2(1-\delta)/q} \,,
\end{align}
where
\begin{align} \label{def:delta}
   \text{$\delta \in (0, 1)$ arbitrary for $d=3$}, \qquad \text{$\delta = 0$ for $d \geq 4$} \,.
\end{align}
Combining~\eqref{bound:I-3-2-1}, \eqref{bound:psi-1}, and \eqref{bound:psi-2}, we have
\begin{align*}
& \quad \,
  \int_{\R} 
     \left|\int_{\R^{d-1}} 
      \frac{1}{\gamma} \, \zeta\left(\frac{b(v) - y}{\gamma}\right) 
        \nabla_v \psi(v) \cdot \, \widehat{\mathcal{R} \tilde g}(t, \xi, v) \dv\right|^2 \dy 
\\[3pt]
& \leq
     c_{d,s,\delta} \, \gamma^{-1 + 2(1-\delta)/q} \, 
        \|\widehat{\mathcal{R} \tilde g}(t, \xi, \cdot)\|_{L^{p_1}(\dv)}^2 \,.
\end{align*}
Therefore,
\begin{equation} 
    \label{bound:I-3-2}
\begin{aligned} 
    \int_0^T \left|I_{32}(t, \xi) \right|^2 \dt 
& \leq 
    \frac{2^{2(d-1)} c_{0,8} }{\lambda |\xi|} \left(\frac{\lambda}{|\xi|}\right)^{-1+2(1-\delta)/q}
       \int_0^T \|\widehat{\mathcal{R} \tilde g}(t, \xi, \cdot)\|_{L^{p_1}(\dv)}^2 \dt
\\
& \leq 
     c_{d,s,\delta} \frac{\lambda^{-2+2(1-\delta)/q}}{|\xi|^{2(1-\delta)/q}}
       \int_0^T |\hat g(t, \xi, v)|^2 \dv\dt \,.
\end{aligned}
\end{equation}
The estimate for $I_{31}$ is similar since 
\begin{align*}
    |\La v \Ra^{-(d-1)+2s}\nabla_v \theta|
\leq 
   c_{0, d} \La v \Ra^{-d+2s} \,,
\end{align*}
which is the same bound as $\nabla_v \psi$ in~\eqref{est:Grad-psi}. 
The only difference is that $I_{31}$ has an extra coefficient $i \tau \xi $, which gives an extra coefficient $\frac{|\xi|^2}{\lambda^2}$ in a similar step in~\eqref{ineq:I-32} when estimating $\int_0^T \abs{I_{31}(t, \xi)}^2 \dt$. Therefore, 
\begin{align} \label{bound:I-3-1}
   \int_0^T \left|I_{31}(t, \xi) \right|^2 \dt
\leq 
   c_{d,s,\delta} \frac{\lambda^{-4 + 2(1-\delta)/q}}{|\xi|^{-2+2(1-\delta)/q}}
          \int_0^T |\hat g(t, \xi, v)|^2 \dv\dt 
\end{align}
Combining~\eqref{bound:I-1-I-2}, \eqref{bound:I-3-2}, and~\eqref{bound:I-3-1}, we have
\begin{align*}
    \int_0^T |\hat\rho(t, \xi)|^2 \dt
&\leq 
   c_{d,s,\delta} \left(\frac{1}{|\xi|} \int_{\R^{d-1}} |\hat \UJ^{o}(\xi, v)|^2 J\dv
                 + \frac{\lambda}{|\xi|} \int_0^T \int_{\R^{d-1}} |\hat \UJ(t, \xi, v)|^2 \dv\dt    
                 \right)
\\
& \quad \,
   + c_{d,s,\delta} \left(\frac{\lambda^{-2+2(1-\delta)/q}}{|\xi|^{2(1-\delta)/q}}
                   + \frac{\lambda^{-4 + 2(1-\delta)/q}}{|\xi|^{-2+2(1-\delta)/q}}\right)
             \int_0^T |\hat g(t, \xi, v)|^2 \dv\dt \,.
\end{align*}
Choose 
\begin{align*}
     \lambda = |\xi|^{\frac{3-\beta_1}{5-\beta_1}} \,,
\qquad
     \beta_1 = \frac{2(1-\delta)}{q} \,.
\end{align*}
Since $\rho \in L^2((0, T) \times \R^d)$, we only need to integrate over $|\xi| > 1$ and obtain
\begin{align*}
    \|\rho\|_{L^2((0, T), H^\beta(\dx))}^2
\leq 
    c_{d,s,\delta} \left(\|\UJ^{o}\|_{L^2( J\dv\dx)}^2
                  + \|\UJ\|_{L^2(J\dt\dv\dx)}^2
                  + \|g_1\|_{L^2_{t, x, v}}^2
                  + \left\|g \right\|_{L^2_{t, x, v}}^2\right) \,,
\end{align*}
where recall that 
$\delta$ is defined in~\eqref{def:delta} and
\begin{align} \label{def:regularity-exponent}
    \beta = \frac{2}{5-\beta_1} = \frac{2q}{5q - (1-\delta)} \,,
\qquad
    q = \frac{d-1}{1-s} \,.
\end{align}
We thereby finish the proof of the regularization of $\rho$.
\end{proof}
\begin{rmk}
Note that although we assume $g_2 \in L^2((0, T) \times \R^d; H^s(\R^{d-1}))$ to make the proof of Proposition \ref{Lemma:VA-Weak} rigorous, the bound in~\eqref{ineq:VA-Weak} only depends on the $L^2$-norm of $g_2$. Hence a typical density argument can relax the assumption to $g_2 \in L^2((0, T) \times \R^d; L^2(\R^{d-1}))$.
\end{rmk}

\subsection{Strong Regularizing Lemma}
The objective of the following discussion is to prove a key estimate to obtain the regularizing effect in the spatial variable of a solution $u(t,x,\theta)$ satisfying the radiative transfer equation  in the highly peaked forward regime.
\begin{Theorem}\label{st-reg-lemma}
Fix any dimension $d\geq3$ and assume that $u\in \mathcal{C}\big([t_0,t_1),L^{2}(\mathbb{R}^{d}\times\mathbb{S}^{d-1})\big)$ solve the transport problem
\begin{equation}\label{sr:main}
\partial_{t}u + \theta\cdot\nabla_{x}u = \mathcal{I}(u)\,,\quad\quad t\in[t_0,t_{1})\,.
\end{equation}
Then for any $s\in(0,1)$, there exists a constant $C:=C(d,s)$ independent of time such that 
\begin{align}\label{sr:main:estimate}
\big\|(-\Delta_{x})^{\frac{s_0}{2}} u&\big\|_{L^{2}([t_0,t_1)\times\mathbb{R}^{d}\times\mathbb{S}^{d-1})}\leq C\,\Big( \big\| u(t_0)\big\|_{L^{2}(\mathbb{R}^{d}\times\mathbb{S}^{d-1})} + \big\|u \big\|_{L^{2}([t_0,t_1)\times\mathbb{R}^{d}\times\mathbb{S}^{d-1})}  \nonumber\\
&\hspace{1cm}+\big\|(-\Delta_{v})^{s/2}w_{\mathcal{J}} \big\|_{L^{2}([t_0,t_1)\times\mathbb{R}^{d}\times\mathbb{R}^{d-1})} \Big)\,,\quad\quad s_0=\frac{s/4}{2s+1}\,.
\end{align}
\end{Theorem}
\begin{proof}
We follow the method given in \cite{FB} and adapt it to the advection operator $\theta\cdot\nabla_{x}$.  We start with an approximation of the identity in the sphere $\{\rho_{\epsilon}\}_{\epsilon>0}$ defined through an smooth function $\rho \in \mathcal{C}(-1,1)$ satisfying the properties
\begin{equation}\label{sre1}
\int^{1}_{-1}\rho(z)\,z^{\frac{d-3}{2}}\,\dz=1\,,\quad \quad 0\leq\rho(z)\lesssim \frac{1}{z^{\frac{d-1}{2}+s}}\,.
\end{equation}
Introduce the quantity
\begin{equation}\label{sre2}
C_{\epsilon}=|\mathbb{S}^{d-2}|\int^{1}_{-1}\rho(z)\,z^{\frac{d-3}{2}}\big(2-\epsilon\,z\big)^{\frac{d-3}{2}}\,\dz\,,\quad \epsilon\in(0,1]\,,
\end{equation}
and note that $\inf_{\epsilon\in(0,1]}C_{\epsilon}>0$.  Thus, define the approximation of the identity as
\begin{equation}\label{sre3}
\rho_{\epsilon}(z) = \tfrac{1}{C_{\epsilon}\,\epsilon^{ \frac{d-1}{2} }}\,\rho\Big(\frac{z}{\epsilon}\Big)\,.
\end{equation}
It is not difficult to prove that
\begin{equation}\label{sre4}
\int_{\mathbb{S}^{d-1}}\rho_{\epsilon}\big(1-\theta\cdot\theta' \big)\,\dtheta' = 1\,,\quad \epsilon>0\,.
\end{equation}
In the sequel we understand the convolution in the sphere, for any real function $\psi$ defined on the sphere, as
\begin{equation}\label{sre5}
\big(\rho \star \psi\big) (\theta) = \int_{\mathbb{S}^{d-1}} \rho\big(1 - \theta\cdot\theta' \big)\,\psi(\theta')\,\dtheta'\,.
\end{equation}
Now, we wish to analyze $u$, a solution of \eqref{sr:main}, in the interval $[t_0,t_1)$ for any $0<t_0<t_1<\infty$.  To this end, multiply \eqref{sr:main} by $\text{1}_{[t_0,t_{1})}$ and take the Fourier transform in time and spatial variables to obtain
\begin{equation}\label{sre6.5}
i(w+\theta\cdot k)\,\hat{u}(w,k,\theta) = \mathcal{I}(\hat{u})(w,k,\theta) + \hat{u}(t_0,k,\theta)\,e^{-iwt_0}\,,
\end{equation}
were we have denoted $\hat{u}(w,k,\theta)$ the Fourier transform of $u(t,x,\theta)\,\text{1}_{[t_0,t_{1})}$ in the time and spatial variables.  Note that the boundary component at $t_{1}$ is disregarded by causality of the equation.  A key step in the proof is to decompose $\hat{u}$, for any fixed $(w,k)$, as
\begin{equation}\label{sre7}
\hat{u}(w,k,\theta) = \big(\rho_{\epsilon}\star \hat{u}\big)(w,k,\theta) + \big[ \hat{u}(w,k,\theta) - \big(\rho_{\epsilon}\star \hat{u}\big)(w,k,\theta) \big]\,,
\end{equation}
and observe that, thanks to \eqref{sre1} - \eqref{sre4} and Proposition 2.1, the error can be estimated in terms of the regularity in the variable $\theta$ as
\begin{align}
& \quad \,
\big\| \hat{u}(w,k,\cdot) - \big(\rho_{\epsilon} \star \hat{u}\big)(w,k,\cdot) \big\|^{2}_{ L^{2}(\mathbb{S}^{d-1}) } \nonumber
\\
&= \int_{\mathbb{S}^{d-1}}\Big| \int_{\mathbb{S}^{d-1}}\rho_{\epsilon}\big( 1 - \theta\cdot\theta'\big)\big(\hat{u}(w,k,\theta) - \hat{u}(w,k,\theta')\big)\,\dtheta' \Big|^{2}\,\dtheta \nonumber
\\
&\leq \int_{\mathbb{S}^{d-1}} \int_{\mathbb{S}^{d-1}}\rho_{\epsilon}\big( 1 - \theta\cdot\theta'\big)\big|\hat{u}(w,k,\theta) - \hat{u}(w,k,\theta')\big|^{2}\,\dtheta'\,\dtheta \label{sre8}
\\
&\lesssim \frac{\epsilon^{s}}{C_{\epsilon}}\int_{\mathbb{S}^{d-1}} \int_{\mathbb{S}^{d-1}}\frac{\big|\hat{u}(w,k,\theta) - \hat{u}(w,k,\theta')\big|^{2}}{( 1 - \theta\cdot\theta')^{\frac{d-1}{2} + s} }\,\dtheta'\,\dtheta \nonumber
\\
& = -\frac{2\,\epsilon^{s}}{C_{\epsilon}\,b(1)}\int_{\mathbb{S}^{d-1}} \mathcal{I}_{b(1)}(\hat{u})(w,k,\theta)\,\overline{\hat{u}(w,k,\theta)}\,\dtheta \lesssim \frac{\epsilon^{s}}{C_{\epsilon}} \big\|(-\Delta_{v})^{s/2}\widehat{w_{\mathcal{J}}}(w,k,\cdot) \big\|^{2}_{L^{2}(\mathbb{R}^{d-1})}\,. \nn
\end{align}
Let us estimate the term $\rho_{\epsilon}\star \hat{u}$, for each fixed $(w,k)$, which we compute from \eqref{sre6.5}
\begin{equation}\label{sre9}
\widehat{u} = \frac{ \lambda\, \widehat{u} + \mathcal{I}(\hat{u}) + \hat{u}(t_0,k,\theta)\,e^{-iwt_0}}{\lambda + i(w + \theta\cdot k)}\,,
\end{equation}
where $\lambda>0$ is an interpolation parameter depending only on $|k|$ (the parameter $\epsilon$ will depend only on $|k|$ as well).  Formulas \eqref{smooth9} and \eqref{sre9} in turn lead to
\begin{align}
\big(\rho_{\epsilon}\star \widehat{u}\big)(w,k,\theta) &= \int_{\mathbb{S}^{d-1}}\rho_{\epsilon}\big(1 - \theta\cdot\theta' \big)\,\frac{\lambda\,\widehat{u}(w,k,\theta') + \mathcal{I}(\hat{u})(w,k,\theta') + \hat{u}(t_0,k,\theta')\,e^{-iwt_0}}{\lambda + i(w + k\cdot\theta')}\,\dtheta' \nonumber
\\
& =\int_{\mathbb{S}^{d-1}}\rho_{\epsilon}\big(1 - \theta\cdot\theta' \big)\,\frac{\widehat{u}(w,k,\theta') + \frac{1}{\lambda}\mathcal{K}(\hat{u})(w,k,\theta') }{ 1 + i(w + k\cdot\theta')/\lambda}\,\dtheta'   \nonumber
\\
& \hspace{.5cm} -\frac{D}{\lambda}\int_{\mathbb{S}^{d-1}}\rho_{\epsilon}\big(1 - \theta\cdot\theta' \big)\,\frac{(-\Delta_{\theta'})^{s}\hat{u}(w,k,\theta') }{ 1 + i(w + k\cdot\theta')/\lambda}\,\dtheta'  \label{sre10}
\\
& \hspace{.5cm} + \frac{1}{\lambda}\int_{\mathbb{S}^{d-1}}\rho_{\epsilon}\big(1 - \theta\cdot\theta' \big)\,\frac{ \hat{u}(t_0,k,\theta')\,e^{-iwt_0} }{ 1 + i(w + k\cdot\theta')/\lambda}\,\dtheta' 
\Denote \text{T}_{1} + \text{T}_{2} + \text{T}_{3}\,, \nn
\end{align}
where $\mathcal{K}: =c_{s,d}\textbf{1} + \mathcal{I}_{h}$ is the bounded part of $\mathcal{I}$.\\\\
\noindent
\textit{Estimating the term $\text{T}_{1}$}.  Simply note that
\begin{align}
\big|\text{T}_{1}(w,k,\theta)\big| \leq \Bigg(\int_{\mathbb{S}^{d-1}}&\frac{\rho_{\epsilon}(1 - \theta\cdot\theta')}{\big| 1 + i(w + k\cdot\theta')/\lambda\big|^{2}}\,\dtheta' \Bigg)^{\frac{1}{2}}\times \nonumber
\\
& \times\Bigg[\Bigg(\int_{\mathbb{S}^{d-1}}\rho_{\epsilon}(1 - \theta\cdot\theta') \big|\widehat{u_{s}}(w,k,\theta')\big|^{2}\,\dtheta' \Bigg)^{\frac{1}{2}} \label{sre10.1} 
\\
&\hspace{1.5cm}+ \frac{1}{\lambda}\Bigg(\int_{\mathbb{S}^{d-1}}\rho_{\epsilon}(1 - \theta\cdot\theta') \big|\mathcal{K}(\hat{u})(w,k,\theta')\big|^{2}\,\dtheta' \Bigg)^{\frac{1}{2}}\Bigg]\,. \nn
\end{align}
The first integral in \eqref{sre10.1} is estimated observing that
\begin{equation*}
\rho_{\epsilon}(z) \lesssim \frac{1}{C_{\epsilon}\,\epsilon^{\frac{d-1}{2}}}\,\text{1}_{\{|z|\leq\epsilon\}}\quad\quad\text{and}\quad\quad |\theta - \theta'|^{2} = 2\,\big(1-\theta\cdot\theta'\big)\,.
\end{equation*}
Additionally, choosing $\hat{k}$ as the north pole of $\mathbb{S}^{d-1}$ we can decompose any vector $\theta\in\mathbb{S}^{d-1}$ as $\theta = (\theta\cdot\hat{k})\,\hat{k} + \theta_{\perp}$ with $\theta_{\perp}\in\mathbb{D}^{d-1}$.   It follows that
\begin{align}\label{sre11}
\rho_{\epsilon}\big(1 - \theta\cdot\theta' \big) &\lesssim \frac{1}{C_{\epsilon}\,\epsilon^{\frac{d-1}{2}}}\,\text{1}_{\{|\theta - \theta'|^{2}\leq 2\epsilon\}} \nonumber\\
&= \frac{1}{C_{\epsilon}\,\epsilon^{\frac{d-1}{2}}}\,\text{1}_{\{|\theta \cdot \hat{k} - \theta'\cdot\hat{k}|^{2} + |\theta_{\perp} - \theta'_{\perp}|^{2} \leq 2\epsilon\}} \\
&\leq \frac{1}{C_{\epsilon}\,\epsilon^{\frac{d-1}{2}}}\,\text{1}_{\{|\theta \cdot \hat{k} - \theta'\cdot\hat{k}|^{2} \leq 2\epsilon\}}\,\text{1}_{ \{|\theta_{\perp} - \theta'_{\perp}|^{2} \leq 2\epsilon\}}\,. \nn
\end{align}
In this way, using \eqref{sre11} we can establish  that
\begin{align}\label{sre12}
& \quad \,
\int_{\mathbb{S}^{d-1}}\frac{\rho_{\epsilon}\big(1 - \theta\cdot\theta' \big)}{\big|1 + i(w + k\cdot\theta')/\lambda \big|^{2}}\,\dtheta' \nonumber\\
&\leq \frac{1}{C_{\epsilon}\,\epsilon^{\frac{d-1}{2}}}\int^{\pi}_{0}\frac{ \text{1}_{\{|\theta \cdot \hat{k} - \cos(\alpha)|^{2} \leq 2\epsilon\}} }{ \big|1 + i(w + |k|\cos(\alpha))/\lambda \big|^{2} } \Bigg(\int_{\mathbb{S}^{d-2}} \text{1}_{ \{|\theta_{\perp} -\sin(\alpha)\sigma |^{2} \leq 2\epsilon\}}\,\text{d}\sigma\,\sin^{d-2}(\alpha)\Bigg)\,\text{d}\alpha \nonumber\\
&\lesssim \frac{1}{C_{\epsilon}\,\sqrt{\epsilon}}\int^{\pi}_{0}\frac{ \text{1}_{\{|\theta \cdot \hat{k} - \cos(\alpha)|^{2} \leq 2\epsilon\}} }{ \big|1 + i(w + |k|\cos(\alpha))/\lambda \big|^{2} } \,\text{d}\alpha\,.
\end{align}
The last integral in \eqref{sre12} can be estimated using Parseval's theorem
\begin{align}\label{sre13}
\frac{1}{C_{\epsilon}\,\sqrt{\epsilon}}\int^{\pi}_{0}\frac{ \text{1}_{\{|\theta \cdot \hat{k} - \cos(\alpha)|^{2} \leq 2\epsilon\}} }{ \big|1 + i(w + |k|\cos(\alpha))/\lambda \big|^{2} }& \,\text{d}\alpha \lesssim \frac{1}{C_{\epsilon}\,\sqrt{\epsilon}}\int^{1}_{-1}\frac{ 1 }{ 1 + \big( ( w + |k|\,z)/\lambda \big)^{2} } \, \frac{1}{\sqrt{1 - z^{2}}}\dz \nonumber\\
& \hspace{-4cm}= \frac{1}{\sqrt{2}\,C_{\epsilon}\,\sqrt{\epsilon}}\int^{\infty}_{-\infty}\frac{ 1 }{ 1 + \big( ( w + |k|\,z)/\lambda \big)^{2} } \, \Bigg( \frac{1}{|1-z|^{\frac{1}{2}}} + \frac{1}{|1+z|^{ \frac{1}{2} } }\Bigg)\dz \\
& \hspace {-4cm} = \frac{C}{\sqrt{2}\,C_{\epsilon}\,\sqrt{\epsilon}}\frac{\lambda}{2\,|k|}\int^{\infty}_{-\infty} e^{-i\frac{w}{|k|}\xi}\,e^{-\frac{\lambda}{|k|}|\xi|}\,\frac{\cos(\xi)}{|\xi|^{\frac{1}{2}}}\,\dxi \lesssim \frac{1}{C_{\epsilon}\,\sqrt{\epsilon}}\sqrt{ \frac{\lambda}{|k|} }\,. \nn
\end{align}
Using \eqref{sre13} in \eqref{sre10.1} one obtains that the $L^{2}_{\theta}$-norm of $\text{T}_{1}$ is estimated by
\begin{align}\label{sre15}
& \quad \,
\big\| \text{T}_{1}(w,k,\cdot) \big\|_{L^{2}(\mathbb{S}^{d-1})}\nonumber
\\
& \leq C\,\Bigg(\frac{1}{\sqrt{\epsilon}}\sqrt{ \frac{\lambda}{|k|} }\Bigg)^{\frac{1}{2}}\Big( \big\|\widehat{u}(w,k,\cdot) \big\|_{L^{2}(\mathbb{S}^{d-1})} + \tfrac{1}{\lambda}\big\|\mathcal{K}(\widehat{u})(w,k,\cdot) \big\|_{L^{2}(\mathbb{S}^{d-1})} \Big)\,.
\end{align}
\textit{Estimating the term $T_{2}$}.  Let us use the stereographic projection and the definition of the operator $(-\Delta_{\theta})^{s}$ to obtain
\begin{align}\label{sre16}
\text{T}_{2}(w,k,\theta) &= -2^{d-1}\,\frac{D}{\lambda}\int_{\mathbb{R}^{d-1}}\rho_{\epsilon}\big(1 - \theta\cdot\mathcal{J}(v) \big)\,\frac{\big[(-\Delta_{\theta'})^{s}\hat{u}\big]_{\mathcal{J}}(w,k,v) }{ 1 + i(w + k\cdot\mathcal{J}(v))/\lambda}\,\frac{\dv}{\langle v \rangle^{d-1}} \nn
\\
& \hspace{-1cm}= -2^{d-1}\,\frac{D}{\lambda}\int_{\mathbb{R}^{d-1}} \frac{\rho_{\epsilon}\big(1 - \theta\cdot\mathcal{J}(v) \big)}{ \big(1 + i(w + k\cdot\mathcal{J}(v))/\lambda\big)\,\langle v \rangle^{\frac{d-1}{2} - s} }\,(-\Delta_{v})^{s}\widehat{w_{\mathcal{J}}}(w,k,v)\,\dv  
\\
& \hspace{-1cm}=-2^{d-1}\,\frac{D}{\lambda}\int_{\mathbb{R}^{d-1}} \nabla_{v}\Bigg[\frac{\rho_{\epsilon}\big(1 - \theta\cdot\mathcal{J}(v) \big)}{ \big(1 - i(w + k\cdot\mathcal{J}(v))/\lambda\big)\,\langle v \rangle^{\frac{d-1}{2} - s} }\Bigg]\cdot\nabla^{2s-1}_{v}\widehat{w_{\mathcal{J}}}(w,k,v)\,\dv\,, \nn
\end{align}
where the fractional gradient operator $\nabla^{2s-1}$ is defined by Fourier transform as
\begin{equation*}
\mathcal{F}\{\nabla^{2s-1}_{v}\psi\}(\xi)=-i|\xi|^{2s-1}\,\hat{\xi}\,\mathcal{F}\{\psi\}(\xi)\,.
\end{equation*}
Now, explicitly compute the gradient inside the last integral in \eqref{sre16} to obtain 3 terms, namely,
\begin{align*}
& \quad \,
\nabla_{v}\Bigg[\frac{\rho_{\epsilon}\big(1 - \theta\cdot\mathcal{J}(v) \big)}{ \big(1 - i(w + k\cdot\mathcal{J}(v))/\lambda\big)\,\langle v \rangle^{\frac{d-1}{2} - s} }\Bigg] \\
&= \frac{\nabla_{v}\rho_{\epsilon}\big(1 - \theta\cdot\mathcal{J}(v) \big)}{ \big(1 - i(w + k\cdot\mathcal{J}(v))/\lambda\big)\,\langle v \rangle^{\frac{d-1}{2} - s} }
+\frac{\rho_{\epsilon}\big(1 - \theta\cdot\mathcal{J}(v) \big)}{\langle v \rangle^{\frac{d-1}{2} - s}}\, \nabla_{v} \frac{1}{\big(1 - i(w + k\cdot\mathcal{J}(v))/\lambda\big) } 
\\
& \quad \,
+ \frac{\rho_{\epsilon}\big(1 - \theta\cdot\mathcal{J}(v) \big)}{\big(1 - i(w + k\cdot\mathcal{J}(v))/\lambda\big)}\,\nabla_{v}\frac{1}{\langle v \rangle^{\frac{d-1}{2} - s} }\,,
\end{align*}
which give us the decomposition $\text{T}_{2} = \text{T}^{1}_{2} + \text{T}^{2}_{2} + \text{T}^{3}_{2}$ respectively.  Additionally, note that for any vector $x\in\mathbb{R}^{d-1}$
\begin{equation}\label{sre17}
\big|\nabla_{v} \big( x \cdot \mathcal{J}(v) \big)\big|\lesssim \frac{|x|}{\langle v \rangle}\,,
\end{equation}
that leads to the estimate for $\text{T}^{1}_{2}$:
\begin{align}\label{sre18}
\big|\text{T}^{1}_{2}(w,k,\theta)\big| &\lesssim \frac{D}{\lambda}\int_{\mathbb{R}^{d-1}} \frac{\big| \rho'_{\epsilon}\big(1 - \theta\cdot\mathcal{J}(v) \big)\big| }{ \big|1 - i(w + k\cdot\mathcal{J}(v))/\lambda\big|\,\langle v \rangle^{\frac{d-1}{2} - s+1} } \Big| \nabla^{2s-1}_{v}\widehat{w_{\mathcal{J}}}(w,k,v)\Big|\,\dv \nonumber\\
&\leq \frac{D}{\lambda}\Bigg(\int_{\mathbb{R}^{d-1}} \frac{\big| \rho'_{\epsilon}\big(1 - \theta\cdot\mathcal{J}(v) \big)\big| }{ \big|1 - i(w + k\cdot\mathcal{J}(v))/\lambda\big|^{q}\,\langle v \rangle^{(\frac{d-1}{2} - s+1)q} }\,\dv\Bigg)^{\frac{1}{q}} \times \\
&\hspace{1cm}\times\Bigg(\int_{\mathbb{R}^{d-1}} \big| \rho'_{\epsilon}\big(1 - \theta\cdot\mathcal{J}(v) \big)\big|\,\Big| \nabla^{2s-1}_{v}\widehat{w_{\mathcal{J}}}(w,k,v)\Big|^{p}\,\dv\Bigg)^{\frac{1}{p}}\,,\quad \frac{1}{q}+\frac{1}{p}=1\,. \nn
\end{align}
Using Sobolev embedding one has
\begin{align}\label{sre18.5}
\big\|\nabla^{2s-1}_{v}\widehat{w_{\mathcal{J}}}(w,k,\cdot) \big\|_{L^{p}(\mathbb{R}^{d-1})} &\leq C_{d,s}\big\|\nabla^{s}_{v}\widehat{w_{\mathcal{J}}}(w,k,\cdot) \big\|_{L^{2}(\mathbb{R}^{d-1})} \nonumber\\
&\hspace{-1cm}= C_{d,s}\big\|(-\Delta_{v})^{s/2}\widehat{w_{\mathcal{J}}}(w,k,\cdot) \big\|_{L^{2}(\mathbb{R}^{d-1})}\,,
\end{align}
for $\frac{1}{2} - \frac{1-s}{d-1} = \frac{1}{p}$. This defines our choice of $p:=p(d,s)>2$ in \eqref{sre18}.  In this way, 
\begin{align*}
   \big(\tfrac{d-1}{2} - s+1\big)\,q=d-1\,,
\end{align*}
and estimate \eqref{sre18} reduces to
\begin{align}\label{sre19}
\big|\text{T}^{1}_{2}(w,k,\theta)\big| &\lesssim \frac{D}{\lambda}\Bigg(\int_{\mathbb{S}^{d-1}} \frac{\big| \rho'_{\epsilon}\big(1 - \theta\cdot\theta' \big)\big| }{ \big|1 - i(w + k\cdot\theta')/\lambda\big|^{q} }\,\dtheta'\Bigg)^{\frac{1}{q}}\times\nonumber\\
&\hspace{1cm}\times\Bigg(\int_{\mathbb{R}^{d-1}} \big| \rho'_{\epsilon}\big(1 - \theta\cdot\mathcal{J}(v) \big)\big|\,\Big| \nabla^{2s-1}_{v}\widehat{w_{\mathcal{J}}}(w,k,v)\Big|^{p}\,\dv\Bigg)^{\frac{1}{p}}\,.
\end{align}
It follows, after estimating the integral in the sphere as previously done for the term $\text{T}_{1}$, that
\begin{equation}\label{sre20}
\int_{\mathbb{S}^{d-1}}\frac{\big|\rho'_{\epsilon}\big(1 - \theta\cdot\theta' \big)\big|}{\big|1 - i(w + k\cdot\theta')/\lambda \big|^{q}}\,\dtheta' \leq \frac{C}{\epsilon\,\sqrt{\epsilon} }\sqrt{\frac{\lambda}{|k|}}\,.
\end{equation}
Thus, estimates \eqref{sre18.5}, \eqref{sre19} and \eqref{sre20} lead to
\begin{equation}\label{sre21}
\big\|\text{T}^{1}_{2}(w,k,\cdot)\big\|_{L^{p}(\mathbb{S}^{d-1})}\leq \frac{D}{\lambda\,\epsilon}\Bigg(\frac{1}{\sqrt{\epsilon} }\sqrt{\frac{\lambda}{|k|}}\Bigg)^{\frac{1}{q}}\,\big\|(-\Delta_{v})^{s/2}\widehat{w_{\mathcal{J}}}(w,k,\cdot) \big\|_{L^{2}(\mathbb{R}^{d-1})}\,.
\end{equation}
Similarly, the term $\text{T}^{2}_{2}$ is simply computed as
\begin{align*}
\big|\text{T}^{2}_{2}(w,k,\theta)\big| &\lesssim \frac{D\,|k|}{\lambda^{2}}\int_{\mathbb{R}^{d-1}} \frac{ \rho_{\epsilon}\big(1 - \theta\cdot\mathcal{J}(v) \big) }{ \big|1 - i(w + k\cdot\mathcal{J}(v))/\lambda\big|^{2}\,\langle v \rangle^{\frac{d-1}{2} - s+1} } \Big| \nabla^{2s-1}_{v}\widehat{w_{\mathcal{J}}}(w,k,v)\Big|\,\dv \nonumber\\
&\leq \frac{D\,|k|}{\lambda^{2}}\Bigg(\int_{\mathbb{S}^{d-1}} \frac{\rho_{\epsilon}\big(1 - \theta\cdot\theta' \big) }{ \big|1 - i(w + k\cdot\theta')/\lambda\big|^{2q}}\,\dtheta'\Bigg)^{\frac{1}{q}} \times \nonumber\\
&\hspace{2cm}\times\Bigg(\int_{\mathbb{R}^{d-1}} \rho_{\epsilon}\big(1 - \theta\cdot\mathcal{J}(v) \big)\,\Big| \nabla^{2s-1}_{v}\widehat{w_{\mathcal{J}}}(w,k,v)\Big|^{p}\,\dv\Bigg)^{\frac{1}{p}} \nonumber\\
&\lesssim \frac{D\,|k|}{\lambda^{2}}\Bigg(\frac{1}{\sqrt{\epsilon}}\sqrt{\frac{\lambda}{|k|}} \Bigg)^{\frac{1}{q}} \Bigg(\int_{\mathbb{R}^{d-1}} \rho_{\epsilon}\big(1 - \theta\cdot\mathcal{J}(v) \big)\,\Big| \nabla^{2s-1}_{v}\widehat{w_{\mathcal{J}}}(w,k,v)\Big|^{p}\,\dv\Bigg)^{\frac{1}{p}}\,,
\end{align*}
where the exponents $p$ and $q$ are those of the term $\text{T}^{1}_{2}$.  Previous estimate lead us to the bound
\begin{equation}\label{sre22}
\big\|\text{T}^{2}_{2}(w,k,\cdot)\big\|_{L^{p}(\mathbb{S}^{d-1})}\leq \frac{D\,|k|}{\lambda^{2}}\Bigg(\frac{1}{\sqrt{\epsilon}}\sqrt{\frac{\lambda}{|k|}} \Bigg)^{\frac{1}{q}} \,\big\|(-\Delta_{v})^{s/2}\widehat{w_{\mathcal{J}}}(w,k,\cdot) \big\|_{L^{2}(\mathbb{R}^{d-1})}\,.
\end{equation}
For the final term $\text{T}^{3}_{2}$ note that $\nabla_{v} \frac{1}{\langle v \rangle^{\frac{d-1}{2}-s}} = -(\frac{d-1}{2}-s) \frac{2\,v}{\langle v \rangle^{\frac{d-1}{2}-s+1}}$, therefore the stereographic projection leads to
\begin{align}\label{sre23}
\big|\text{T}^{3}_{2}(w,k,\theta)\big| & \lesssim \frac{D}{\lambda}\Bigg(\int_{\mathbb{S}^{d-1}} \frac{\rho_{\epsilon}\big(1 - \theta\cdot\theta' \big) }{ \big|1 - i(w + k\cdot\theta')/\lambda\big|^{q}}\Bigg(\frac{|\theta^{' \perp}|}{1-\theta'_{d}}\Bigg)^{q}\,\dtheta'\Bigg)^{\frac{1}{q}} \times \nonumber\\
&\hspace{2cm}\times\Bigg(\int_{\mathbb{R}^{d-1}} \rho_{\epsilon}\big(1 - \theta\cdot\mathcal{J}(v) \big)\,\Big| \nabla^{2s-1}_{v}\widehat{w_{\mathcal{J}}}(w,k,v)\Big|^{p}\,\dv\Bigg)^{\frac{1}{p}}\,.
\end{align}
Note that $\frac{|\theta^{'\perp}|}{1-\theta'_{d}}\leq \frac{1}{\sin(\alpha)}$, with $\alpha$ the polar angle.  Therefore, the following estimate is valid for any $d\geq3$ (recall that $q\in(0,2)$)
\begin{align}\label{sre24}
& \quad \,
\int_{\mathbb{S}^{d-1}} \frac{\rho_{\epsilon}\big(1 - \theta\cdot\theta' \big)}{\big|1 + i(w + k\cdot\theta')/\lambda \big|^{q}}\Big(\frac{|\theta^{'\perp}|}{1-\theta'_{d}}\Big)^{q}\,\text{d}\theta' \nonumber
\\
& \leq \frac{1}{C_{\epsilon}\,\epsilon^{\frac{d-1}{2}}}\int^{\pi}_{0}\frac{ \text{1}_{\{|\theta \cdot \hat{k} - \cos(\alpha)|^{2} \leq 2\epsilon\}} }{ \big|1 + i(w + |k|\cos(\alpha))/\lambda \big|^{q} } \Bigg(\int_{\mathbb{S}^{d-2}} \text{1}_{ \{|\theta_{\perp} -\sin(\alpha)\sigma |^{2} \leq 2\epsilon\}}\,\text{d}\sigma\,\sin^{d-2-q}(\alpha)\Bigg)\,\text{d}\alpha \nonumber
\\
& \lesssim \frac{1}{C_{\epsilon}\,\epsilon}\int^{\pi}_{0}\frac{ \text{1}_{\{|\theta \cdot \hat{k} - \cos(\alpha)|^{2} \leq 2\epsilon\}} }{ \big|1 + i(w + |k|\cos(\alpha))/\lambda \big|^{q} } \frac{\text{d}\alpha}{\sin(\alpha)^{q-1}} \lesssim \frac{1}{C_{\epsilon}\,\epsilon}\Bigg(\frac{\lambda}{|k|}\Bigg)^{1-\frac{q}{2}}\,.
\end{align}
With estimate \eqref{sre24} we finally conclude that
\begin{equation}\label{sre25}
\big\|\text{T}^{3}_{2}(w,k,\cdot)\big\|_{L^{p}(\mathbb{S}^{d-1})}\leq \frac{D}{\lambda\,\sqrt{\epsilon}}\,\Bigg(\frac{1}{\sqrt{\epsilon}}\sqrt{\frac{\lambda}{|k|}}\Bigg)^{\frac{2-q}{q}}\big\|(-\Delta_{v})^{s/2}\widehat{w_{\mathcal{J}}}(w,k,\cdot) \big\|_{L^{2}(\mathbb{R}^{d-1})}\,.
\end{equation}
\textit{Estimating the boundary term $\text{T}_{3}$} In the same spirit of previous calculations we have
\begin{align}\label{sre26}
\big|T_{3}(w,k,\theta)\big|& =\frac{1}{\lambda}\int_{\mathbb{S}^{d-1}}\rho_{\epsilon}\big(1 - \theta\cdot\theta' \big)\,\frac{ \big|\hat{u}(t_0,k,\theta')\big| }{ \big| 1 + i(w + k\cdot\theta')/\lambda\big|}\,\dtheta'\nonumber
\\
&\leq \frac{1}{\lambda}\Bigg(\int_{\mathbb{S}^{d-1}}\frac{ \rho_{\epsilon}\big(1 - \theta\cdot\theta' \big) }{ \big| 1 + i(w + k\cdot\theta')/\lambda\big|^{2(1-s_0)}}\,\dtheta'\Bigg)^{\frac{1}{2}} \times 
\\
& \qquad \, \times \Bigg(\int_{\mathbb{S}^{d-1}}\rho_{\epsilon}\big(1 - \theta\cdot\theta' \big)\,\frac{ \big|\hat{u}(t_0,k,\theta')\big|^{2} }{ \big| 1 + i(w + k\cdot\theta')/\lambda\big|^{2s_0}}\,\dtheta'\Bigg)^{\frac{1}{2}} \,, \nn
\end{align}
where $s_0\in(\frac{1}{2},1)$ will be chosen in a moment.  Observe that for the first integral
\begin{align}\label{sre27}
& \quad \,
\int_{\mathbb{S}^{d-1}}\frac{ \rho_{\epsilon}\big(1 - \theta\cdot\theta' \big) }{ \big| 1 + i(w + k\cdot\theta')/\lambda\big|^{2(1-s_0)}}\,\dtheta' \nonumber
\\
&\lesssim  \frac{1}{C_{\epsilon}\,\sqrt{\epsilon}}\int^{1}_{-1}\frac{ 1 }{ \big| 1 + \big( ( w + |k|\,z)/\lambda \big)^{2} \big|^{1-s_0} } \, \frac{1}{\sqrt{1 - z^{2}}}\dz 
\\
&\lesssim \frac{1}{\sqrt{\epsilon}}\sqrt{\frac{\lambda}{|k|}}\int^{\infty}_{-\infty}\mathcal{F}\big\{\mathcal{B}_{2(1-s_0)}\big\}(\xi)\,\frac{1}{|\xi|^{\frac{1}{2}}}\,\dxi \lesssim \frac{1}{\sqrt{\epsilon}}\sqrt{\frac{\lambda}{|k|}}\,. \nn
\end{align}
Recall that $\mathcal{B}_{2(1-s_0)}$ is the Bessel potential of order $2(1-s_0)$, thus, previous estimate is valid for $s_0$ sufficiently close to $\frac{1}{2}$ and such that the singularity at $\xi=0$ becomes integrable.  More precisely, from the short discussion in the appendix about Bessel potentials one notices that any $s_0\in(\frac{1}{2},\frac{3}{4})$ will do.  Plug estimate \eqref{sre27} in \eqref{sre26} and integrating in $(w,\theta)$ variables to obtain
\begin{equation}\label{sre28}
\big\|\text{T}_{3}(\cdot,k,\cdot)\big\|_{L^{2}(\mathbb{R}\times\mathbb{S}^{d-1})}\leq \frac{C}{\sqrt{\lambda}}\,\Bigg(\frac{1}{\sqrt{\epsilon}}\sqrt{\frac{\lambda}{|k|}}\Bigg)^{\frac{1}{2}}\big\|\hat{u}(t_0,k,\cdot) \big\|_{L^{2}(\mathbb{S}^{d-1})}\,.
\end{equation}
\textit{Conclusion of the proof.}  From the decomposition \eqref{sre7} and estimates \eqref{sre8}, \eqref{sre15}, \eqref{sre21}, \eqref{sre22}, \eqref{sre25} and \eqref{sre28} one concludes
\begin{align}\label{sre29}
\big\|\hat{u}(\cdot,k,\cdot)\big\|_{L^{2}(\mathbb{R}\times\mathbb{S}^{d-1})} &\leq \Bigg(\frac{1}{\sqrt{\epsilon}}\sqrt{ \frac{\lambda}{|k|} }\Bigg)^{\frac{1}{2}}\Big( \big\|\widehat{u}(\cdot,k,\cdot) \big\|_{L^{2}(\mathbb{R}\times\mathbb{S}^{d-1})} + \tfrac{1}{\lambda}\big\|\mathcal{K}(\widehat{u})(\cdot,k,\cdot) \big\|_{L^{2}(\mathbb{R}\times\mathbb{S}^{d-1})} \Big) \nonumber\\
&\hspace{-2.8cm} + \frac{C}{\sqrt{\lambda}}\,\Bigg(\frac{1}{\sqrt{\epsilon}}\sqrt{\frac{\lambda}{|k|}}\Bigg)^{\frac{1}{2}}\big\|\hat{u}(t_0,k,\cdot) \big\|_{L^{2}(\mathbb{S}^{d-1})} + D\Bigg(\frac{1}{\lambda\,\epsilon}\Bigg(\frac{1}{\sqrt{\epsilon} }\sqrt{\frac{\lambda}{|k|}}\Bigg)^{\frac{1}{q}}+\frac{|k|}{\lambda^{2}}\Bigg(\frac{1}{\sqrt{\epsilon}}\sqrt{\frac{\lambda}{|k|}} \Bigg)^{\frac{1}{q}}+ \nonumber\\
&\hspace{0.5cm}\frac{1}{\lambda\,\sqrt{\epsilon}}\,\Bigg(\frac{1}{\sqrt{\epsilon}}\sqrt{\frac{\lambda}{|k|}}\Bigg)^{\frac{2-q}{q}}+\epsilon^{\frac{s}{2}}\Bigg) \big\|(-\Delta_{v})^{s/2}\widehat{w_{\mathcal{J}}}(\cdot,k,\cdot) \big\|_{L^{2}(\mathbb{R}\times\mathbb{R}^{d-1})}\,.
\end{align}
With estimate \eqref{sre29} we are looking to find control for $|k|$ large.  Indeed, assume that $|k|\geq 1$ and set $\epsilon = \frac{1}{|k|^{a}}$ and $\lambda=|k|^{b}$ with $a, b >0$.  Since we expect that
\begin{equation*}
\frac{1}{\sqrt{\epsilon}}\sqrt{ \frac{\lambda}{|k|} } \sim \frac{1}{|k|^{s_0}}\,,\quad s_0>0\,,
\end{equation*}
we can control the term
\begin{equation*}
\frac{|k|}{\lambda^{2}}\Bigg(\frac{1}{\sqrt{\epsilon}}\sqrt{\frac{\lambda}{|k|}} \Bigg)^{\frac{1}{q}} \quad \text{by choosing} \quad \frac{|k|}{\lambda^{2}}=1\,,
\end{equation*}
that is, choosing $b=\frac{1}{2}$.  Using that $q\in(1,2)$ one concludes that the leading terms are
\begin{equation*}
\Bigg(\frac{1}{\sqrt{\epsilon}}\sqrt{ \frac{\lambda}{|k|} }\Bigg)^{\frac{1}{2}}\,,\;\;\frac{1}{\lambda\,\sqrt{\epsilon}}\,\Bigg(\frac{1}{\sqrt{\epsilon}}\sqrt{\frac{\lambda}{|k|}}\Bigg)^{\frac{2-q}{q}}\,,\;\;\text{and}\;\;\epsilon^{\frac{s}{2}}\,.
 \end{equation*}
Therefore, the best option independent of the dimension is choosing $a$ such that
\begin{equation*}
\max\Big\{\Bigg(\frac{1}{\sqrt{\epsilon}}\sqrt{ \frac{\lambda}{|k|} }\Bigg)^{\frac{1}{2}}\,,\frac{1}{\lambda\,\sqrt{\epsilon}} \Big\} = \epsilon^{\frac{s}{2}}\,.
\end{equation*}
A simple calculation shows that $a=\frac{1/2}{2s+1}$, and therefore, from \eqref{sre29} one concludes that
\begin{align*}\label{sre20}
\big\|\hat{u}(\cdot,k,\cdot)&\big\|_{L^{2}(\mathbb{R}\times\mathbb{S}^{d-1})}\leq \frac{C}{|k|^{\frac{s/4}{2s+1}}}\Bigg( \big\|\widehat{u}(\cdot,k,\cdot) \big\|_{L^{2}(\mathbb{R}\times\mathbb{S}^{d-1})} + \big\|\mathcal{K}(\widehat{u})(\cdot,k,\cdot) \big\|_{L^{2}(\mathbb{R}\times\mathbb{S}^{d-1})} \\
&+\big\|(-\Delta_{v})^{s/2}\widehat{w_{\mathcal{J}}}(\cdot,k,\cdot) \big\|_{L^{2}(\mathbb{R}\times\mathbb{R}^{d-1})} + \big\|\hat{u}(t_0,k,\cdot) \big\|_{L^{2}(\mathbb{S}^{d-1})} \Bigg)\,,\quad |k|\geq 1\,.
\end{align*}
This inequality proves the result recalling that $\mathcal{K}$ is a bounded operator in $L^{2}(\mathbb{S}^{d-1})$.
\end{proof}

\begin{Corollary}\label{Cor:strong-reg-avg-time}
Let $u$ be a solution to~\eqref{sr:main} which satisfies the conditions in Theorem \ref{st-reg-lemma}. Then for any $t_\ast \in (t_0, t_1)$, we have
\begin{equation} 
    \label{est:time-average}
\begin{aligned}
     \big\|(-\Delta_{x})^{\frac{s_0}{2}} u \big\|_{L^{2}([t_\ast,t_1)\times\mathbb{R}^{d}\times\mathbb{S}^{d-1})}
&\leq 
    C\,\Big( \frac{1}{\sqrt{t_\ast - t_0}} + 1 \Big) \big\|u \big\|_{L^{2}([t_0,t_1)\times\mathbb{R}^{d}\times\mathbb{S}^{d-1})} \\
& \quad \,
     + C\big\|(-\Delta_{v})^{s/2}w_{\mathcal{J}} \big\|_{L^{2}([t_0,t_1)\times\mathbb{R}^{d}\times\mathbb{R}^{d-1})}\,,
\end{aligned}
\end{equation}
where $s_0=\frac{s/4}{2s+1}$. 
\end{Corollary}
\begin{proof}
Let $\tau \in [t_0, t_1)$ be arbitrary. Then~\eqref{sr:main:estimate} gives
\begin{align*}
     \big\|(-\Delta_{x})^{\frac{s_0}{2}} u &\big\|_{L^{2}([\tau,t_1)\times\mathbb{R}^{d}\times\mathbb{S}^{d-1})}^2
\leq 
    C\,\Big(  \big\| u(\tau)\big\|_{L^{2}(\mathbb{R}^{d}\times\mathbb{S}^{d-1})}^2 
                + \big\|u \big\|_{L^{2}([t_0,t_1)\times\mathbb{R}^{d}\times\mathbb{S}^{d-1})}^2  \nonumber\\
&\hspace{1cm}
               +\big\|(-\Delta_{v})^{s/2}w_{\mathcal{J}} \big\|_{L^{2}([t_0,t_1)\times\mathbb{R}^{d}\times\mathbb{R}^{d-1})}^2 \Big)\,,
\quad\quad s_0=\frac{s/4}{2s+1}\,.
\end{align*}
Taking the average of the above inequality over $[t_0, t_\ast]$, we have
\begin{align*}
& \quad \,
    \big\|(-\Delta_{x})^{\frac{s_0}{2}} u \big\|_{L^{2}([t_\ast,t_1)\times\mathbb{R}^{d}\times\mathbb{S}^{d-1})}^2
\leq    \frac{1}{t_\ast - t_0}\int_{t_0}^{t_*} \big\|(-\Delta_{x})^{\frac{s_0}{2}} u \big\|_{L^{2}([\tau,t_1)\times\mathbb{R}^{d}\times\mathbb{S}^{d-1})}^2 {\, \rm d}\tau
\\
&\leq 
    C\,\Big( \frac{1}{t_\ast - t_0}\int_{t_0}^{t_*} \big\| u(\tau)\big\|_{L^{2}(\mathbb{R}^{d}\times\mathbb{S}^{d-1})}^2  {\, \rm d}\tau
                + \norm{u}_{L^{2}([t_0,t_1)\times\mathbb{R}^{d}\times\mathbb{S}^{d-1})}^2 +
\\
& \hspace{7cm}  
               +\norm{(-\Delta_{v})^{s/2}w_{\mathcal{J}}}_{L^{2}([t_0,t_1)\times\mathbb{R}^{d}\times\mathbb{R}^{d-1})}^2 \Big)
\\
&\leq 
    C\,\Big( \frac{1}{t_\ast - t_0}\int_{t_0}^{t_1} \big\| u(\tau)\big\|_{L^{2}(\mathbb{R}^{d}\times\mathbb{S}^{d-1})}^2  {\, \rm d}\tau
                + \big\|u \big\|_{L^{2}([t_0,t_1)\times\mathbb{R}^{d}\times\mathbb{S}^{d-1})}^2  + 
\\
& \hspace{7cm}
               +\big\|(-\Delta_{v})^{s/2}w_{\mathcal{J}} \big\|_{L^{2}([t_0,t_1)\times\mathbb{R}^{d}\times\mathbb{R}^{d-1})}^2 \Big) \,.
\end{align*}
Inequality~\eqref{est:time-average} is then obtained by taking square root on both sides of the above inequality. 
\end{proof}

\section{A Priori Estimates: Smoothing}
In this section we study the regularity of the solution. In particular, we will show that solutions with $L^1$ initial data will gain immediate smoothness.  Generally speaking, the solution will enjoy higher regularity in the space and time variables.  The solution will enjoy regularity in the angular variable as well, however, this regularity will be tied to the regularity of the scattering kernel $b$.
\subsection{Regularity - From $L^1$ to $L^2$}\label{section:regularity-L-2} 
First we show that solutions with $L^1$ initial data will become $L^2$ for any positive time. In addition, we will use the work done in Section 3 to shown a gain of a fractional derivative in both $x$ and $\theta$.  We start by showing an interpolation between the total mass of the density function $u$ and its fractional derivatives in $x$ and $\Vtheta$. This will give us an $L^{2\omega}$-bound (in time) for $u$ with some $\omega > 1$. Recall the notation
\begin{align*}
    \rho(t, x) = \int_{\Ss^{d-1}} u(t, x, \Vtheta) \dtheta \,.
\end{align*}

Throughout this subsection, we use $c_{d,s,s'}$ to denote any constant that depends only on $d,s,s'$. We use  $c_{1,m_0}$ for any constant that only depends on $d,s, s',m_0$ where $m_0 = \int_{\R^d}\int_{\Ss^{d-1}} u(t, x, \Vtheta) \dtheta \dx$ is the total mass. These constants may change from line to line.

\begin{Proposition} \label{prop:interpolation}
Suppose $u \geq 0$ and $u \in L^2\big([t_0, t] \times \R^d; H^s(\Ss^{d-1})\big)$ for some $s > 0$.  Suppose $\rho \in L^2\big([t_0, t]; H^{s'}(\R^d)\big)$ for some $s' \in (0, 1)$. Then there exists $\omega > 1$ such that $u \in L^{2\omega}\big([t_0, t]; L^2(\R^d \times \Ss^{d-1})\big)$. Moreover, there exists a constant $c_{1,m_0} > 0$ such that
\begin{equation} \nn
     \int_{t_0}^t \|u\|_{L^2_{x, \Vtheta}}^{2\omega}(\tau) \dtau
\leq
     c_{1,m_0}
     \left(\int_{t_0}^t \int_{\R^d} \|u\|_{H^s_\theta}^2 (\tau, x) \dx \dtau
           + \int_{t_0}^t \|(-\Delta_x)^{s'/2} \rho\|_{L^2_x}^2 (\tau) \dtau \right) \,.
\end{equation}
\end{Proposition}
\begin{proof}
By~\eqref{smooth10} and Sobolev imbedding, for each $\tau \in (t_0, t)$ we have 
\begin{equation} \nn
\begin{aligned}
     \int_{\R^d}\|u(\tau, x, \cdot)\|_{H^s_\theta}^2 \dx
&\geq 
     c_{0, d, s} \int_{\R^d} 
              \Big(\int_{\Ss^{d-1}} u^p(x, \Vtheta) \dtheta\Big)^{2/p} 
            \dx \,,
\\
    \|(-\Delta_x)^{s'/2} \rho(\tau, \cdot)\|_{L^2_x}^2
&\geq 
     c_{0, d, s} \left(\int_{\Ss^{d-1}} \rho^q(x, \Vtheta) \dtheta\right)^{2/q} 
             \,,
\end{aligned}
\end{equation}
where 
\begin{equation} \nn
     \frac{1}{q_2} = \frac 12 - \frac{s'}{d} \,, 
\qquad
     \frac{1}{p_s} = \frac 12 - \frac{s}{d-1} \,.
\end{equation}
Note that $q_2, p_s > 2$. Let
\begin{equation} \label{def:parameters-1}
            \alpha_1 = \frac{2 \, q_2 - 2}{p_s \, q_2 - 2} \in (0, 1) \,, 
\qquad \alpha_2 = \frac{p_s}{2} \, \alpha_1 \in (0, 1) \,,
\qquad  r = p_s \, \alpha_1 + (1 - \alpha_1) > 2 \,,
\end{equation} 
such that 
\begin{equation} \nn
\begin{aligned}
    \frac{\alpha_1}{\alpha_2} = \frac{2}{p_s} \,,
\qquad 
    &\frac{1-\alpha_1}{1-\alpha_2} = q_2 \,, 
\\[3pt]     
     r 
     = 2 + \frac{(p_s - 2)(q_2 - 2)}{p_s \, q_2 - 2} 
     = \, & 3 - \frac{2(p_s + q_2 - 2)}{p_s \, q_2 -2}
     \in (2, 3) \,.
\end{aligned}
\end{equation}
Then by H\"{o}der inequality, we have
\begin{equation} \nn
\begin{aligned}
   \left(\int_{\R^d} \int_{\Ss^{d-1}} u^r \dtheta\dx\right)^{2/r}
&\leq 
   \left(\int_{\R^d} 
      \left(\int_{\Ss^{d-1}} u^{p_s} \dtheta\right)^{\alpha_1} 
      \rho^{1-\alpha_1}
      \dx \right)^{2/r}
\\
&\hspace{-0.5cm}\leq 
   \left(\int_{\R^d} 
      \left(\int_{\Ss^{d-1}} u^{p_s} \dtheta\right)^{\frac{\alpha_1}{\alpha_2}} 
      \dx\right)^{2\alpha_2/r}
   \left(\int_{\R^d} \rho^{\frac{1-\alpha_1}{1-\alpha_2}} \dx \right)^{2(1-\alpha_2)/r}
\\
& \hspace{-0.5cm}= 
   \left(\int_{\R^d} 
      \left(\int_{\Ss^{d-1}} u^{p_s} \dtheta\right)^{2/p_s} 
      \dx\right)^{2\alpha_2/r}
   \left(\int_{\R^d} \rho^{q_2} \dx \right)^{2(1-\alpha_2)/r}
\\
& \hspace{-0.5cm}\leq
    c_{d,s,s'} \left(\int_{\R^d} \|u\|_{H^s_\theta}^2 \dx\right)^{2\alpha_2/r}
    \left(\int_{\R^d} \abs{(-\Delta_x)^{s'/2} \rho}^2 \dx\right)^{(1-\alpha_2)q/r}
\end{aligned}
\end{equation}
Note that by the choice of~\eqref{def:parameters-1}, the parameters satisfy 
\begin{equation} \nn
     \frac{2\alpha_2}{r} + \frac{(1-\alpha_2)q_2}{r} = 1 \,.
\end{equation}
Thus if we integrate in time, then
\begin{equation} \label{est:u-L-r}
\begin{aligned}
\int_{t_0}^t \Big(& \int_{\R^d} \int_{\Ss^{d-1}} u^r \dtheta\dx \Big)^{\frac{2}{r}}\dtau
\\
& \leq
    c_{d,s,s'} \int_{t_0}^t \left(\left(\int_{\R^d}\|u\|_{H^s_\theta}^2 \dx\right)^{\frac{2\alpha_2}{r}}
    \left(\|(-\Delta_x)^{s'/2} \rho\|_{L^2_x}^2 \right)^{\frac{(1-\alpha_2)q_2}{r}} \right)\dtau
\\
& \hspace{0.5cm}\leq
   c_{d,s,s'} \left(\int_{t_0}^t \int_{\R^d}\|u\|_{H^s_\theta}^2 \dx \dtau \right)^{\frac{2\alpha_2}{r}}
   \left(\int_{t_0}^t
    \|(-\Delta_x)^{s'/2} \rho\|_{L^2_x}^2 \dtau\right)^{\frac{(1-\alpha_2)q_2}{r}}
\\
&\hspace{1cm} \leq
   c_{d,s,s'} \left(\int_{t_0}^t \int_{\R^d}\|u\|_{H^s_\theta}^2 \dx\dtau
                + \int_{t_0}^t
    \|(-\Delta_x)^{s'/2} \rho\|_{L^2_x}^2 \dtau\right). 
\end{aligned}
\end{equation}
Let 
\begin{equation} \label{def:p-alpha}
     \omega = \frac{2(r-1)}{r} > 1\,, \qquad \alpha = \frac{1}{r-1} \in (0, 1)
\end{equation}
Then
\begin{equation} \nn
      r \alpha + (1- \alpha) = 2 \,,
\qquad
      \omega = \frac{2}{\alpha r}\,.
\end{equation}
By H\"{o}der inequality and~\eqref{est:u-L-r}, we have
\begin{equation} \nn
\begin{aligned}
   \int_{t_0}^t \left(\int_{\R^d} \int_{\Ss^{d-1}} u^2 \dtheta\dx \right)^{\omega}\dtau
&\leq
   \int_{t_0}^t 
      \left(\int_{\R^d} \int_{\Ss^{d-1}} u^r \dtheta\dx \right)^{\alpha \omega}
      \left(\int_{\R^d} \int_{\Ss^{d-1}} u \dtheta\dx \right)^{(1-\alpha) \omega}
   \dtau
\\
&\hspace{-1.5cm}\leq
   c_{1,m_0} \int_{t_0}^t 
      \left(\int_{\R^d} \int_{\Ss^{d-1}} u^r \dtheta\dx \right)^{2/r}
   \dtau
\\
&\hspace{-1cm} \leq
   c_{1,m_0} \left(\int_{t_0}^t \int_{\R^d}\|u\|_{H^s_\theta}^2 \dx \dtau
                + \int_{t_0}^t
    \|(-\Delta_x)^{s'/2} \rho\|_{L^2_x}^2 \dtau\right) \,,
\end{aligned}
\end{equation}
where $c_{1,m_0} = c_{d,s,s'} \, m_0^{(1-\alpha)\omega}$ is an increasing function in $m_0$.
\end{proof}
Proposition~\ref{prop:interpolation} shows that only spatial regularity is needed on the averaged quantity $\rho$ to obtain a bound on the full norm $\|u\|_{L^{2\omega}([t_0, t]; L^2_{x, \theta})}$.  An immediate corollary of Proposition \ref{Lemma:VA-Weak} is the following.
\begin{Corollary} \label{Cor:3-2}
Suppose $u \in L^2\big((t_0, t_1) \times \R^d \times \mathbb{S}^{d-1}\big)$ is a weak solution to the transport equation~\eqref{rte1} on $[t_0, t_1]$ with $0 < t_0 < t_1 < \infty$. Let $\epsilon_0 = \min\{t_1 - t_0, 1\}$. Then 
\begin{equation} \nn 
\begin{aligned}
     \sup_{t \in (t_0, t_0 + \epsilon_0)} \vpran{\|u\|_{L^2_{x, \Vtheta}}^2(t)}
     + \int_{t_0}^t \int_{\R^d}\|u\|_{H^s_\theta}^2 \dx\dtau
     + \int_{t_0}^t \|(-\Delta_x)^{\beta} \rho\|_{L^2_x}^2 \dtau 
\leq 
    c_2 \|u(t_0, \cdot, \cdot)\|_{L^2_{x, \Vtheta}}^2 \,,
\end{aligned}
\end{equation}
for any $t \in (t_0, t_0 + \epsilon_0)$. Here the constant $c_2$ depends on $d, s, \delta,$ and $\tilde b$ with $\delta$ defined in~\eqref{def:delta}.
\end{Corollary}
\begin{proof}
Recall that Proposition \ref{Lemma:VA-Weak} gives
\begin{align*} 
     \int_{t_0}^t \|(-\Delta_x)^{\beta} \rho\|_{L^2_{x}}^2 \dtau
\leq
    c_{d,s,\delta} \left(\|u(t_0, \cdot, \cdot)\|_{L^2_{x, \theta}}^2
                  + \|u\|_{L^2_{t,x,\theta}}^2 
                  + \|(-\Delta_v)^{s/2} \WJ\|_{L^2_{t, x, v}}^2
                  \right).
\end{align*}
Hence by the energy estimate~\eqref{MEE2}, 
\begin{equation} \nn 
\begin{aligned}
     \sup_{t \in (t_0, t_0 + \epsilon_0)} &\vpran{\|u\|_{L^2_{x, \Vtheta}}^2(t)}
     + \int_{t_0}^t \int_{\R^d}\|u\|_{H^s_\theta}^2 \dx\dtau
     + \int_{t_0}^t \|(-\Delta_x)^{\beta} \rho\|_{L^2_x}^2 \dtau 
\\
&\leq 
    c_{2,1} \|u(t_0, \cdot, \cdot)\|_{L^2_{x, \Vtheta}}^2
    + c_{2,1} \int_{t_0}^t \|u\|_{L^2_{x, \Vtheta}}^2 \dtau \,.
\end{aligned}
\end{equation}
where $c_{2,1}$ only depends on $d,s, \delta,$ and $\tilde b$. By~\eqref{MEE1}, we have
\begin{align*}
    \|u\|_{L^2_{x, \Vtheta}}^2 (t)
\leq 
   \|u\|_{L^2_{x, \Vtheta}}^2 (t_0) \,,
\qquad \text{for any $t \in (t_0, t_0 + \epsilon_0)$. }
\end{align*}
And thus,
\begin{equation} \nn 
\begin{aligned}
     \sup_{t \in (t_0, \epsilon_0)} &\vpran{\|u\|_{L^2_{x, \Vtheta}}^2(t)}
     + \int_{t_0}^t \int_{\R^d}\|u\|_{H^s_\theta}^2 \dx\dtau
     + \int_{t_0}^t \|(-\Delta_x)^{\beta} \rho\|_{L^2_x}^2 \dtau 
\\
&\leq 
    c_2 \|u(t_0, \cdot, \cdot)\|_{L^2_{x, \Vtheta}}^2 \,,
\end{aligned}
\end{equation}
where $c_2 = 2 c_{2,1}$ which only depends on $d,s, \delta,$ and $\tilde b$.
\end{proof}
\begin{Proposition} \label{prop:L-1-L-2}
Suppose $u$ is a weak solution to the transport equation~\eqref{rte1} on $[0, T]$. Let $T_1 = \min\{T, 1\}$. Then there exists a constant $c_3 = c_3(d,s,s',m_0, \delta, \tilde b)$ which is increasing in $m_0$ such that 
\begin{equation} \nn
    \|u(t)\|_{L^2_{x, \Vtheta}} 
\leq 
    c_3 \, t^{-\frac{1}{\omega-1}}\,,
\qquad \text{for all $0 < t < T_1$} \,,
\end{equation}
where $m_0 = \int_{\R^d} \int_{\Ss^{d-1}} u(x, \Vtheta) \dx\dtheta$ is the conserved mass 
and $\omega$ is defined in~\eqref{def:p-alpha}.
\end{Proposition}
\begin{proof} 
Let $T_1 = \min\{T, 1\}$. Then for any $t_0 \in (0, T_1)$, we apply Proposition~\ref{prop:interpolation} and Corollary~\ref{Cor:3-2} and obtain
\begin{equation} \label{est:u-2p}
      \int_{t_0}^t \|u(\tau)\|_{L^2_{x, \Vtheta}}^{2\omega} \dtau
 \leq 
      c_{3,1} \, \|u(t_0)\|_{L^2_{x, \Vtheta}}^2 \,,
\qquad t \in (t_0, T_1)  \,,
\end{equation}
where $c_{3,1} = c_{1,m_0} c_2 = c_{2} \, c_{d,s,s'} m_0^{(1-\alpha) \omega}$ is an increasing function in $m_0$.  Denote
\begin{equation} \nn
      X(\tau) = \|u(\tau)\|_{L^2_{x, \Vtheta}}^{2\omega} \,, 
\qquad \tau \in (t_0, T_1) \,.
\end{equation}
Then~\eqref{est:u-2p} becomes
\begin{equation} \nn
     \left(\int_{t_0}^t X(\tau) \dtau \right)^\omega 
 \leq (c_{3,1})^\omega X(t_0) \,,  
\qquad t \in (t_0, T_1) \,.
\end{equation}
If we fix $t \in (0, T_1)$ and further denote
\begin{equation} \nn
     Y(t_0) = \int_{t_0}^t X(\tau) \dtau \,, 
\qquad 0 < t_0 < t < T_1 \,,
\end{equation}
then
\begin{equation} \nn
     c_{3,1}^\omega Y'(t_0) + Y^\omega(t_0) \leq 0  \,, 
\qquad \omega > 1 \,.
\end{equation}
The key observation here is there exists a universal constant $c_0 > 0$ such that
\begin{equation} \nn
      \int_{t_0}^t \|u(\tau)\|_{L^2_{x, \Vtheta}}^{2\omega} \dtau
      = Y(t_0) 
      \leq c_{3,2} \, t_0^{-\frac{1}{\omega-1}} \,,
\qquad 0 < t_0 < t < T_1 \,,
\end{equation}
where $c_{3,2} = \vpran{\frac{c_{3,1}^\omega}{\omega - 1}}^{\frac{1}{\omega - 1}}$.
Recall the basic $L^2$-bound of the solution 
\begin{equation} \label{est:L-2-decay-1}
     \|u(t)\|_{L^2_{x, \Vtheta}}^{2\omega}
\leq \|u(\tau)\|_{L^2_{x, \Vtheta}}^{2\omega} \,, 
\qquad 0 < t_0 < \tau < t < T_1 \,.
\end{equation}
Taking the average of~\eqref{est:L-2-decay-1} from $t_0$ to $t$, we have
\begin{equation} \nn
     \|u(t)\|_{L^2_{x, \Vtheta}}^{2\omega}
\leq
    \frac{1}{t-t_0} \int_{t_0}^t \|u(\tau)\|_{L^2_{x, \Vtheta}}^{2\omega} \dtau
\leq
    \frac{c_{3,2} }{t-t_0} t_0^{-\frac{1}{\omega-1}} \,,
\qquad 0 < t_0 < t < T_1 \,.
 \end{equation}
In particular, if we take $0 < t_0 < T_1/2$ and $t = 2 t_0$, then
\begin{equation} \nn
     \|u(2 t_0)\|_{L^2_{x, \Vtheta}}^{2\omega}
\leq
    c_{3,2} \, t_0^{-\frac{\omega}{\omega-1}} 
\leq 
    c_3 \, (2 t_0)^{-\frac{\omega}{\omega-1}}\,,
\qquad \text{for all $0 < t_0 < \frac{T_1}{2}$} \,,
\end{equation}
where $c_3 = 2^{\frac{\omega}{\omega - 1}} c_{3,2}$ which is increasing in $m_0$. Hence,
\begin{equation} \nn
     \|u(t)\|_{L^2_{x, \Vtheta}} 
\leq 
    c_3 \, t^{-\frac{1}{p-1}}
\qquad \text{for all $0 < t < T_1$} \,,
\end{equation}
which proves the $L^1$ to $L^2$ regularization. 
\end{proof}


\subsection{Regularity - From $L^{2}$ to higher norms} Using the strong regularization lemma it is shown that a solution to the transport equation~\eqref{rte1} has higher smoothing in both spatial and angular variables for any positive time.  A boot-strapping argument is used after we show a basic $L^2$ estimate on the transport equation~\eqref{rte1}.
\begin{Lemma} \label{prop:basic-L-2}
Let $u \in L^2\big([t_0, t_1] \times \R^{d} \times \Ss^{d-1}\big)$ be a solution to equation~\eqref{rte1}.  Let $\WJ = \frac{\UJ}{\vint{v}^{d-1-2s}}$. Then $u, \WJ$ satisfy the estimate
\begin{align} \label{bound:frac-u-1}
   \norm{(-\Delta_v)^{s/2} \WJ}^2_{L^2((t_\ast, t_1) \times \R^{d} \times \R^{d-1})}
\leq 
   \left(\frac{1}{D_0(t_\ast - t_0)} + \frac{D_1}{D_0}\right) 
   \norm{u}_{L^2((t_0, t_1) \times\R^{d} \times \Ss^{d-1})}^2
\end{align}
and
\begin{align*}
    \sup_{t \in (t_\ast, t_1)}
      \norm{u}_{L^2(\R^{d} \times \Ss^{d-1})}^2(t)
\leq 
   \left(\frac{1}{t_\ast - t_0} + D_1\right) 
  \norm{u}_{L^2((t_0, t_1) \times\R^{d} \times \Ss^{d-1})}^2
\end{align*}
for any $t_\ast \in (t_0, t_1)$. Here $D_0, D_1$ are the two constants in~\eqref{MEE2}. Moreover, there exists $c_{0, d,s}$ which only depends on $d, s$ such that
\begin{align} \label{bound:frac-u}
     \big\|(-\Delta_{x})^{\frac{s_0}{2}} u \big\|_{L^{2}((t_\ast,t_1)\times\mathbb{R}^{d}\times\mathbb{S}^{d-1})}
&\leq 
    c_{0, d,s}\,\Big( \frac{1}{\sqrt{t_\ast - t_0}} + 1 \Big) \|u\|_{L^{2}((t_0,t_1)\times\mathbb{R}^{d}\times\mathbb{S}^{d-1})}  \,,
\end{align}
where $s_0 = \frac{s/4}{2s + 1}$.
\end{Lemma}
\begin{proof}
Let $t_\ast \in (t_0, t_1)$ be arbitrary. For any $t_{m,1} \in (t_0, t_\ast)$ and $t \in (t_\ast, t_1)$, we derive from~\eqref{MEE2} that  
\begin{align*}
\norm{u}_{L^2(\R^{d} \times \Ss^{d-1})}^2&(t)
   + D_0 \int_{t_\ast}^{t} \norm{(-\Delta_v)^{s/2} \WJ}^2_{L^2(\R^{d} \times \R^{d-1})} \dtau
\\
& 
\leq 
   \norm{u}_{L^2(\R^{d} \times \Ss^{d-1})}^2(t_m)
   + D_1 \int_{t_0}^{t_1} \norm{u}_{L^2(\R^{d} \times \Ss^{d-1})}^2 \dtau \,.
\end{align*}
Taking average over $t_{m,1} \in (t_0, t_\ast)$ gives
\begin{align*}
   \int_{t_\ast}^{t_1} \norm{(-\Delta_v)^{s/2} \WJ}^2_{L^2(\R^{d} \times \R^{d-1})} \dtau
\leq 
   \left(\frac{1}{D_0(t_\ast - t_0)} + \frac{D_1}{D_0}\right) 
   \int_{t_0}^{t_1} \norm{u}_{L^2(\R^{d} \times \Ss^{d-1})}^2 \dtau \,.
\end{align*}
Similarly, we have
\begin{align*}
    \sup_{t \in (t_\ast, t_1)}
      \norm{u}_{L^2(\R^{d} \times \Ss^{d-1})}^2(t)
\leq 
   \left(\frac{1}{t_\ast - t_0} + D_1 \right) 
   \int_{t_0}^{t_1} \norm{u}_{L^2(\R^{d} \times \Ss^{d-1})}^2 \dtau \,.
\end{align*}
Combining~\eqref{bound:frac-u-1} with Corollary~\ref{Cor:strong-reg-avg-time}, we then obtain~\eqref{bound:frac-u}.
\end{proof}
\begin{Proposition} \label{prop:infty-reg}
Suppose $u \in L^2([t_0, t_1] \times \R^d \times \Ss^{d-1})$ is a solution to the transport equation \eqref{rte1}. Let $\WJ = \frac{\UJ}{\vint{v}^{d-1-2s}}$. Assume that for some integer $N_0 \geq 1$,
\begin{align} \label{def:h}
    h(z) 
    = \frac{\tilde b(z)}{(1-z)^{\frac{d-1}{2}+s}}
    \in \mathcal{C}^{N_0}([-1, 1])\,.
\end{align}
Then for any $l \geq 0, 1 \leq k \leq \big[\frac{N_0}{s} \big] - 1$, and any $t_\ast \in (t_0, t_1)$, we have
\begin{align*}
    (-\Delta_x)^l u \in L^2((t_*, t_1) \times \R^{d} \times \Ss^{d-1}) \,,
\quad
    (-\Delta_v)^{\frac{k+1}{2} s} \WJ \in L^2((t_*, t_1) \times \R^{d} \times \R^{d-1}) \,.
\end{align*}
More specifically,
\begin{equation}
    \label{bound:high-order-deriv-u-w-J}
\begin{aligned}
    \norm{(-\Delta_x)^{l} u}_{L^2((t_\ast, t_1) \times \R^{d} \times \Ss^{d-1})}
\leq 
   c_5 \norm{u}_{L^2((t_0, t_1) \times \R^d \times \Ss^{d-1})} \,,
\\
  \norm{(-\Delta_v)^{\frac{k+1}{2}s} \WJ}_{L^2((t_*, t_1) \times \R^{d} \times \R^{d-1})}
\leq 
   c_6 \norm{u}_{L^2((t_0, t_1) \times \R^d \times \Ss^{d-1})} \,,
\end{aligned}
\end{equation}
where $c_5$ only depends on $l, d, s$, $\frac{1}{t_\ast - t_0}$ and $c_6$ only depends on $k,d,s$, $\frac{1}{t_\ast - t_0}$. In particular, they are independent of $t_1$.
\end{Proposition}
\begin{proof}
We first establish the regularity in $x$. To this end, fix $t_\ast \in (t_0, t_1)$ and apply the operator $(-\Delta_x)^{s_0/2}$ to \eqref{rte1}. The resulting equation is
\begin{align*}
    \del_t \left((-\Delta_x)^{s_0/2} u \right)
    + \theta \cdot &\Grad \left( (-\Delta_x)^{s_0/2} u \right)
    = \CalL((-\Delta_x)^{s_0/2} u) \,, 
\qquad t \in \big(\tfrac{t_0 + t_\ast}{2}, t_1\big) \,,
\end{align*}
where $(-\Delta_x)^{s_0/2} u \in L^2((\tfrac{t_0 + t_\ast}{2}, t_1) \times \R^d \times \Ss^{d-1})$ by~\eqref{bound:frac-u}.
Applying estimate \eqref{bound:frac-u} twice gives
\begin{align*}
     \big\|(-\Delta_{x})^{s_0} u \big\|_{L^{2}([t_{\ast},t_1)\times \R^d \times \Ss^{d-1})}
&\leq 
    c_4\,\Big( \frac{1}{\sqrt{t_\ast - t_0}} + 1 \Big) 
    \big\| (-\Delta_x)^{s_0/2} u \big\|_{L^{2} \left(\left(\tfrac{t_0 + t_\ast}{2},t_1 \right)\times\R^{d} \times \Ss^{d-1} \right)}  
\\
&\leq
    c_{4,1}\,\Big( \frac{1}{t_\ast - t_0} + 1 \Big) 
    \| u \|_{L^{2}([t_0,t_1)\times\R^{d} \times \Ss^{d-1})}  \,,
\end{align*}
where $c_{4,1}$ only depends on $d,s$.
Any higher order derivative in $x$ can then be derived by finitely many iterations.  Specifically, for any $l \geq 0$, we have
\begin{align*}
     \big\|(-\Delta_{x})^{l} u \big\|_{L^{2}([t_{\ast},t_1)\times \R^d \times \Ss^{d-1})}
&\leq
    c_5 \| u \|_{L^{2}((t_0,t_1)\times\R^{d} \times \Ss^{d-1})}  \,,
\end{align*}
where $c_5$ depends on $\frac{1}{t_\ast - t_0}, l, d, s$.\\
 
\noindent
We now show the regularization in $v$ given the aforementioned smoothing in $x$ by applying an induction argument.  Since higher order derivatives in $v$ will introduce remainder terms, we add in a forcing term to the original  equation solved by $\WJ$.  Specifically, we consider the equation
\begin{equation} 
     \del_t \WJ + \theta \cdot \Grad \WJ
     = - D_0 \vint{v}^{4s} (-\Delta_v)^s \WJ
        + \CalR_w(\WJ) 
        + \CalR_f\,, \label{eq:W-J}
\end{equation}
where $\CalR_w = \CalR_{w, 1} + \CalR_{w, 2}$ with
\begin{equation}\label{def:R-w-1}
\begin{aligned}
   \CalR_{w,1} (\WJ)
&  = D_0 \, \WJ \vpran{\vint{\cdot}^{d-1+2s} (-\Delta_v)^s \frac{1}{\vint{\cdot}^{d-1-2s}}}=c_{d,s,3} \, \WJ \,,
\\
   \CalR_{w, 2}(\WJ)
& = \frac{1}{\vint{\cdot}^{d-1-2s}}\,\big[\mathcal{I}_{h}(u)\big]_{\mathcal{J}}\,,
\end{aligned}
\end{equation}
where $c_{d,s,3} = c_{d,s}\,D_0$ and $h$ was defined in~\eqref{def:h}.  We assume here that the forcing term $\CalR_f$ satisfies the bound
\begin{align} \label{bound:R-f-H-1}
     \norm{\vpran{I + \nabla_v}
                \vpran{\frac{\CalR_f}{\vint{\cdot}^{2s}}}}_{L^2((t_\ast, t_1) \times \R^{d} \times \R^{d-1})}
\leq \tilde c_{6}  \norm{u}_{L^2((t_0, t_1) \times \R^d \times \Ss^{d-1})} \,,
\end{align}
with the coefficient $\tilde c_6$ only depending on $d, s, b, \frac{1}{t_\ast - t_0}$.  By Lemma~\ref{prop:basic-L-2}, we have
\begin{align*}
    \|(-\Delta_v)^{s/2} \WJ \|_{L^2((t_2, t_1) \times\R^d \times \R^{d-1})}
\leq c_{6,1} \|u\|_{L^2((t_0, t_1) \times \R^d \times \Ss^{d-1})} \,, 
\end{align*}
for any $t_2 \in (t_0, t_1)$ and with $c_{6,1} = \frac{1}{D_0(t_\ast - t_0)} + \frac{D_1}{D_0}$ which only depends on $d, s, \frac{1}{t_2-t_0}$.  Multiplying~\eqref{eq:W-J} by $(-\Delta_v)^s \WJ$ and integrating in $x, v$, we have 
\begin{equation} \label{ineq:W-J-1}
\begin{aligned}
    \frac{1}{2} \frac{\rm d}{\dt}
         \|(-\Delta_v)^{s/2} \WJ \|_{L^2_{x,v}}^2
&\leq -D_0 \int_{\R^d}\int_{\R^{d-1}}
           \vint{v}^{4s} \left| (-\Delta_v)^s \WJ \right|^2\dv\dx
\\
&\hspace{-1cm}
  + \int_{\R^d}\int_{\R^{d-1}} 
        \abs{\vint{v}^{2s}(-\Delta_v)^s \WJ} 
         \abs{\frac{\CalR_w (\WJ)}{\vint{v}^{2s}}} \dv\dx
\\
&\hspace{-1cm}
  + \int_{\R^d}\int_{\R^{d-1}} 
        \abs{\vint{v}^{2s}(-\Delta_v)^s \WJ} 
         \abs{\frac{\CalR_f}{\vint{v}^{2s}}} \dv\dx
\,.
\end{aligned}
\end{equation}
By Lemma~\ref{app:coercive} and Lemma~\ref{app:convex}, we have
\begin{align*}
    \norm{\frac{1}{\vint{\cdot}^{2s}}\CalR_{w,1}(\WJ)}_{L^2_{x,v}}
&=
   c_{d,s,3} \norm{\frac{\WJ}{\vint{\cdot}^{2s}}}_{L^2_{x,v}}
  = c_{d,s,3} \norm{u}_{L^2_{x,\theta}} \,,
\\[2pt]
    \norm{\frac{1}{\vint{\cdot}^{2s}}\CalR_{w,2}(\WJ)}_{L^2_{x,v}}
&\leq 
   c_{6,3} \norm{u}_{L^2_{x,\theta}} \,,
\end{align*}
where $c_{6,3}$ only depends on $\norm{h}_{L^1}$. For any 
$t_{m, 2} \in (t_0, t_2)$ and $t \in (t_2, t_1)$ such that $(-\Delta_v)^{s/2} \WJ (t_{m, 2}, \cdot, \cdot) \in L^2_{x,v}$ and $(-\Delta_v)^{s/2} \WJ (t, \cdot, \cdot) \in L^2_{x,v} $, 
apply Cauchy-Schwarz to the right-hand side of~\eqref{ineq:W-J-1} and integrate over $[t_{m,2}, t]$. We have
\begin{align*} 
& \quad \,
   \|(-\Delta_v)^{s/2} \WJ(t) \|_{L^2_{x,v}}^2
+ D_0 \!\!
   \int_{t_2}^t \int_{\R^d}\int_{\R^{d-1}}
           \vint{v}^{4s} \left| (-\Delta_v)^s \WJ \right|^2 (t, x, v) \dv\dx\dt
\\
&\hspace{3cm}\leq 
   \|(-\Delta_v)^{s/2} \WJ(t_{m,2}) \|_{L^2_{x,v}}^2
   + c_{6,4} \|u\|_{L^2((t_0, t_1) \times \R^d \times \Ss^{d-1})} \,,
\end{align*}
where $c_{6,4}$ only depends on $d,s,\norm{h}_{L^1}$. Taking average over $t_{m,2} \in (t_0, t_2)$ and applying~\eqref{bound:frac-u-1} then gives
\begin{align} \label{bound:W-J-s}
    \|\vint{\cdot}^{2s}(-\Delta_v)^{s} \WJ \|_{L^2((t_2, t_1) \times\R^d \times \R^{d-1})}
\leq c_{6,5} \|u\|_{L^2((t_0, t_1) \times \R^d \times \Ss^{d-1})} \,,
\end{align}
for any $t_2 \in (t_0, t_1)$ where $c_{6, 5}$ only depends on $\frac{1}{t_2-t_0}, d,s,\norm{h}_{L^1_\theta}$.
In general, suppose we have obtained the bound 
\begin{align} \label{assump:W-J-reg}
    \norm{(-\Delta_v)^{\frac{k+1}{2}s} \WJ}_{L^2((t_3, t_1) \times\R^d \times \R^{d-1})}
\leq c_{6,6} \|u\|_{L^2((t_0, t_1) \times \R^d \times \Ss^{d-1})} \,,
\end{align}
where $t_3 \in (t_0, t_1)$ is arbitrary, $k \geq 1$, $ks < 1$, and $c_{6,6}$ only depends on $\frac{1}{t_3-t_0}, k, d, s, \norm{h}_{L^1}$, and $\norm{h}_{\mathcal{C}^{N_0}([-1,1])}$. 
We want to show that for any $t_4 \in (t_0, t_1)$, there exists $c_8$ independent of $t_1$ such that 
\begin{align} \label{assump:W-J-reg-4}
    \norm{(-\Delta_v)^{\frac{k+2}{2}s} \WJ}_{L^2((t_4, t_1) \times\R^d \times \R^{d-1})}
\leq c_8 \|u\|_{L^2((t_0, t_1) \times \R^d \times \Ss^{d-1})} \,,
\end{align}
Note that by interpolation between~\eqref{assump:W-J-reg} and~\eqref{bound:frac-u-1}, we have
\begin{align} \label{assump:W-J-reg-1}
    \norm{(-\Delta_v)^{\frac{k}{2}s} \WJ}_{L^2((t_3, t_1) \times\R^d \times \R^{d-1})}
\leq c_{6,7} \|u\|_{L^2((t_0, t_1) \times \R^d \times \Ss^{d-1})} \,,
\end{align}
where $c_{6,7}$ only depends on $\frac{1}{t_3-t_0}, k, d, s, \norm{h}_{L^1}$, and $\norm{h}_{\mathcal{C}^{N_0}([-1,1])}$. In order to show further regularization in $v$, define the difference quotient 
\begin{align} \label{def:w-J-diff-quo}
     \WJD := \frac{\WJ(v+y) - \WJ(v)}{|y|^{\frac{d-1}{2} + \frac{k}{2}s}}
\qquad
    \text{for any $y \in \R^{d-1}\setminus\{0\}$} \,.
\end{align}
The induction assumption~\eqref{assump:W-J-reg} and~\eqref{assump:W-J-reg-1} imply that 
\begin{align} \label{bound:WJD-1}
    \norm{(I + (-\Delta_v)^{s/2})\WJD}_{L^2((t_3, t_1) \times\R^d \times \R^{d-1})}
\leq c_{6,8} \|u\|_{L^2((t_0, t_1) \times \R^d \times \Ss^{d-1})} \,.
\end{align}
where $c_{6,8}$ only depends on $\frac{1}{t_3-t_0}, k, d, s, \norm{h}_{L^1}$, and $\norm{h}_{\mathcal{C}^{N_0}([-1,1])}$. The equation for $\WJD$ has the form
\begin{equation}
   \label{eq:w-j-difference-quotient}
\begin{aligned} 
    \del_t \WJD
    \!+\! \theta(v \!+\! y) \cdot \!\Grad \WJD
&   =\! -D_0 \vint{v}^{4s} (-\Delta_v)^s \WJD
       \!-\! R_1 \!-\! R_2  
+ \! \tilde\CalR_{w, 1} \!+ \!\tilde\CalR_{w, 2}
+ \!\tilde\CalR_f ,
\end{aligned}
\end{equation}
where 
\begin{align*}
    R_1 = \frac{\theta(v+y) - \theta(v)}{|y|^{\frac{d-1}{2} + \frac{k}{2}s}}
         \cdot \Grad \WJ(v) \,,
\qquad
    R_2  = D_0 \frac{ \vint{v+y}^{4s} -  \vint{v}^{4s}}{|y|^{\frac{d-1}{2} + \frac{k}{2}s}} (-\Delta_v)^s \WJ(v+y) \,,
\end{align*}
\begin{align*}
   \tilde\CalR_f 
   = \frac{\CalR_f (v + y) - \CalR_f (v)}{|y|^{\frac{d-1}{2} + \frac{k}{2}s}} \,,
\qquad
    \tilde\CalR_{w, j} 
= \frac{\CalR_{w, k}(\WJ)(v+y) - \CalR_{w, k}(\WJ)(v)}
          {|y|^{\frac{d-1}{2} + \frac{k}{2}s}}  \,,
\quad j = 1, 2 \,,
\end{align*}
Note that ~\eqref{bound:R-f-H-1} implies
\begin{align} \label{bound:R-f-H-ks}
     \norm{\frac{1}{\vint{\cdot}^{2s}} \CalR_f }_{L^2((t_{\star},t_{1})\times\mathbb{R}^{d}; H^{\frac{ks}{2}}(\R^{d-1}))}
\leq \tilde c_{6}  \norm{u}_{L^2((t_0, t_1) \times \R^d \times \Ss^{d-1})} \,,
\end{align}
for any $0 \leq \frac{ks}{2} \leq 1$. This in particular implies
\begin{align}
\int_{t_\ast}^{t_1} \int_{\R^d} \int_{\{|y| \leq 1\}} \int_{\R^{d-1}}
     \frac{1}{\vint{v}^{4s}}&\frac{\abs{\CalR_f(v+y) - \CalR_f(v)}^2}{|y|^{d-1 + ks}} 
  \dv\dy\dx\dt \nn\\
&\leq \tilde c_{6,1}  \norm{u}_{L^2((t_0, t_1) \times \R^d \times \Ss^{d-1})} \,,
           \label{bound:R-f-H-ks-1}
\end{align}
where $\tilde c_{6,1}$ only depends on $\tilde c_6$ and $d,s$. Note that although we have assumed that $\CalR_f$ satisfies~\eqref{bound:R-f-H-1}, the real bound that we need for $\CalR_f$ is~\eqref{bound:R-f-H-ks-1}.  Our first step is to prove that 
\begin{align*}
   \int_{t_4}^{t_1} \int_{\{|y| \leq 1\}} \int_{\R^d} \int_{\R^{d-1}} 
      \left|\vint{v}^{2s} (-\Delta_v)^s \WJD \right|^2 \dv\dx\dy
< \infty \,,
\end{align*}
for any $t_4 \in (t_0, t_1)$.  To this end, multiply~\eqref{eq:w-j-difference-quotient} by $(-\Delta_v)^s\WJD$ and integrate in $(x, v)$ to obtain
\begin{align}
\frac{1}{2} \frac{\rm d}{\dt}&
        \norm{(-\Delta_v)^{\frac{s}{2}}\WJD}_{L^2_{x,v}}^2  \leq   -D_0 \norm{\vint{v}^{2s}(-\Delta_v)^{s} \WJD}_{L^2_{x,v}}^2  \nn\\
& + \iint_{\R^d \times \R^{d-1}} 
           \abs{R_1} \abs{(-\Delta_v)^s\WJD} \dx\dv + \iint_{\R^d \times \R^{d-1}} 
          \abs{R_2} \abs{(-\Delta_v)^s\WJD} \dx\dv \nn\\
& \hspace{-.5cm}  + \iint_{\R^d \times \R^{d-1}} 
           \vpran{\abs{\tilde\CalR_{w,1}} + \abs{\tilde\CalR_{w,2}}
                      + \abs{\tilde\CalR_f}} \abs{(-\Delta_v)^s\WJD} \dx\dv\,. \label{est:energy-w-J-delta}
\end{align}
Now we estimate the terms involving $R_1, R_2, \tilde\CalR_{w, 1}, \tilde\CalR_{w, 2}$.  By Cauchy-Schwarz, 
\begin{align*}
 \iint_{\R^d \times \R^{d-1}} &
        \abs{R_1} \abs{(-\Delta_v)^s\WJD} \dx\dv\\
&\hspace{-.5cm}\leq
      \iint_{\R^d \times \R^{d-1}}
         \frac{\abs{\theta(v+y) - \theta(v)}}{|y|^{\frac{d-1}{2} + \frac{k}{2}s}}
         \abs{\Grad \frac{\WJ}{\vint{v}^{2s}}}
         \abs{\vint{v}^{2s}(-\Delta_v)^s\WJD} \dx\dv
\\
&\hspace{1cm} \leq 
     \frac{c_{6,9}}{\abs{y}^{\frac{d-1}{2} + \frac{k}{2}s - 1}}
     \norm{\Grad \frac{\WJ}{\vint{\cdot}^{2s}}}_{L^2_{x,v}}
     \norm{\vint{\cdot}^{2s}(-\Delta_v)^s\WJD}_{L^2_{x,v}} \,,
\end{align*}
and for $\{|y| \leq 1\}$,
\begin{align*}
& \quad \,
     \iint_{\R^d \times \R^{d-1}} 
        \frac{\abs{R_2} \abs{\WJD}}{\vint{v}^{4s}} \dx\dv
\\
&\leq
    D_0 \iint_{\R^d \times \R^{d-1}}
            \frac{ \abs{\vint{v+y}^{4s} -  \vint{v}^{4s}}}
                   {|y|^{\frac{d-1}{2} + \frac{k}{2}s} \vint{v}^{4s}}
            \abs{\vint{v}^{2s}(-\Delta_v)^s \WJ(v+y)} 
            \abs{\vint{v}^{2s}(-\Delta_v)^s\WJD} \dx\dv
\\
&\hspace{1cm}\leq
     \frac{c_{6, 10}}{\abs{y}^{\frac{d-1}{2} + \frac{k}{2}s - 1}}
     \norm{\vint{\cdot}^{2s}(-\Delta_v)^s \WJ}_{L^2_{x,v}}
     \norm{\vint{\cdot}^{2s}(-\Delta_v)^s\WJD}_{L^2_{x,v}}^2 \,.
\end{align*}
where $c_{6,9}, c_{6, 10}$ only depend on $d,s$.  Integrating the above terms over $\{|y| \leq 1\}$ gives
\begin{equation} \label{est:R-1-R-2}
\begin{aligned} 
   \int_{\{|y| \leq 1\}} &\iint_{\R^d \times \R^{d-1}}
	        \vpran{\abs{R_1} + \abs{R_2}} \abs{(-\Delta_v)^s\WJD} \dx\dv\dy
\\
&\quad \leq
   c_{6,11} \vpran{\norm{\Grad u}_{L^2_{x, \theta}}^2
                     + \norm{\vint{\cdot}^{2s}(-\Delta_v)^s \WJ}_{L^2_{x,v}}^2}
\\
&\hspace{2cm}\,
   + \frac{D_0}{8}
        \vpran{\int_{\{|y| \leq 1\}}\norm{\vint{\cdot}^{2s}(-\Delta_v)^s\WJD}_{L^2_{x,v}}^2 \dy}^{\frac{1}{2}},
\end{aligned}
\end{equation}
where $c_{6, 11}$ only depends on $d, s$.  We estimate the term involving $\tilde\CalR_{w,1}$ using \eqref{def:R-w-1}
\begin{equation} \label{bound:R-w-1}
\begin{aligned}
    \int_{\{|y| \leq 1\}}&\iint_{\R^d \times \R^{d-1}} 
           \abs{\tilde\CalR_{w,1}} \abs{(-\Delta_v)^s\WJD} \dx\dv\dy
\\
& \leq  
   c_{6,17} \norm{\WJD}_{L^2_{x,v}}^2
   + \frac{D_0}{8} 
        \int_{\{|y| \leq 1\}} \norm{\vint{v}^{2s} (-\Delta_v)^s \WJD}_{L^2_{x,v}}^2 \dy\,,
\end{aligned}
\end{equation}
where $c_{6,17}$ only depends on $d,s$.  The term involving $\tilde\CalR_{w,2}$ follows rewriting it as
\begin{align*}
   \tilde\CalR_{w,2} = c_{7} \, \WJ - \Big[\frac{1}{\vint{\cdot}^{d-1-2s}}
      \int_{\Ss^{d-1}} u(\theta') h(\theta \cdot \theta') \dtheta'\Big]_{\mathcal{J}} \,,
\end{align*}
where the constant $c_7 := \int_{\Ss^{d-1}} h(\theta \cdot \theta') \dtheta'$.  Hence,
\begin{align*}
     \tilde\CalR_{w,2}(v)
&   = c_7\, \WJD(v) 
      -  \frac{\vint{v+y}^{-(d-1-2s)} - \vint{v}^{-(d-1-2s)}}{|y|^{\frac{d-1}{2} + \frac{k}{2}s}}
         \int_{\Ss^{d-1}} u(\theta') h(\theta(v) \cdot \theta') \dtheta'
\\
&\hspace{-1cm} 
      - \frac{1}{\vint{v}^{d-1-2s}}
        \frac{\int_{\Ss^{d-1}} u(\theta') \, h(\theta(v+y) \cdot \theta') \dtheta' 
                - \int_{\Ss^{d-1}} u(\theta') \, h(\theta(v) \cdot \theta') \dtheta'}
                {|y|^{\frac{d-1}{2} + \frac{k}{2}s}}
\\
& =:    
      c_7\, \WJD  -  \tilde\CalR_{w, 2,2} - \tilde\CalR_{w,2,3}   \,.
\end{align*}
Thus,
\begin{align*}
   \abs{\tilde\CalR_{w,2,2}(v)} 
= \abs{\frac{\vint{v+y}^{-(d-1-2s)} - \vint{v}^{-(d-1-2s)}}{|y|^{\frac{d-1}{2} + \frac{k}{2}s}}}
\leq
  \frac{c_{7, 1}}{\vint{v}^{d-2s}} \frac{1}{|y|^{\frac{d-1}{2} + \frac{k}{2}s - 1}} \,,
\qquad |y| \leq 1 \,,
\end{align*}
where $c_{7,1}$ only depends on $d,s$. Therefore the term involving $\tilde\CalR_{w,2,2}$ has the bound
\begin{align*}
\int_{\{|y| \leq 1\}} &\iint_{\R^d \times \R^{d-1}}
       \abs{\tilde\CalR_{w,2,2}} \abs{(-\Delta_v)^{s}\WJD} \dx\dv\dy
\\
& \leq 
   c_{7, 2}  \vpran{\int_{|y| \leq 1} 
                   \frac{\norm{\vint{\cdot}^{2s} (-\Delta_v)^s \WJD}_{L^2_{x,v}}}{|y|^{\frac{d-1}{2} + \frac{k}{2}s - 1}}
                    \dy}
       \norm{\int_{\Ss^{d-1}} u(\theta') h(\theta \cdot \theta') \dtheta'}_{L^2_{x, \theta}}
\\
&\hspace{-1cm} \leq 
   c_{7,3}\norm{h}_{L^1(\Ss^{d-1})}  
    \norm{u}_{L^2_{x,\theta}}
     \vpran{\int_{\{|y|\leq 1\}} \norm{\vint{v}^{2s} (-\Delta_v)^s \WJD}_{L^2_{x,v}}^2 \dy}^{\frac{1}{2}}
 \,,
\end{align*}
where $c_{7,2}, c_{7,3}$ only depend on $d,s$.
The bound for $\tilde\CalR_{w,2,3}$ is
\begin{align*}
\abs{\tilde\CalR_{w,2,3}(v)}
&\leq \frac{1}{\vint{v}^{d-1-2s}}
    \int_{\Ss^{d-1}} u(\theta') \, \frac{\abs{h(\theta(v+y) \cdot \theta') 
                - h(\theta(v) \cdot \theta')}} 
                {|y|^{\frac{d-1}{2} + \frac{k}{2}s}}\dtheta'
\\
&\leq
   \sqrt{\abs{\Ss^{d-1}}}\vpran{ \sup_{z \in [-1,1]} \abs{h'(z)}} 
   \frac{1}{\vint{v}^{d-1-2s}}
   \frac{1}{|y|^{\frac{d-1}{2} + \frac{k}{2}s-1}} 
   \norm{u}_{L^2_{x,\theta}} \,.
\end{align*}
Therefore,
\begin{align*}
\int_{\{|y| \leq 1\}} &\iint_{\R^d \times \R^{d-1}}
       \abs{\tilde\CalR_{w,2,3}} \abs{(-\Delta_v)^s\WJD} \dx\dv\dy   
\\
& \leq c_{7, 4} \norm{h'}_{L^\infty(-1,1)} \norm{u}_{L^2_{x, \theta}} 
            \vpran{\int_{\{|y| \leq 1\}} \norm{\vint{\cdot}^{2s} (-\Delta_v)^s\WJD}_{L^2_{x,v}}^2 \dy}^{\frac{1}{2}},
\end{align*}
where $c_{7,4}$ only depends on $d,s$. Hence,
\begin{equation} \label{bound:R-w-2}
\begin{aligned}
& \hspace{2cm} \,
    \int_{\{|y| \leq 1\}}\iint_{\R^d \times \R^{d-1}} 
           \abs{\tilde\CalR_{w,2}} \abs{(-\Delta_v)^s\WJD} \dx\dv\dy
\\
& \leq  
   c_{7, 5} \vpran{\norm{\WJD}_{L^2_{x,v}}^2
                     + \norm{u}_{L^2_{x, \theta}}^2}
   + \frac{D_0}{8} 
     \vpran{\int_{\{|y| \leq 1\}} \norm{\vint{\cdot}^{2s} (-\Delta_v)^s \WJD}_{L^2_{x,v}}^2 \dy}^{\frac{1}{2}} \,,
\end{aligned}
\end{equation}
where $c_{7, 5}$ only depends on $d,s, \norm{h}_{L^1(-1,1)},$ and $\norm{h}_{\mathcal{C}^1([-1,1])}$.  Furthermore, the estimate for the forcing term is 
\begin{align*}
& \hspace{2cm}\,
   \int_{\{|y| \leq 1\}} \iint_{\R^d \times \R^{d-1}}
       \abs{\tilde\CalR_f} \abs{(-\Delta_v)^s\WJD} \dx\dv\dy
\\
&\leq 
   \frac{4}{D_0} \norm{\frac{\tilde\CalR_f}{\vint{\cdot}^{2s}}}_{L^2(\R^d \times \R^{d-1})}^2 
   + \frac{D_0}{8} 
     \vpran{\int_{\{|y| \leq 1\}} \norm{\vint{\cdot}^{2s} (-\Delta_v)^s \WJD}_{L^2_{x,v}}^2 \dy}^{\frac{1}{2}} \,,
\end{align*}
Using estimate \eqref{bound:WJD-1},
\begin{align*}
  (-\Delta_v)^{s/2}\WJD(t, \cdot, \cdot) \in L^2_{x,v} \,, 
\qquad
   t \in (t_3, t_1) \,\, a.e.
\end{align*}
Let $t_4 \in (t_3, t_1)$ be arbitrary. For any $t_{m, 3} \in (t_3, t_4)$ and $t \in (t_4, t_1)$ such that 
\begin{align*}
   (-\Delta_v)^{s/2} \WJD (t_{m, 3}, \cdot, \cdot) \in L^2_{x,v} \,,
\qquad
   (-\Delta_v)^{s/2} \WJD (t, \cdot, \cdot) \in L^2_{x,v} \,,
\end{align*} 
we integrate~\eqref{est:energy-w-J-delta} over $(t_{m,3}, t)$ and apply the estimates~\eqref{bound:W-J-s}, \eqref{bound:R-f-H-ks-1}, \eqref{est:R-1-R-2}, \eqref{bound:R-w-1}, and \eqref{bound:R-w-2}. Then 
\begin{align}
\frac{D_0}{2} \int_{t_4}^{t_1}& \int_{\{|y| \leq 1\}} 
         \norm{\vint{\cdot}^{2s} (-\Delta_v)^s \WJD}_{L^2_{x,v}}^2 (y) \dy \dt \nn
\\
&\leq
   c_{7,6} \vpran{\norm{\WJD}_{L^2((t_3, t_1) \times \R^d \times \R^{d-1})}^2
   + \norm{(u, \Grad u)}_{L^2((t_3, t_1) \times \R^d \times \Ss^{d-1})}^2} \nn
\\
& \hspace{-1cm}
   + c_{7,6} \norm{\vint{\cdot}^{2s}(-\Delta_v)^s \WJ}_{L^2((t_3, t_1) \times \R^d \times \R^{d-1})}^2
   + \int_{\{|y| \leq 1\}}\norm{(-\Delta_v)^{s/2} \WJD}_{L^2_{x,v}}^2 (t_{m,3}, y) \dy \nn \,,
\\
&\hspace{1cm}\leq
  c_{7, 7} \norm{u}_{L^2((t_0, t_1) \times \R^d \times \Ss^{d-1})}^2
  + c_{7,8} \norm{(-\Delta_v)^{\frac{k+1}{2} s} \WJ (t_{m,3})}_{L^2_{x,v}}^2 \,,
 \nn
\end{align}
where $c_{7,6}, c_{7, 7}, c_{7,8}$ only depend on $d, s, \norm{h}_{L^1(-1,1)},$ and $\norm{h}_{\mathcal{C}^1([-1,1])}$. Taking average over $t_{m,3} \in (t_0, t_4)$ and using~\eqref{bound:WJD-1}, we obtain that 
\begin{align} \label{est:W-J-D-1}
   \int_{t_4}^{t_1} \int_{\{|y| \leq 1\}} 
         \norm{\vint{\cdot}^{2s} (-\Delta_v)^s \WJD}_{L^2_{x,v}}^2 (y) \dy \dt
\leq
  c_{7, 9} \norm{u}_{L^2((t_0, t_1) \times \R^d \times \Ss^{d-1})}^2 \,,
\end{align}
where $c_{7,9}$ only depend on $d, s, \norm{h}_{L^1(-1,1)}$, $\norm{h}_{\mathcal{C}^1([-1,1])}$, and $\frac{1}{t_4 - t_0}$.  In addition to this,
\begin{equation} 
   \label{est:W-J-D-2}
\begin{aligned}
& \hspace{2cm}
   \int_{t_4}^{t_1} \int_{\{|y| \geq 1\}} 
         \norm{(-\Delta_v)^s \WJD}_{L^2_{x,v}}^2 (y) \dy\dt
\\
&= \int_{t_4}^{t_1} \iint_{\R^d \times \R^{d-1}} \int_{\{|y| \geq 1\}} 
         \frac{\abs{(-\Delta_v)^s \WJ(v+y) - (-\Delta_v)^s \WJ(v)}^2}
                {\abs{y}^{d-1 + ks}} \dy\dv\dx\dt
\\
&\hspace{2cm}\leq 
    c_{7, 10} \int_{t_4}^{t_1} \iint_{\R^d \times \R^{d-1}}
              \abs{ (-\Delta_v)^s \WJ}^2 \dv\dx\dt \,.
\end{aligned}
\end{equation}
where $c_{7, 10}$ only depends on $d,s$. Combining~\eqref{bound:W-J-s}, ~\eqref{est:W-J-D-1}, and~\eqref{est:W-J-D-2}, we have
\begin{align*}
     \int_{t_4}^{t_1}  
        \norm{(-\Delta_v)^{\frac{k+2}{2}s}\WJ(t)}_{L^2_{x,v}}^2 \dt
\leq
    c_{7, 11} \norm{u}_{L^2((t_0, t_1) \times \R^d \times \Ss^{d-1})}^2\,.
\end{align*}
where $t_4 \in (t_0, t_1)$ is arbitrary, and $c_{7, 11}$ only depends on $d, s, \norm{h}_{L^1(-1,1)}$, $\norm{h}_{\mathcal{C}^1([-1,1])}$, and $\frac{1}{t_4 - t_0}$.  We thereby finish the proof for the induction~\eqref{assump:W-J-reg-4} for $k \geq 1$ and $\frac{k+1}{2} s < 1$.  Furthermore,  the above argument applied to $\Grad u$ gives 
\begin{align} \label{bound:w-J-x-v}
     \norm{\nabla_v \Grad \WJ}_{L^2((t_*, t_1) \times \R^{d} \times \R^{d-1})}
\leq 
   c_{7, 12} \norm{u}_{L^2((t_0, t_1) \times \R^d \times \Ss^{d-1})} \,,
\end{align}
where $c_{7, 12}$ only depends on $d, s$, $\norm{h}_{\mathcal{C}^1([-1,1])}$, and $\frac{1}{t_\ast - t_0}$.\\

\noindent
If $k$ is sufficiently large such that $\frac{k+1}{2}s \geq 1$, we can apply $\nabla_v^{\big[\frac{k+1}{2}s \big]}$ to equation~\eqref{eq:W-J} first and repeat the above procedures for the fractional derivatives. Specifically, suppose we have shown that for some integer $\mathcal{M} \geq 1$ and any $m \in \N^{d-1}$ satisfying $1 \leq |m| \leq \mathcal{M}$ and any $t_\ast \in (t_0, t_1)$,
\begin{align}
   & \int_{t_\ast}^{t_1}
        \norm{\nabla_v^m \WJ(t)}_{L^2(\R^{d-1}_v; H^1(\R^d_x))}^2 \dt
\leq
    c_{7, 13} \norm{u}_{L^2((t_0, t_1) \times \R^d \times \Ss^{d-1})}^2\,,
    \label{bound:w-J-m-1}
\\
    \int_{t_\ast}^{t_1} \int_{\{|y| \leq 1\} }
       & \norm{\vint{\cdot}^{2s} (-\Delta_v)^{s} \nabla_v^{\hat m} w_{\CalJ, l}(t)}_{L^2_{x,v}}^2 \dt
\leq
    c_{7, 13} \norm{u}_{L^2((t_0, t_1) \times \R^d \times \Ss^{d-1})}^2\,,
    \label{bound:w-J-m-2}
\end{align}
where $|\hat m| = |m| - 1$, $c_{7, 13}$ depends on $d,s,\frac{1}{t_\ast - t_0}$ and $\norm{h}_{C^{\mathcal{M}}(-1, 1)}$, and
\begin{align*}
     w_{\CalJ, l} = \frac{\WJ(v+y) - \WJ(v)}{|y|^{\frac{d-1}{2} + \frac{l}{2}s}}
\qquad
    \text{for any $y \in \R^{d-1}\setminus\{0\}$} \,.
\end{align*}
for any $l \geq 1$ and $l s < 1$. Note that~\eqref{bound:w-J-m-2} indeed holds for $\mathcal{M} = 1$ as shown in~\eqref{est:W-J-D-1}. Apply $\nabla_v^m$ to~\eqref{rte1} to obtain
\begin{align} 
     \del_t (\nabla_v^m \WJ) + &\theta(v) \cdot \Grad (\nabla_v^m \WJ) \nn
\\
&    = - D_0 \vint{v}^{4s} (-\Delta_v)^s (\nabla_v^m \WJ)
        + \CalR_{w,1}(\nabla_v^m \WJ) 
        + \sum_{j=1}^4 R_{f, j} \,, \label{eq:W-J-deriv-m}
\end{align}
where $\CalR_{w,1}$ is defined in~\eqref{def:R-w-1} and the remainder terms $R_{f,j}$'s are 
\begin{align*}
    R_{f, 1} 
    &= \sum_{\substack{|m_1| + |m_2| = |m|, \\ |m_2| \leq |m|-1}} 
          c_{m_1, m_2} \nabla_v^{m_1} \theta (v) \cdot \Grad \nabla_v^{m_2} \WJ \,,
\\
    R_{f, 2} 
   & = -D_0 \sum_{\substack{|m_1| + |m_2| = |m|, \\ |m_2| \leq |m|-1}} 
          c_{m_1, m_2} \nabla_v^{m_1} \vpran{\vint{v}^{4s}} 
          (-\Delta_v)^s \nabla_v^{m_2} \WJ \,,
\\
    R_{f,3} 
   & = \norm{h}_{L^1} \nabla_v^m \WJ \,,
\qquad
    R_{f,4}
    = \int_{\Ss^{d-1}} u(\theta') 
             \nabla_v^m \vpran{\frac{h(\theta \cdot \theta')}{\vint{v}^{d-1-2s}}} \dtheta' \,.
\end{align*}
Recall that $h$ is defined in~\eqref{def:h}. Thus, if we can show that each $R_{f, j}$ ($1\leq j \leq 4$) satisfies the same bound as $R_f$ in~\eqref{bound:R-f-H-ks-1}, 
then we can derive by the previous argument that for any $|m'| = \mathcal{M} + 1$ and some $c_{7, 14}$ depending only on $d,s,\frac{1}{t_\ast-t_0}$, and $\norm{h}_{\mathcal{C}^{\mathcal{M}+1}}$,
\begin{equation} \label{bound:w-J-induction}
\begin{aligned}
   & \int_{t_\ast}^{t_1}
        \norm{\nabla_v^{m'} \WJ(t)}_{L^2(\R^{d-1}_{v}; H^1(\R^{d}_{x}))}^2 \dt
\leq
    c_{7, 14} \norm{u}_{L^2((t_0, t_1) \times \R^d \times \Ss^{d-1})}^2\,,
\\
    \int_{t_\ast}^{t_1} \int_{\{|y| \leq 1\} }
       & \norm{\vint{\cdot}^{2s} (-\Delta_v)^{|m'|-1+s}  \WJD(t)}_{L^2_{x,v}}^2 \dt
\leq
    c_{7, 14} \norm{u}_{L^2((t_0, t_1) \times \R^d \times \Ss^{d-1})}^2\,,
\end{aligned}
\end{equation}
which then concludes the induction proof, and hence, we prove the estimates~\eqref{bound:high-order-deriv-u-w-J}.  Note that $R_{f,3}$ can be obsorbed into $\CalR_{w,1}(\nabla_v^m \WJ)$ in~\eqref{eq:W-J-deriv-m}. So we only need to check $R_{f,j}$ for $j=1,2,4$.

Let us show first that $R_{f,1}$ and $R_{f,4}$ satisfy the bound~\eqref{bound:R-f-H-ks-1} with $R_f$ replaced by these $R_{f, j}$'s. Indeed, 
\begin{align*}
    \int_{t_\ast}^{t_1} &\norm{\nabla_v R_{f, 1}}_{L^2_{x,v}}^2 \dt \\
&\leq 
    c_{7, 15} 
      \sum_{|m|=1} ^{\mathcal{M}}\int_{t_\ast}^{t_1}
           \norm{\Grad (-\Delta_v)^{\frac{|m|}{2}} \WJ}_{L^2_{x,v}}^2 \dt
\leq c_{7, 16} \norm{u}_{L^2{((t_0, t_1) \times \R^d \times \R^{d-1}})}^2 \,,
\end{align*}
where $c_{7, 16}$ only depends on $d, c_{7, 13}$ hence $d,s, \frac{1}{t_\ast - t_0}$, and $\norm{h}_{\mathcal{C}^{\mathcal{M}}(-1, 1)}$.  Hence~\eqref{bound:R-f-H-ks-1} is satisfied by interpolation.  Next, by the assumption of $h$ in~\eqref{def:h}, we have that for any $|m'| = \mathcal{M} + 1$,
\begin{align*}
    \abs{\nabla_v^{m'} \vpran{\frac{h(\theta \cdot \theta')}{\vint{v}^{d-1-2s}}}}
\leq 
    c_{m'} \frac{\norm{h}_{\mathcal{C}^{\mathcal{M} + 1}}}{\vint{v}^{d-1-2s}} \,.
\end{align*}
Therefore, 
\begin{align*}
   \int_{t_\ast}^{t_1} 
     \norm{\frac{1}{\vint{\cdot}^{2s}}\nabla_v R_{f,4}}_{L^2_{x,v}}^2 \dt
\leq 
   c_{m'} \norm{h}_{\mathcal{C}^{\mathcal{M} + 1}}^2 
   \norm{u}_{L^2{((t_0, t_1) \times \R^d \times \R^{d-1}})}^2 \,,
\end{align*}
where $c_{m'}$ only depends on $\mathcal{M}$. In addition, we have
\begin{align*}
    \int_{t_\ast}^{t_1} \norm{\frac{R_{f, j}}{\vint{\cdot}^{2s}}}_{L^2_{x,v}}^2 \dt
\leq
    c_{7, 19} \norm{u}_{L^2{((t_0, t_1) \times \R^d \times \R^{d-1}})}^2\,,\;\; \text{for}\;\;j=1,4\,.
\end{align*}
Hence by interpolation, the remainder terms $R_{f, 1}$ and $R_{f,4}$ satisfy~\eqref{bound:R-f-H-ks-1}. 
Finally, let 
\begin{align*}
    \tilde\CalR_{f,2} 
= \frac{\CalR_{f,2}(v+y) - \CalR_{f,2}(v)}
          {|y|^{\frac{d-1}{2} + \frac{k}{2}s}}  \,.
\end{align*}
The leading order term in $\tilde\CalR_{f,2}$ are bounded as
\begin{align*}
\int_{t_\ast}^{t_1} &\int_{\{|y| \leq 1\}}
      \norm{\frac{1}{\vint{\cdot}^{2s}} \nabla_v^{m_1}\vint{\cdot}^{4s}(-\Delta_v)^{s} \nabla_v^{m_2} \WJD(t)}_{L^2_{x,v}}^2 \dt
\\
&\hspace{-.3cm}\leq
    4 \int_{t_\ast}^{t_1} \int_{\{|y| \leq 1\}}
       \norm{\vint{\cdot}^{2s} (-\Delta_v)^{s} \nabla_v^{m_2} \WJD(t)}_{L^2_{x,v}}^2 \dt
\leq   
   4 \, c_{7, 13} \norm{u}_{L^2((t_0, t_1) \times \R^d \times \Ss^{d-1})}^2 \,,
\end{align*}
for any $|m_1| = 1$ and $|m_2| = |m| - 1$. The rest of the terms in $\tilde\CalR_{f, 2}$ satisfy that
\begin{align*}
    \int_{t_\ast}^{t_1} &\int_{|y| \leq 1 }
      \norm{\frac{1}{\vint{v}^{2s}} 
                \frac{\abs{\nabla_v^{m_1}\vint{v+y}^{4s}
                        - \nabla_v^{m_1}\vint{v}^{4s}}}{|y|^{\frac{d-1}{2}+\frac{k}{2}s}} 
                (-\Delta_v)^{s} \nabla_v^{m_2} \WJ(t)}_{L^2_{x,v}}^2 \dy\dt
\\
& \leq 
    \int_{t_\ast}^{t_1} \iint_{\R^d \times \R^{d-1}} \int_{\{|y| \leq 1\} }
      \frac{\vint{v}^{4s}}{|y|^{d-1+ks-1}} 
                \abs{(-\Delta_v)^{s} \nabla_v^{m_2} \WJ(t)}^2 \dy\dv\dx\dt
\\
&\hspace{-.3cm} \leq c_{7,21}
    \int_{t_\ast}^{t_1} \big\|\vint{\cdot}^{2s}(-\Delta_v)^{s} \nabla_v^{m_2} \WJ(t)\big\|^{2}_{L^{2}_{x,v}}\dt
\leq c_{7,21} c_{7, 13}
            \norm{u}_{L^2((t_0, t_1) \times \R^d \times \Ss^{d-1})}^2
\end{align*}
for any $|m_1| \geq 2$ and $|m_2| = |m| - |m_1|$.  Here $c_{7, 21}$ only depends on $d,s$.  Hence $\CalR_{f,2}$ also acts similarly as $R_f$ in the previous proof.  In conclusion, all the remainder terms does not affect the energy bound and similar estimates as for $\mathcal{M} = 1$ apply to~\eqref{eq:W-J-deriv-m} which give the desired higher order bounds~\eqref{bound:w-J-induction} in the induction argument.  This concludes the proof.\end{proof}
%
\begin{Proposition} \label{prop:u-time-reg}
Suppose $u \in L^2((t_0, t_1) \times \R^d \times \Ss^{d-1})$ is a solution to~\eqref{rte1}. Then for any $j_1, j_2 \in \N$ and any $t_\ast \in (t_0, t_1)$  there exists $c_{j_1,j_2}$ such that
\begin{align*} 
     \norm{(-\Delta_x)^{j_1} \del_t^{j_2} \frac{\WJ}{\vint{v}^{2s}}}_{L^2((t_{\ast}, t_1) \times \R^d \times \R^{d-1})}
\leq 
     c_{j_1,j_2} \norm{u}_{L^2((t_0, t_1) \times \R^d \times \Ss^{d-1})} \,,
\end{align*}
where $c_{j_1,j_2}$ only depends on $d, s, b, \frac{1}{t_\ast - t_0}$. In particular, $c_{j_1, j_2}$ is independent of $t_1$.
\end{Proposition}
\begin{proof}
We will apply an induction argument. First, by Proposition~\ref{prop:infty-reg} and~\eqref{bound:W-J-s},
\begin{align} \label{est:W-J-6}
     \norm{\del_t u}_{L^2((t_\ast, t_1) \times \R^d \times \Ss^{d-1})}
     = \norm{\del_t \frac{\WJ}{\vint{v}^{2s}}}_{L^2((t_\ast, t_1) \times \R^d \times \R^{d-1})}
\leq 
     c_{8,1} \norm{u}_{L^2((t_0, t_1) \times \R^d \times \Ss^{d-1})} \,,
\end{align}
where $c_{8,1}$ only depends on $d,s, \frac{1}{t_\ast - t_0}$, and the kernel $b$. In general, suppose 
\begin{align*} 
     \norm{\del_t^j u}_{L^2((t_\ast, t_1) \times \R^d \times \Ss^{d-1})}
     = \norm{\del_t^j \frac{\WJ}{\vint{v}^{2s}}}_{L^2((t_\ast, t_1) \times \R^d \times \R^{d-1})}
\leq 
     c_{8,2} \norm{u}_{L^2((t_0, t_1) \times \R^d \times \Ss^{d-1})} \,,
\end{align*}
for some $j \geq 1$ and $c_{8,2}$ depending on $d,s,j, \frac{1}{t_\ast - t_0}$, and the kernel $b$. Then $\del_t^j  u(t, \cdot, \cdot) \in L^2_{x,v}$ for $t \in (t_\ast, t_1)$ a.e.  Moreover, $\del_t^j  u$ satisfies the transport equation
\begin{align*} 
    \del_t \vpran{\del_t^j u}
    + \theta\cdot \Grad \vpran{\del_t^j u}
  = \CalL\vpran{\del_t^j u}
\end{align*}
with the initial data in $L^2_{x,\theta}$. Hence~\eqref{est:W-J-6} applies and gives that for any $t_{\ast\ast} \in (t_\ast, t_1)$,
\begin{align*} 
     \norm{\del_t^{j+1} \frac{\WJ}{\vint{v}^{2s}}}_{L^2((t_{\ast\ast}, t_1) \times \R^d \times \R^{d-1})}
\leq 
     c_{8,3} \norm{u}_{L^2((t_0, t_1) \times \R^d \times \Ss^{d-1})} \,,
\end{align*}
where $c_{8,3}$ only depends on $d,s,b,j, \frac{1}{t_{\ast\ast} - t_\ast}$ for any $j \geq 1$. Since $t_\ast$ is arbitrary, we have that for any $j \geq 1$,
\begin{align*} 
     \norm{\del_t^{j} \frac{\WJ}{\vint{v}^{2s}}}_{L^2((t_{\ast}, t_1) \times \R^d \times \R^{d-1})}
\leq 
     c_{8,4} \norm{u}_{L^2((t_0, t_1) \times \R^d \times \Ss^{d-1})} \,,
\end{align*}
where $c_{8,4}$ only depends on $d,s,b,j, \frac{1}{t_\ast - t_0}$. Similarly, one can use similar induction argument to show that for any $j_1, j_2 \geq 0$ and $t_\ast \in (t_0, t_1)$,
\begin{align*} 
     \norm{(-\Delta_x)^{j_1} \del_t^{j_2} \frac{\WJ}{\vint{v}^{2s}}}_{L^2((t_{\ast}, t_1) \times \R^d \times \R^{d-1})}
\leq 
     c_{j_1,j_2} \norm{u}_{L^2((t_0, t_1) \times \R^d \times \Ss^{d-1})} \,,
\end{align*}
where $c_{j_1,j_2}$ only depends on $d,s,b,j_1, j_2, \frac{1}{t_\ast - t_0}$. 
\end{proof}

\section{Existence, Uniqueness and stability}
In this section we use the previous work to prove the main theorem of the paper.  We start with a lemma that approximates the limiting model by integrable Henyey-Greenstein models.  The main theorem will follow from here.  We just recall the notation here
\begin{subequations}
\begin{align}
\text{Approximating Kernel :} \; b^{g}_{s}(z)&= b^{g}(z) + h(z)\,,\;\;h(z)=\frac{\tilde{b}(z)}{(1-z)^{\frac{d-1}{2}+s}}\,,\;\; g\in(0,1)\,.\label{e0EUS}\\
\text{Limiting Kernel : }\;b_{s}(z)&=\frac{1}{(1-z)^{\frac{d-1}{2}+s}}+ h(z)\,. \label{e0.1EUS}
\end{align}
\end{subequations}
Recall that the explicit form of  the approximating scattering kernels is the rescaled Henyey-Greenstein type models
\begin{equation*}
b^{g}(z) = \frac{1+g}{(1+g^2 - 2\,g\,z)^{\frac{d-1}{2}+s}} = \frac{1+g}{( (g-1)^2 + 2g(1-z) )^{\frac{d-1}{2}+s}}\,.
\end{equation*}
With this in mind it will be convenient to introduce the operator $(-\Delta_{v})^{s}_{g}$ which approximates the $s$-fractional Laplacian
\begin{align*}
(-\Delta_{v})^{s}_{g}\psi_{\mathcal{J}}:&=-\int_{\mathbb{R}^{d-1}}\frac{ \psi_{\mathcal{J}}(v') -  \psi_{\mathcal{J}}(v)}{\delta_{g}(v,v')}\,\dv' \\
&=-\int_{\mathbb{R}^{d-1}}\frac{ \psi_{\mathcal{J}}(v+z) -  \psi_{\mathcal{J}}(v)}{\delta_{g}(v,v+z)}\,\dz \\
\text{where }\;\delta_{g}(v,v'):&=(1+g)^{-1}\big( (g-1)^{2}\langle v \rangle^{2}\langle v' \rangle^{2} + 4g|v'-v|^{2} \big)^{\frac{d-1}{2}+s}\,.
\end{align*}
Note that for each $g \in (0, 1)$, the operator $(-\Delta_v)^s_g$ is bounded on $L^2(\R^{d-1})$. Furthermore, we have the following useful inequalities that follows from the symmetry of $\delta_{g}(v,v')$
\begin{align}\label{e-1EUS}
0\leq \int_{\mathbb{R}^{d-1}}\psi_{\mathcal{J}}&\,(-\Delta_{v})^{s}_{g}\psi_{\mathcal{J}} \dv 
= \tfrac{1}{2} \int_{\mathbb{R}^{d-1}}\int_{\mathbb{R}^{d-1}}\frac{\big(\psi_{\mathcal{J}}(v') - \psi_{\mathcal{J}}(v)\big)^{2}}{\delta_{g}(v,v')} \dv\dv'\nonumber\\
&\lesssim \int_{\mathbb{R}^{d-1}}\big|(-\Delta_{v})^{s/2}\psi_{\mathcal{J}}\big|^{2}\dv=\|\psi\|^{2}_{H^{s}_{\theta}}\,,\quad\quad g\in(0,1]\,.
\end{align}
Finally, it will also be convenient to express the scattering approximating operator in terms of $(-\Delta_{v})^{s}_{g}$ as we did in the second section with the limiting operator in equation \eqref{main:smooth1}.  Indeed, performing analog computations to those of \eqref{smooth1} it follows that
\begin{equation}\label{e-2EUS}
\frac{\big[\mathcal{I}_{b^g}(u)\big]_{\mathcal{J}}}{\langle\cdot\rangle^{d-1}} = -\langle\cdot\rangle^{2s}(-\Delta_{v})^{s}_{g}w_{\mathcal{J}}+u_{\mathcal{J}}\langle\cdot\rangle^{2s}(-\Delta_{v})^{s}_{g}\frac{1}{\langle\cdot\rangle^{d-1-2s}}\,.
\end{equation}  
\begin{Proposition}\label{l1EUS}
(Estimates on physical solutions) Let $u_o$ be a nonnegative initial state such that $[u_{o}]_{\mathcal{J}}\in\mathcal{C}^{2}_{x,v}$ with compact support and consider a scattering kernel \eqref{e0EUS} with $h \in L^{1}_{\theta}$.  Then, the radiative transport equation with scattering kernel $b^{g}_{s}$ has a unique solution $u^{g}\geq0$ such that $u^{g},\,\nabla_{x}u^{g},\,\partial_{t}u^{g},\,\mathcal{I}_{b^g}(u^{g})\in \mathcal{C}\big([0,T);L^{2}_{x,\theta}\big)$. Moreover, for any time $T > 0$, the solution $u^g$ satisfies
\begin{subequations}
\begin{align}
\sup_{t\geq0}\|u^{g}(t)\|_{L^{2}_{x,\theta}}\leq \|u_o\|_{L^{2}_{x,\theta}}\,, &\quad\quad \sup_{t\geq0}\|\nabla_{x} u^{g}(t)\|_{L^{2}_{x,\theta}}\leq \|\nabla_{x} u_o\|_{L^{2}_{x,\theta}}\,,\label{e1EUS}\\
\sup_{t\geq0}\|\partial_{t}u^{g}(t)\|_{L^{2}_{x,\theta}} &+ \sup_{t\geq0}\|\mathcal{I}_{b^g}(u^{g})\|_{L^{2}_{x,\theta}}\leq C\|[u_o]_{\mathcal{J}}\|_{\mathcal{C}^{2}_{x,v}}\,,\label{e1.1EUS}
\end{align}
\end{subequations}
where the constant $C:=C(supp(u_o))$ is independent of the approximating parameter $g$.
\end{Proposition}
\begin{proof}
Since the scattering kernel $b^{g}_{s}$ is integrable for any $g\in(0,1)$ one has that the scattering operator $\mathcal{I}_{b^g_{s}}$ is a bounded operator in $L^{2}_{x,\theta}$.  Since the initial condition $0\leq u_o\in L^{2}_{x,\theta}$, it follows that the RTE has a unique nonnegative solution $u^{g}\in\mathcal{C}([0,T);L^{2}_{x,\theta})$ satisfying such initial datum and the estimate
\begin{equation*}
\sup_{t\geq0}\|u^{g}(t)\|_{L^{2}_{x,\theta}}\leq\|u_o\|_{L^{2}_{x,\theta}}\,.
\end{equation*}
We refer to \cite[Chapter XXI - Theorem 3]{DL} for the details of the proof.  As a consequence, $\mathcal{I}_{b^g_{s}}(u^{g})\in\mathcal{C}([0,T);L^{2}_{x,\theta})$.  Furthermore, $\nabla_{x}u^{g}$ satisfies the same RTE with initial condition $\nabla_{x}u_o$, therefore, using the same theorem it follows that $\nabla_{x}u^{g}\in\mathcal{C}([0,T);L^{2}_{x,\theta})$ with estimate
\begin{equation*}
\sup_{t\geq0}\|\nabla_{x}u^{g}(t)\|_{L^{2}_{x,\theta}}\leq \|\nabla_{x}u_o\|_{L^{2}_{x,\theta}}\,.
\end{equation*}
Thus,
\begin{equation*}
\partial_{t}u^{g}=-\theta\cdot\nabla_{x}u^{g} + \mathcal{I}_{b^{g}_{s}}(u^{g})\in\mathcal{C}([0,T);L^{2}_{x,\theta})\,,
\end{equation*} 
and hence,
\begin{equation*}
(\partial_{t}u^{g})_{o}= -\theta\cdot\nabla_{x}u_o + \mathcal{I}_{b^{g}_{s}}(u_o)\,.
\end{equation*}
But $\partial_{t}u^{g}$ satisfies the same RTE with initial condition $(\partial_{t}u^{g})_o$, therefore using the same rationale 
\begin{align*}
\sup_{t\geq0}\|\partial_{t}u^{g}\|_{L^{2}_{x,\theta}}\leq\|(\partial_{t}u^{g})_o\|_{L^{2}_{x,\theta}}&=\|-\theta\cdot\nabla_{x}u_o + \mathcal{I}_{b^{g}_{s}}(u_o)\|_{L^{2}_{x,\theta}}\\
&\hspace{-3cm}\leq \|-\theta\cdot\nabla_{x}u_o\|_{L^{2}_{x,\theta}} + \|\mathcal{I}_{b^{g}}(u_o)\|_{_{L^{2}_{x,\theta}}} + \|\mathcal{I}_{h}(u_o)\|_{L^{2}_{x,\theta}}\leq C\|[u_o]_{\mathcal{J}}\|_{\mathcal{C}^{2}_{x,v}}\,,
\end{align*}
with constant $C:=C(supp(u_o))$ independent of $g\in(0,1)$.  For the last inequality we have used the formula \eqref{e-2EUS} and Lemma \ref{app:Dapproxconv} to obtain the estimate
\begin{align*}
\|\mathcal{I}_{b^{g}}(u_o)\|_{_{L^{2}_{x,\theta}}} = \Big\|\frac{[\mathcal{I}_{b^{g}}(u_o)]_{\mathcal{J}}}{\langle \cdot \rangle^{d-1}}\Big\|_{L^{2}_{x,v}} &\leq \Big\| -\langle\cdot\rangle^{2s}(-\Delta_{v})^{s}_{g}[w_{o}]_{\mathcal{J}} \Big\|_{L^{2}_{x,v}} +\nonumber\\
&\hspace{-1.5cm}+\Big\| [u_{o}]_{\mathcal{J}}\langle\cdot\rangle^{2s}(-\Delta_{v})^{s}_{g}\frac{1}{\langle\cdot\rangle^{d-1-2s}}\Big\|_{L^{2}_{x,v}} \leq C\|[u_o]_{\mathcal{J}}\|_{\mathcal{C}^{2}_{x,v}}\,,
\end{align*}
valid for a compactly supported function $[u_o]_{\mathcal{J}}\in\mathcal{C}^{2}_{x,v}$.  Finally, using the RTE once more
\begin{equation*}
\mathcal{I}_{b^{g}}(u^{g}) = \partial_{t}u^{g} + \theta\cdot\nabla_{x}u^{g} - \mathcal{I}_{h}(u^{g})
\end{equation*}
proves the estimate for $\mathcal{I}_{b^{g}}(u^{g})$.
\end{proof}
\begin{Proposition}\label{p1EUS}
(Stability and existence of solutions) Let $u_o$ be a nonnegative initial state such that $[u_{o}]_{\mathcal{J}}\in\mathcal{C}^{2}_{x,v}$ with compact support and consider a scattering kernel \eqref{e0EUS} with $h \in L^{1}_{\theta}$.  Then, the solutions $u^{g}$ of the radiative transport equation associated to the scattering kernel \eqref{e0EUS} converge weakly in $L^{2}\big([0,T); L^{2}_{x,\theta}\big)$ to a nonnegative limit $u\in \mathcal{C}\big([0,T);L^{2}_{x,\theta})$ which additionally satisfies $\nabla_{x}u,\,\partial_{t}u\in L^{\infty}\big([0,T);L^{2}_{x,\theta})$ and $(-\Delta_{v})^{s}w_{\mathcal{J}}\in L^{\infty}\big([0,T);L^{2}_{x,v})$.  Such limit is the unique solution of the radiative transport equation with kernel \eqref{e0.1EUS} satisfying the initial condition $u_0$ and the estimates
\begin{subequations}
\begin{align}
\sup_{t\geq0}\|u(t)\|_{L^{2}_{x,\theta}}\leq \|&u_o\|_{L^{2}_{x,\theta}}\,,\quad\quad \|\nabla_{x} u(t)\|_{L^{\infty}(L^{2}_{x,\theta})}\leq \|\nabla_{x} u_o\|_{L^{2}_{x,\theta}}\,,\label{e21EUS}\\
\|\partial_{t}u(t)\|_{L^{\infty}(L^{2}_{x,\theta})} &+\|\langle \cdot \rangle^{2s}(-\Delta_{v})^{s}w_{\mathcal{J}}\|_{L^{\infty}(L^{2}_{x,\theta})}\leq C\|[u_o]_{\mathcal{J}}\|_{\mathcal{C}^{2}_{x,v}}\,,\label{e22EUS}
\end{align}
\end{subequations}
where $C:=C(supp(u_o))$.
\end{Proposition}
\begin{proof}
Let $\{u^{g}\}$ the sequence formed by the approximate problems.  Thanks to Proposition \ref{l1EUS} there exists a function $u\in L^{2}\big([0,T);L^{2}_{x,\theta}\big)$ such that the following \textit{weak}-$L^{2}\big([0,T);L^{2}_{x,\theta}\big)$ convergence happens as $g\rightarrow1$ (up to extracting a subsequence if necessary)
\begin{align*}
u^{g} \rightharpoonup u\,, \quad \nabla_{x}u^{g} \rightharpoonup \nabla_{x}u\,,\quad \partial_{t}u^{g} \rightharpoonup \partial_{t}u\,, \quad \mathcal{I}_{b^g_{s}}(u^{g}) \rightharpoonup \mathcal{I}\,.
\end{align*}
Let us prove that $\mathcal{I}=\mathcal{I}_{b_s}(u)$.  Clearly $\mathcal{I}_{h}(u^{g})\rightharpoonup\mathcal{I}_{h}(u)$ since $\mathcal{I}_{h}$ is a bounded operator, therefore, we need only to identify the weak limit of $\mathcal{I}_{b^{g}}(u^{g})$.  To this end, it suffices to identify the distributional limit of each piece of the right side in the formula \eqref{e-2EUS}.  First note that for any $\psi\in L^{2}\big([0,T); \mathcal{D}_{x,v}\big)$
\begin{align*}
\int^{T}_0\int_{\mathbb{R}^{d}}\int_{\mathbb{R}^{d-1}}\langle \cdot \rangle^{2s}(-\Delta_{v})^{s}_{g}w_{\mathcal{J}}^{g}\,&\psi \dv \dx \dt= \int^{T}_0\int_{\mathbb{R}^{d}}\int_{\mathbb{R}^{d-1}}w_{\mathcal{J}}^{g}\,(-\Delta_{v})^{s}_{g}\phi\dv \dx \dt \\
&\hspace{-2.2cm}=\int^{T}_0\int_{\mathbb{R}^{d}}\int_{\mathbb{R}^{d-1}}w_{\mathcal{J}}^{g}\,(-\Delta_{v})^{s}\phi\dv \dx \dt\\
& \hspace{-.5cm}+ \int^{T}_{0}\int_{\mathbb{R}^{d}}\int_{\mathbb{R}^{d-1}}w_{\mathcal{J}}^{g}\,\big((-\Delta_{v})^{s}_{g} - (-\Delta_{v})^{s}\big)\phi\dv \dx \dt\,,
\end{align*}
where $\phi=\langle \cdot \rangle^{2s}\,\psi$.  Since $\langle\cdot\rangle^{2s}(-\Delta_{v})^{s}\phi \in L^{2}\big(L^{2}_{x,v}\big)$ it follows that 
\begin{equation*}
\int^{T}_{0}\int_{\mathbb{R}^{d}}\int_{\mathbb{R}^{d-1}}w_{\mathcal{J}}^{g}\,(-\Delta_{v})^{s}\phi\dv \dx \dt \rightarrow \int^{T}_{0}\int_{\mathbb{R}^{d}}\int_{\mathbb{R}^{d-1}}w_{\mathcal{J}}\,(-\Delta_{v})^{s}\phi\dv \dx \dt\,.
\end{equation*}
Furthermore, using H\"{older}'s inequality, Lemma \ref{app:Dapproxconv}, and Lebesgue dominated convergence, 
\begin{align*}
\Big|\int^{T}_{0}\int_{\mathbb{R}^{d}}\int_{\mathbb{R}^{d-1}}w_{\mathcal{J}}^{g} \,&\big((-\Delta_{v})^{s}_{g} - (-\Delta_{v})^{s}\big)\phi\dv \dx \dt \Big| \\
&\leq \big\|u^{g}\big\|_{L^{2}(L^{2}_{x,\theta})}\big\|\langle\cdot\rangle^{2s}\big((-\Delta_{v})^{s}_{g} - (-\Delta_{v})^{s}\big)\phi\big\|_{L^{2}(L^{2}_{x,v})}\\
&\hspace{2cm}\leq C\big\|\langle\cdot\rangle^{2s}\big((-\Delta_{v})^{s}_{g} - (-\Delta_{v})^{s}\big)\phi\big\|_{L^{2}(L^{2}_{x,v})}\longrightarrow 0\,.
\end{align*}
In this way
\begin{equation}\label{e22.5EUS}
\langle \cdot \rangle^{2s}\, (-\Delta_{v})^{s}_{g}w^{g}_{\mathcal{J}} \rightarrow \langle \cdot \rangle^{2s}\, (-\Delta_{v})^{s}w_{\mathcal{J}}\quad\text{in}\quad L^{2}\big([0,T); \mathcal{D}'_{x,v}\big)\,.
\end{equation}
The distributional limit of the second term in formula \eqref{e-2EUS} follows the same rationale
\begin{equation}\label{e22.51EUS}
\langle \cdot \rangle^{2s}u^{g}_{\mathcal{J}}(-\Delta_{v})^{s}_{g}\frac{1}{\langle \cdot \rangle^{d-1-2s}}\rightarrow \langle \cdot \rangle^{2s}u_{\mathcal{J}}(-\Delta_{v})^{s}\frac{1}{\langle \cdot \rangle^{d-1-2s}}\quad\text{in}\quad L^{2}\big([0,T); \mathcal{D}'_{x,v}\big)\,,
\end{equation}
and, as a consequence of \eqref{e22.5EUS} and \eqref{e22.51EUS}
\begin{equation*}
\frac{\big[\mathcal{I}_{b^{g}}(u^{g})\big]_{\mathcal{J}}}{\langle\cdot\rangle^{d-1}}\rightarrow -\langle \cdot \rangle^{2s}\, (-\Delta_{v})^{s}w_{\mathcal{J}} +  \langle \cdot \rangle^{2s}u_{\mathcal{J}}(-\Delta_{v})^{s}\frac{1}{\langle \cdot \rangle^{d-1-2s}} \quad \text{in}\quad L^{2}\big([0,T); \mathcal{D}'_{x,v}\big)\,.
\end{equation*}
This readily implies that
\begin{equation*}
\frac{\big[\mathcal{I}_{b^{g}_{s}}(u^{g})\big]_{\mathcal{J}}}{\langle\cdot\rangle^{d-1}} = \frac{\big[\mathcal{I}_{b^{g}}(u^{g})\big]_{\mathcal{J}}}{\langle\cdot\rangle^{d-1}}+\frac{\big[\mathcal{I}_{h}(u^{g})\big]_{\mathcal{J}}}{\langle\cdot\rangle^{d-1}}\rightarrow\frac{\big[\mathcal{I}_{b_{s}}(u)\big]_{\mathcal{J}}}{\langle\cdot\rangle^{d-1}}\quad \text{in} \quad L^{2}\big([0,T); \mathcal{D}'_{x,v}\big)\,.
\end{equation*}
But, it is known that
\begin{equation*}
\frac{\big[\mathcal{I}_{b^{g}_{s}}(u^{g})\big]_{\mathcal{J}}}{\langle\cdot\rangle^{d-1}} \rightarrow\frac{\big[\mathcal{I}\big]_{\mathcal{J}}}{\langle\cdot\rangle^{d-1}}\quad \text{weakly in} \quad L^{2}\big([0,T); L^{2}_{x,v}\big)\,,
\end{equation*}
thus, due to uniqueness of distributional limits $\mathcal{I}=\mathcal{I}_{b_s}(u)$.  Now, take the weak limit in $L^{2}\big([0,T); L^{2}_{x,\theta}\big)$ in the equation for $u^{g}$
\begin{align}\label{e23.111EUS}
\partial_{t}u^{g} + \theta\cdot\nabla_{x}u^{g} = \mathcal{I}_{b^{g}_{s}}(u^{g})\rightarrow \partial_{t}u + \theta\cdot\nabla_{x}u = \mathcal{I}_{b_{s}}(u)\,,
\end{align}
and conclude that $u$ satisfies the radiative transfer equation in the peaked regime \eqref{DWS}.  Estimates \eqref{e21EUS} and \eqref{e22EUS} are an easy consequence of \eqref{e1EUS} and \eqref{e1.1EUS} and the fact that the weak limit does not increases the $L^{\infty}(L^{2}_{x,\theta})$-norm.  In particular, the estimate of the fractional Laplacian follows by noticing that
\begin{equation*}
-\langle \cdot \rangle^{2s}\, (-\Delta_{v})^{s}w_{\mathcal{J}} = \frac{\big[\mathcal{I}_{b_s}(u)\big]_{\mathcal{J}}}{\langle\cdot\rangle^{d-1}} - \frac{\big[\mathcal{I}_{h}(u)\big]_{\mathcal{J}}}{\langle\cdot\rangle^{d-1}} - \langle \cdot \rangle^{2s}u_{\mathcal{J}}(-\Delta_{v})^{s}\frac{1}{\langle \cdot \rangle^{d-1-2s}}\,.
\end{equation*}
Therefore, for a.e $t\in[0,T)$ it holds
\begin{align*}
\|\langle \cdot \rangle^{2s}\, (-\Delta_{v})^{s}w_{\mathcal{J}}(t)&\|_{L^{2}_{x,v}} \leq \|\mathcal{I}_{b_s}(u(t))\|_{L^{2}_{x,\theta}} +\\
&+\|\mathcal{I}_{h}(u(t))\|_{L^{2}_{x,\theta}} + C_o\|u(t)\|_{L^{2}_{x,\theta}}\leq C\|[u_o]_{\mathcal{J}}\|_{\mathcal{C}^{2}_{x,v}}\,,
\end{align*}
for some constant $C:=C(supp(u_o))$.  Additionally, the convergence of the time derivative implies that $u\in \mathcal{C}\big([0,T);L^{2}_{x,\theta}\big)$ and, as a consequence, we must have $u(0)=u_o$.  Finally, the fact that the whole approximating sequence $\{u^{g}\}$ converges weakly to $u$ follows by the uniqueness of solutions of the limiting problem.
\end{proof}
\begin{Theorem}\label{t2EUS}
(Existence of solutions for general initial configuration and regularity) Let $u_o\in L^{1}_{x,\theta}$ be a nonnegative state and consider a scattering kernel \eqref{e0.1EUS} with $h\in L^{1}_{\theta}$.  Then, the radiative transport equation in the forward peaked regime has a unique solution $0\leq u\in \mathcal{C}([0,T);L^{1}_{loc})\cap L^{\infty}([0,T);L^{1}_{x,\theta})$ with initial state $u_0$ and satisfying conservation of mass $\int u(t) = \int u_0$ for all $t\geq0$.  Moreover, the solution is such that $(-\Delta_{x})^{k_{1}}\partial^{k_{2}}_{t}u \in L^{\infty}\big((0,T);L^{2}_{x,\theta}\big)$ (for any $k_{1},\,k_{2}\in\N$), $(-\Delta_{v})^{s/2}w_{\mathcal{J}}\in L^{\infty}\big((0,T);L^{2}_{x,v}\big)$ and $(-\Delta_{v})^{s}w_{\mathcal{J}}\in L^{2}\big((0,T);L^{2}_{x,v}\big)$ with estimates for any $t_o>0$
\begin{subequations}
\begin{align}
\|(-\Delta_{x})^{k_{1}}\partial^{k_2}_{t}u\|^{2}_{L^{\infty}((t_o,T);L^{2}_{x,\theta})} &\leq C_{k_1,k_2}\big(t^{-}_o,\|u_o\|_{L^{1}_{x,\theta}}\big)\quad \forall\;k_{1},\,k_2\in\mathbb{N}\,,\label{e3.3EUS}\\
\|(-\Delta_{v})^{s/2}w_{\mathcal{J}}\|^{2}_{L^{\infty}((t_o,T);L^{2}_{x,v})}&\leq C\big(t^{-}_o,\|u_o\|_{L^{1}_{x,\theta}}\big)\,,\label{e3.2EUS}
\\
\|\langle \cdot \rangle^{2s}(-\Delta_{v})^{s}w_{\mathcal{J}}\|^{2}_{L^{2}((t_o,T);L^{2}_{x,v})}&\leq C\big(t^{-}_o,\|u_o\|_{L^{1}_{x,\theta}}\big)\,t\,,\label{e23EUS-1}
\end{align}
\end{subequations}
where $t^{-}_o\in(0,t_o)$.  Furthermore, if $h \in \mathcal{C}^{N_o}\big([-,1,1]\big)$, then it follows that
\begin{equation}\label{e25.0EUS}
\big\|\langle \cdot \rangle^{2s} (-\Delta_{v})^{\frac{k+s}{2}} w_{\mathcal{J}}\big\|_{L^{2}((t_o,T);L^{2}_{x,v})} \leq C_{k}\big(t^{-}_o, \|u_o\|_{L^{1}_{x,\theta}}\big)\,t\,,\quad 0\leq k\leq N_0\,.
\end{equation}

\end{Theorem}
\begin{proof}
Let $\{u^j_o\}^{\infty}_{j=1}$ be a sequence of nonnegative initial states such that $\{[u^{j}_o]_{\mathcal{J}}\}\subset\mathcal{C}^{2}_{x,v}$ with compact support converging \textit{strongly} to $u_o\in L^{1}_{x,\theta}$.  By Proposition \ref{p1EUS} such sequence produces a sequence $\{u^{j}(t)\}^{\infty}_{j=1}$ of solutions to the RTE in the peaked regime  satisfying the estimates stated there.  These solutions belong to $\mathcal{C}\big([0,T);L^{2}_{x,\theta}\big)$. In particular, they belong to $\mathcal{C}\big([0,T);L^{1}_{loc}\big)$.  We can subtract the equations for $u^{j}(t)$ and $u^{l}(t)$, multiply the resulting equation by $\text{sign}(u^{j}(t) - u^{l}(t))$, and integrate in $[s,t]\times B_{R}\times\mathbb{S}^{d-1}$ ($B_{R}$ is the open ball centered at zero and radius $R>0$). Using the contraction property of the scattering operator
\begin{equation*}
\int_{\mathbb{S}^{d-1}}\mathcal{I}(u^{j}(t) - u^{l}(t))\,\text{sign}(u^{j}(t) - u^{l}(t))\,\dtheta\leq 0
\end{equation*}
we conclude that
\begin{align}\label{cauchy0.5EUS}
\|u^j(t) - u^l(t)\|_{L^{1}(B_{R}\times\mathbb{S}^{d-1})} &\leq \|u^j(s) - u^l(s)\|_{L^{1}(B_{R}\times\mathbb{S}^{d-1})}\nonumber\\
& + \int^{t}_{s}\int_{\partial B_{R}}\int_{\mathbb{S}^{d-1}}\big|u^j(t') - u^l(t')\big|\,(\theta\cdot\hat{x})\,\dtheta\text{d}\hat{x}\dt'\,.
\end{align}
Observe that the integral in the right side of this inequality is well defined by the spatial regularity of the sequence $\{u^{j}(t)\}^{\infty}_{j=1}$ (thus, the integral on $\partial B_{R}$ make sense), furthermore, it holds for any $0\leq s\leq t$ due to time continuity in $L^{1}(B_{R}\times\mathbb{S}^{d-1})$.  In particular, evaluating at $s=0$ and then sending $R\rightarrow\infty$
\begin{equation}\label{cauchyEUS}
\sup_{t\geq0}\|u^j(t) - u^l(t)\|_{L^{1}_{x,\theta}} \leq \|u^j_o - u^l_o\|_{L^{1}_{x,\theta}}\,,
\end{equation}
where we used that the integral term in the right side of \eqref{cauchy0.5EUS} belongs to $L^{1}(0,\infty)$ as a function of $R$ (for any fixed times $s$ and $t$).  Thus, the sequence $\{u^j\}^{\infty}_{j=1}$ is Cauchy in $\mathcal{C}([0,T);L^{1}_{loc})$, and therefore, it converges strongly to a limit $u\in \mathcal{C}([0, T); L^{1}_{loc})\cap L^{\infty}([0,T);L^{1}_{x,\theta})$ with $u(0)=u_o$.\\

\noindent
Note that each $u^{j}\in L^{\infty}(L^{p_{s}}_{x,\theta})$ since each $u^{j}_{o}\in L^{p_s}_{x,\theta}$, therefore, $u^{j}\in L^{\infty}(H^{s}_{x,\theta})$ for any $j\in\N$.  In this way, Sobolev inequality \eqref{SE} is valid for such sequence, hence, the \textit{a priori} estimate of Proposition \ref{prop:L-1-L-2}.  As a consequence, it follows from Propositions \ref{prop:L-1-L-2} and \ref{prop:u-time-reg} that
\begin{equation*}
\int^{T}_{t_o}\|(-\Delta_{x})^{k_1}\partial^{k_2}_{t}u^{j}(t')\|^{2}_{L^{2}_{x,\theta}}\,\dt' \leq C_{k_1,k_2}(t^{-}_o,\|u^{j}_o\|_{L^{1}_{x,\theta}})\,T\,, \quad \forall\;k_{1},\,k_{2}\in\mathbb{N}\,.
\end{equation*}
Applying the operator $(-\Delta_{x})^{k_{1}}\partial^{k_2}_{t}$, with $k_{1},\,k_{2}\in\N$,  to the RTE, multiplying the result by $(-\Delta_{x})^{k_{1}}\partial^{k_2}_{t}u$ and integrating in space and angle it follows that
\begin{equation*}
\|(-\Delta_{x})^{k_1}\partial^{k_2}_{t}u^{j}(t)\|^{2}_{L^{2}_{x,\theta}}\leq \|(-\Delta_{x})^{k_1}\partial^{k_2}_{t}u^{j}(s)\|^{2}_{L^{2}_{x,\theta}}\,,\quad 0< t_o <s<t<T\;\;(\text{a.e in}\; s,\,t)\,.
\end{equation*}
Thus, estimate \eqref{e3.3EUS} follows after averaging in $s\in(t_o,2t_o)$ and then using the propagation property of the $L^{2}$-norms of spatial and time derivatives.  Furthermore, from Lemma \ref{prop:basic-L-2} it is concluded that
\begin{equation*}
\big\|(-\Delta_{v})^{s/2}w_{\mathcal{J}}^{j} \big\|_{L^{2}((t_o,T)\times \mathbb{R}^{d}\times\mathbb{R}^{d-1})} \leq C(t^{-}_o)\big\| u^{j} \big\|_{L^{2}((t^{-}_o,T)\times \mathbb{R}^{d}\times\mathbb{S}^{d-1})}\leq C(t^{-}_o,\|u^j_o\|_{L^{1}_{x,\theta}})\,.
\end{equation*}
Thus, multiplying the projected RTE by $(-\Delta_{v})^{s}w_{\mathcal{J}}^{j}$ and integrating in $x$ and $v$ it readily follows that for a.e $t\in(t_o, T)$
\begin{align}\label{e3.0EUS}
\frac{\text{d}}{\dt}\langle w^{j}_{J}(t),\,&(-\Delta_{v})^{s}w^{j}_{J}(t) \rangle
\nonumber\\
&\leq -\tfrac{1}{2}\|\langle \cdot \rangle^{2s}(-\Delta_{v})^{s}w^{j}_{\mathcal{J}}(t)\|^{2}_{L^{2}_{x,v}}+C^{2}\big(\|u^{j}(t_o)\|^{2}_{L^{2}_{x,\theta}} + \|\nabla_{x}u^{j}(t_o)\|^{2}_{L^{2}_{x,\theta}}\big) \nonumber\\
&\hspace{1cm}\leq -\tfrac{1}{2}\|\langle \cdot \rangle^{2s}(-\Delta_{v})^{s}w^{j}_{\mathcal{J}}(t)\|^{2}_{L^{2}_{x,v}} + C(t^{-}_o,\|u^j_o\|_{L^{1}_{x,\theta}})\,.
\end{align}
Integrating \eqref{e3.0EUS} in $t\in(s,T)$ (for a.e $s$) one has in the one hand
\begin{equation*}
\tfrac{1}{2}\int^{T}_{s}\|\langle \cdot \rangle^{2s}(-\Delta_{v})^{s}w^{j}_{\mathcal{J}}(\tau)\|^{2}_{L^{2}_{x,v}}\dtau \leq \|(-\Delta_{v})^{s/2}w^{j}_{\mathcal{J}}(s)\|^{2}_{L^{2}_{x,v}} + C(t^{-}_o,\|u^j_o\|_{L^{1}_{x,\theta}})\,T\,.
\end{equation*}
As a consequence, estimate \eqref{e23EUS-1} is proved after averaging in $s\in(t_0, T)$ and sending $j\rightarrow\infty$.  In the other hand, estimate \eqref{e3.0EUS} also implies that for $0<t^{-}_o < s\leq t<T$ (a.e in $s$ and $t$)
\begin{equation*}
\|(-\Delta_{v})^{s/2}w^{j}_{\mathcal{J}}(t)\|^{2}_{L^{2}_{x,v}} \leq \|(-\Delta_{v})^{s/2}w^{j}_{\mathcal{J}}(s)\|^{2}_{L^{2}_{x,v}} + C(t^{-}_o,\|u^j_o\|_{L^{1}_{x,\theta}})\,.
\end{equation*}
Therefore, estimate \eqref{e3.2EUS} follows after averaging in $s\in(t_o,2t_o)$ and sending $j\rightarrow\infty$.  Having these estimates at hand one can pass in the weak-$L^{2}\big([t_o,T);L^{2}_{x,\theta}\big)$ limit, for any $t_o>0$, and obtain
\begin{equation*}
\partial_{t}u^{j} + \theta\cdot\nabla_{x}u^{j} = \mathcal{I}_{b_s}(u^{j})\rightarrow \partial_{t}u + \theta\cdot\nabla_{x}u = \mathcal{I}_{b_{s}}(u)\,.
\end{equation*}
Therefore, the limiting function $u$ solves the RTE in the peaked regime with initial condition $u(0)=u_o$ which conserves the mass (recall the the sequence $\{u^{j}(t)\}^{\infty}_{j=1}$ is converging strongly in $L^{\infty}(L^{1}_{x,\theta})$)
\begin{equation*}
  \iint_{\R^d \times \Ss^{d-1}} u(t) 
= \lim_{j\rightarrow\infty} \iint_{\R^d \times \Ss^{d-1}} u^{j}(t)
= \lim_{j\rightarrow\infty} \iint_{\R^d \times \Ss^{d-1}} u^{j}_o 
= \iint_{\R^d \times \Ss^{d-1}} u_o\,,\quad t\geq 0\,.
\end{equation*}
Hence $u$ fulfills all the requirements to be a solution.  Additionally, observe that estimate \eqref{e25.0EUS} is a direct consequence of Proposition \ref{prop:infty-reg}.\\
 
\noindent
Regarding uniqueness, let $v(t)$ be any other solution having initial state $u_o$, therefore, it is the case that
\begin{align*}
\|u(t) - v(t)\|_{L^{1}(B_{R}\times\mathbb{S}^{d-1})}& \leq \|u(s) - v(s)\|_{L^{1}(B_{R}\times\mathbb{S}^{d-1})}\\
&\hspace{-1cm} + \int^{t}_{s}\int_{\partial B_{R}}\int_{\mathbb{S}^{d-1}}\big|u(t') - v(t')\big|\,(\theta\cdot\hat{x})\,\dtheta\text{d}\hat{x}\dt'\quad 0<s\leq t\,.
\end{align*}
Therefore, sending first $s\rightarrow0$ and then $R\rightarrow\infty$ it follows that $\|u(t)-v(t)\|_{L^{1}_{x,\theta}}=0$ for a.e $t>0$.
\end{proof}

\section{Decay Estimate}
In this section we borrow the framework used in \cite{CafVas} to show that the solution to the RTE in the peaked regime is bounded for any positive time and its $L^\infty$-norm decays to zero algebraically in time.  The precise statement is give in the following proposition.
\begin{Proposition} \label{prop:u-decay}
Suppose $u$ is a weak solution to the transport equation~\eqref{rte1}. Then there exists a constant $c_9$ which only depends on $d,s, b,$ and $m_o$, the mass of $u_{o}$, such that
\begin{align*}
     u(T, x, \theta) 
\leq 
    c_9 \, T^{-\frac{\mu_1-\mu_2}{1-\mu_2}}
\qquad \text{for all \,\, $T > 1$} \,,
\end{align*}
where the constants $\mu_1, \mu_2$ are defined in~\eqref{def:mu-1-2}. In particular, we have
\begin{align*}
     u(T, x, \theta) 
\leq 
    c_9 \, T^{-\frac{1}{2}}
\qquad \text{as \,\, $T \to \infty$} \,.
\end{align*}
\end{Proposition}
\noindent
Before proving Proposition~\ref{prop:u-decay}, we need to establish some auxiliary bounds for $u$. To this end, define the level set functions
\begin{equation} \label{def:u-w-lambda}
      u_\lambda := (u - \lambda)_+ \,, 
\qquad
     \WJL := \frac{(\UJ - \lambda)_+}{\vint{v}^{d-1-2s}}
\qquad
\text{for any $\lambda > 0$} \,.
\end{equation}
\begin{Proposition}
Let $u \in L^2((t_0, t_1) \times \R^d \times \Ss^{d-1})$ be a weak solution to~\eqref{rte1} for $0 < t_0 < t_1 < \infty$. Then for any $\lambda > 0$, we have
\begin{equation}\label{eq:energy-level}
    \del_t \int_{\R^d}\int_{\Ss^{d-1}} u_\lambda^2 \dtheta\dx
    + D_0 \int_{\R^d}\int_{\R^{d-1}}\big|(-\Delta_v)^{s/2} \WJL\big|^{2} \dv\dx
\leq 
    D_1 \int_{\R^d} \int_{\Ss^{d-1}} u_\lambda^2 \dtheta \dx \,,
\end{equation}
where $D_0, D_1$ are the constants in~\eqref{MEE2}.
\end{Proposition}
\begin{proof}
The level set function $u_\lambda$ satisfies 
\begin{align} \label{eq:u-level}
    \del_t u_\lambda + \theta\cdot \Grad u_\lambda
    = \One_{u > \lambda} \, \CalL(u) \,.
\end{align}
By the definition of $\CalL(u)$ and $u_\lambda$, we have
\begin{equation} \nn
\begin{aligned}
-\int_{\Ss^{d-1}}& u_\lambda(\theta) \,\CalL(u) \dtheta
   = \int_{\Ss^{d-1}}\int_{\Ss^{d-1}}
             u_\lambda(\theta)  
             \frac{u(\theta) - u(\theta')}{(1 - \theta \cdot \theta')^{\frac{d-1}{2} + s}}\,
              b_s (\theta \cdot \theta) \dtheta' \dtheta
\\
   & = \int_{\Ss^{d-1}}\int_{\Ss^{d-1}}
           u_\lambda(\theta)  
           \frac{(u(\theta)-\lambda) \One_{u(\theta) > \lambda} 
                    - (u(\theta')-\lambda) \One_{u(\theta) > \lambda}}
                    {(1 - \theta \cdot \theta')^{\frac{d-1}{2} + s}}\,
              b_s (\theta \cdot \theta) \dtheta' \dtheta
\\
   & \geq  \int_{\Ss^{d-1}}\int_{\Ss^{d-1}}
           u_\lambda(\theta)  
           \frac{(u(\theta)-\lambda) \One_{u(\theta) > \lambda} 
                    - (u(\theta')-\lambda) \One_{u(\theta') > \lambda}
                       \One_{u(\theta) > \lambda}}
                    {(1 - \theta \cdot \theta')^{\frac{d-1}{2} + s}}\,
              b_s (\theta \cdot \theta) \dtheta' \dtheta
\\
   & =  \int_{\Ss^{d-1}}\int_{\Ss^{d-1}}
           u_\lambda(\theta)  
           \frac{(u(\theta)-\lambda) \One_{u(\theta) > \lambda} 
                    - (u(\theta')-\lambda) \One_{u(\theta') > \lambda}}
                    {(1 - \theta \cdot \theta')^{\frac{d-1}{2} + s}}\,
              b_s (\theta \cdot \theta) \dtheta' \dtheta
\\
   & =  \int_{\Ss^{d-1}}\int_{\Ss^{d-1}}
           u_\lambda(\theta)  
           \frac{u_\lambda(\theta) - u_\lambda(\theta')}
                    {(1 - \theta \cdot \theta')^{\frac{d-1}{2} + s}}\,
              b_s (\theta \cdot \theta) \dtheta' \dtheta
       = - \int_{\Ss^{d-1}} u_\lambda\, \CalL(u_\lambda) \dtheta
\\
  & = D_0 \|(-\Delta_v)^{s/2} \WJL\|_{L^2(\R^{d-1})}^2
           - D_1 \|u_\lambda\|_{L^2(\Ss^{d-1})}^2 \,,
\end{aligned}
\end{equation}
where the last inequality follows by~\eqref{main:smooth2}.  Estimate~\eqref{eq:energy-level} is then obtained upon multiplying~\eqref{eq:u-level} by $u_\lambda $ and integrating in $(x, \theta)$.
\end{proof}
\begin{Proposition} \label{prop:L-infty-I}
Suppose $u \in L^2((t_0, t_1) \times \R^d \times \Ss^{d-1})$ is a solution to~\eqref{rte1}.  Suppose $\tilde b \in \mathcal{C}^{N_0}([-1,1])$ with $N_0 \geq d + 2$.
Then $\vint{v}^{-(d-1)}\big[\CalL(u)\big]_{\mathcal{J}} \in L^\infty((t_\ast, t_1) \times \R^d \times \R^{d-1})$ for any $t_\ast \in (t_0, t_1)$. Moreover, 
\begin{align} \label{est:W-J-10}
     \norm{\vint{v}^{-(d-1)}\big[\CalL(u)\big]_{\mathcal{J}}}_{L^\infty((t_\ast, t_1) \times \R^d \times \R^{d-1})}
\leq 
     c_{8} \norm{u}_{L^2((t_0, t_1) \times \R^d \times \Ss^{d-1})} \,,
\end{align}
where $c_8$ only depends on $d, s, b, \frac{1}{t_\ast - t_0}$. In particular, $c_8$ is independent of $t_1$.
\end{Proposition}
\begin{proof}
We will show that $\vint{v}^{-(d-1)}\big[\CalL(u)\big]_{\mathcal{J}} \in H^{d+1}((t_\ast, t_1) \times \R^d \times \R^{d-1})$ and apply the Sobolev imbedding $H^{d+1}((t_\ast, t_1) \times \R^d \times \R^{d-1}) \hookrightarrow L^\infty((t_\ast, t_1) \times \R^d \times \R^{d-1})$.  
By the transport equation~\eqref{rte1}, $\WJ$ satisfies
\begin{align} \label{eq:W-J-5}
    \del_t \frac{\WJ}{\vint{v}^{2s}} 
    + \theta(v) \cdot \Grad \frac{\WJ}{\vint{v}^{2s}} 
  = \vint{v}^{-(d-1)} \big[\CalL(u)\big]_{\mathcal{J}} \,.
\end{align}
Therefore, we only need to show that 
\begin{align*}
    \del_t \frac{\WJ}{\vint{v}^{2s}} \in H^{d+1}((t_\ast, t_1) \times \R^d \times \R^{d-1}) \,,
\qquad
   \Grad \frac{\WJ}{\vint{v}^{2s}} \in H^{d+1}((t_\ast, t_1) \times \R^d \times \R^{d-1}) \,,
\end{align*}
or simply
\begin{align} \label{reg:w-J}
    \frac{\WJ}{\vint{v}^{2s}} \in H^{d+2}((t_\ast, t_1) \times \R^d \times \R^{d-1}) \,.
\end{align}
Using \eqref{smooth9} 
\begin{align*}
     \vint{v}^{-(d-1)}\big[\CalL(u)\big]_{\mathcal{J}}
 = - D \vint{v}^{2s} (-\Delta_v)^s \WJ 
    + \vint{v}^{-{d-1}} \big[\CalK(u)\big]_{\mathcal{J}} \,,
\end{align*}
where $\CalK=c_{s,d}\textbf{1} + \mathcal{I}_{h}$ is a bounded operator on $L^2(\Ss^{d-1})$.  Note that if we denote
\begin{align*}
    W_{j_1,j_2} = (-\Delta_x)^{j_1} \del_t^{j_2} u \,,
\end{align*}
then $W_{j_1, j_2}$ satisfies the transport equation~\eqref{rte1} and $W_{j_1, j_2} (t_\ast, \cdot, \cdot) \in L^2_{x,\theta}$ for a.e. $t_\ast \in (t_0, t_1)$ by Proposition~\ref{prop:u-time-reg}. Hence Proposition~\ref{prop:infty-reg} applies and gives
\begin{align*}
      \norm{(-\Delta_v)^{\frac{d+2}{2}} (-\Delta_x)^{j_1} \del_t^{j_2} \WJ}_{L^2((t_*, t_1) \times \R^{d} \times \R^{d-1})}
\leq 
   c_{8,6} \norm{u}_{L^2((t_0, t_1) \times \R^d \times \Ss^{d-1})} \,,
\end{align*}
for any $j_1, j_2 \geq 0$ and any $t_\ast \in (t_0, t_1)$. Here $c_{8,6}$ only depends on $d,s,b,j_1, j_2, \frac{1}{t_\ast - t_0}$. This in particular implies that
for $c_{8,7} = c_{8,5} + c_{8,6}$,
\begin{align*}
      \norm{\vpran{I + (-\Delta_v)^{\frac{d+2}{2}}} (-\Delta_x)^{j_1} \del_t^{j_2} \frac{\WJ}{\vint{v}^{2s}}}_{L^2((t_*, t_1) \times \R^{d} \times \R^{d-1})}
\leq 
  c_{8,7} \norm{u}_{L^2((t_0, t_1) \times \R^d \times \Ss^{d-1})} \,,
\end{align*}
which proves~\eqref{reg:w-J} and~\eqref{est:W-J-10}.
\end{proof}
\noindent
Let $\rho_{\lambda, N}$ be the density function such that
\begin{align} \label{def:rho-lambda}
   \rho_{\lambda, N}
   = \int_{\Ss^{d-1}} u_\lambda \, (\eta_N\circ\mathcal{S}) \dtheta
   = 2^{d-1} \int_{\R^{d-1}} \big[u_\lambda\big]_{\mathcal{J}} \, \vint{v}^{-2(d-1)} \eta_N(v)\dv \,,
\end{align}
where $\eta_N$ is a cutoff function $\eta_N \in \mathcal{C}_c^\infty(\R^{d-1})$, with $0 \leq \eta_N \leq 1$, and such that
\begin{equation*}
\eta_N = 1\;\; \text{on}\;\; B(0, N-1)\,, \qquad supp(\eta_N) \subseteq B(0, N) \subseteq \R^{d-1} \,.
\end{equation*}
Here $B(0, R)$ denotes the ball centered at $0$ with radius $R$. The introduction of $\eta_N$ will be clear in the proof of the following Proposition which gives a bound on $\rho_{\lambda, N}$.
\begin{Proposition} \label{prop:VA-N}
Let $u \in L^2((t_0, t_1) \times \R^d \times \Ss^{d-1})$ be a strong solution to~\eqref{rte1}. Let $u_\lambda$ and $\rho_{\lambda,N}$ be defined in~\eqref{def:u-w-lambda} and \eqref{def:rho-lambda} respectively. Then 
\begin{align*}
    (-\Delta_x)^{\beta} \rho_{\lambda, N} \in L^2((t_0, t_1) \times \R^{d}) \,,
\end{align*}
where $\beta$ is the same number as in Proposition~\ref{Lemma:VA-Weak}.
Moreover, 
\begin{align} 
\|(-\Delta_x)^{\beta}& \rho_{\lambda, N}\|_{L^2((t_0, t_1) \times \R^d)}^2 \leq
    c_0 \vpran{\|u^{o}_\lambda\|_{L^2_{x,\theta}}^2
                  + \|u_\lambda\|_{L^2((t_0, t_1) \times \R^d \times \Ss^{d-1})}^2}+
       \label{ineq:rho-lambda-N}
\\
& \hspace{-0.5cm}
     + c_0 N^{2(d-1)} 
        \norm{\vint{v}^{-(d-1)}\big[\CalL(u)\big]_{\mathcal{J}}}_{L^\infty((t_0, t_1) \times \R^d \times \R^{d-1})}^2
        \vpran{\int_{t_0}^{t_1} \norm{\One_{\UJ > \lambda} \, \eta_N(v)}_{L^2_{x,v}}^2 \dtau}, \nn
\end{align}
where $c_0$ is the constant in~\eqref{ineq:VA-Weak}. In particular, $c_0$ is independent of $N$.
\end{Proposition}
\begin{proof}
Multiplying~\eqref{eq:u-level} by $\eta_N\circ\mathcal{S}$, we have the equation for $\big[u_\lambda\big]_{\mathcal{J}} \eta_N$ as
\begin{align*}
    \del_t \big(\big[u_\lambda\big]_{\mathcal{J}} \eta_N\big)
    + \theta(v) \cdot \Grad\big(\big[u_\lambda\big]_{\mathcal{J}} \eta_N\big)
    = \eta_N\, \One_{\{\UJ > \lambda\}}\,\big[ \CalL(u)\big]_{\mathcal{J}}.
\end{align*}
We can then apply Proposition~\ref{Lemma:VA-Weak} to the above equation and obtain
\begin{align*} 
\|&(-\Delta_x)^{\beta} \rho_{\lambda, N}\|_{L^2_{t, x}}^2 \leq
    c_0 \vpran{\|u^{o}_\lambda\|_{L^2_{x,\theta}}^2
                  + \|u_\lambda\|_{L^2_{t,x,\theta}}^2
                  + \norm{\eta_N\,\One_{\{\UJ > \lambda\}}\, \big[\CalL(u)\big]_{\mathcal{J}}}_{L^2_{t,x,v}}^2}
\\
& \leq                
c_0 \vpran{\|u^{o}_\lambda\|_{L^2_{x,\theta}}^2
                  + \|u_\lambda\|_{L^2_{t,x,\theta}}^2}
     + c_0 N^{2(d-1)} 
        \norm{\eta_N \,\One_{\{\UJ > \lambda\}}}_{L^2_{t,x,\theta}}^2
        \norm{\vint{v}^{-(d-1)}\big[\CalL(u)\big]_{\mathcal{J}}}_{L^\infty_{t, x, v}}^2 \,,
\end{align*}
where $c_0$ is the constant in~\eqref{ineq:VA-Weak}. 
\end{proof}
\begin{proof}[Proof of Proposition~\ref{prop:u-decay}]
Let $M > 0$ be a constant to be determined. Let
\begin{equation} \nn
     \lambda_k = M (1 - 2^{-k}) \,,
\qquad
     t_k = t_0 + (T_0 - t_0) (1 - 2^{-k}) \,, 
\qquad
     k \geq 1 \,,
\end{equation}
for any $1 < t_0 < T_0$. We want to show that $u \leq M$ a.e. in $(x, \theta)$ for any $t_0 < t \leq T_0$ if $M$ is chosen large enough. Define the functional
\begin{equation} \label{def:U-k}
     U_k 
     = \sup_{t_k \leq t \leq T_0} \norm{u_{\lambda_k}}_{L^2_{x, \theta}}^2
     + \int_{t_k}^{T_0}  
            \norm{(-\Delta_v)^{\frac{s}{2}} w_{\CalJ, \lambda_k}^+}_{L^2_{x,v}}^2(\tau) \dtau 
      + \int_{t_k}^{T_0} \norm{u_{\lambda_k}}_{L^2_{x,v}}^2 \dtau\,,
\end{equation}
where recall that $u_{\lambda_k}, w_{\CalJ, \lambda_k}^+$ are defined in~\eqref{def:u-w-lambda}.
Our goal is to prove that $U_k \to 0$ as $k \to \infty$, which implies $u \leq M$ a.e. in $(x, \theta)$ for any $t \in (t_0, T_0)$. For any $s \in (t_{k-1}, t_k)$ and $t \in (t_k, T_0)$, integrate~\eqref{eq:energy-level} from $s$ to $t$ and from $s$ to $T_0$. This gives 
\begin{equation} \nn
    \sup_{t_k \leq t \leq T_0} \int_{\R^d}\int_{\Ss^{d-1}} u_{\lambda_k}^2(t) \dtheta\dx 
 \leq 
     \int_{\R^d}\int_{\Ss^{d-1}} u_{\lambda_k}^2(s) \dtheta\dx \,,
\end{equation}
and
\begin{align*}
\int_{t_k}^{T_0} \int_{\R^d}\int_{\R^{d-1}} 
             \abs{(-\Delta_v)^{\frac{s}{2}} w_{\CalJ, \lambda_k}^+}^2 \dv\dx \dtau
& \leq    
     \int_{\R^d}\int_{\Ss^{d-1}} u_{\lambda_k}^2(s) \dtheta\dx
\\
& \hspace{-1cm}
     + \frac{D_1}{D_0} \int_{t_{k-1}}^{T_0}  \int_{\R^d} \int_{\Ss^{d-1}} u_{\lambda_k}^2 \dtheta \dx \dtau \,,
\end{align*}
where $D_0, D_1$ are the constants given in~\eqref{MEE2}.  Adding up these two inequalities
\begin{equation} \nn
     U_k
 \leq 
     2 \int_{\R^d}\int_{\Ss^{d-1}} u_{\lambda_k}^2(s) \dtheta\dx
     + \vpran{\frac{2D_1}{D_0} + 1} \int_{t_{k-1}}^{T_0}  \int_{\R^d} \int_{\Ss^{d-1}} u_{\lambda_k}^2 \dtheta \dx \dtau \,,
\end{equation}
and taking the average in $s$ over $[t_{k-1}, t_k]$
\begin{equation} \nn
\begin{aligned}
     U_k
 \leq 
     \frac{2}{t_k - t_{k-1}} &\int_{t_{k-1}}^{t_k} \int_{\R^d}\int_{\Ss^{d-1}} u_{\lambda_k}^2(s) \dtheta\dx
     + \frac{2D_1}{D_0} \int_{t_{k-1}}^{T_0}  \int_{\R^d} \int_{\Ss^{d-1}} u_{\lambda_k}^2 \dtheta \dx \dtau
\\
& \leq  
     \left(\frac{2^{k+2}}{T_0 - t_0} + \frac{2D_1}{D_0} + 1\right) 
     \int_{t_{k-1}}^{T_0} \int_{\R^d}\int_{\Ss^{d-1}} u_{\lambda_k}^2(s) \dtheta\dx \,.
\end{aligned}
\end{equation}
Additionally, using Proposition~\ref{prop:VA-N}, Proposition~\ref{prop:L-infty-I}, Proposition~\ref{prop:L-1-L-2}, and the non-increase property of $\norm{u}_{L^2_{x,\theta}}$,
\begin{equation} 
   \label{est:velocity-average}
\begin{aligned}
\int_{t_k}^{T_0}  \int_{\R^d}&
          \left| (-\Delta_x)^{\beta} \rho_{\lambda_k, N}(x, \tau)\right|^2 \dx\dtau   
\\
&\leq 
      c_0 U_k + c_{9,0} (T_0 - t_0) N^{2(d-1)} 
        \int_{t_k}^{T_0} \norm{\eta_N\,\One_{\{u > \lambda_k\}}}_{L^2_{x,v}}^2 \dtau\,, \end{aligned}
\end{equation}
where $\rho_{\lambda_k, N}$ is defined in~\eqref{def:rho-lambda} with $\lambda$ replaced by $\lambda_k$. Here the constant $c_{9,0}$ only depends on $d, s, b$, and $\norm{u_{o}}_{L^1_{x, \theta}}$. In particular, $c_{9,0}$ is independent of $T_0$.  Now, following \cite{CafVas} we note that
\begin{equation} \label{bound:u-lambda-k-1-One}
     \One_{\{u_{\lambda_k} > 0\}} 
\leq 
     \left( \frac{2^k}{M} u_{\lambda_{k-1}}\right)^2 \,,
\end{equation}
hence,
\begin{align*}
   \int_{t_k}^{T_0} \norm{\eta_N\,\One_{\{\UJ > \lambda_k\}}}_{L^2_{x,v}}^2 \dtau
\leq 
   \frac{2^{2k}}{M^2} 
   \int_{t_k}^{T_0} \norm{u_{\lambda_k}}_{L^2_{x,\theta}}^2 \dtau
\leq 
   \frac{2^{2k}}{M^2} U_k \,.
\end{align*}
In light of \eqref{est:velocity-average} and the Sobolev imbedding over $x \in\R^d$ it follows that
\begin{equation} \label{ineq:Sobolev-rho}
     \int_{t_k}^{T_0} \left(\int_{\R^d} \rho_{\lambda_k, N}^{p_\beta} \dx\right)^{2/{p_\beta}} \dtau
\leq 
      \vpran{c_0 + c_{9,0} (T_0 - t_0) N^{2(d-1)} \frac{2^{2k}}{M^2}} \, U_k \,,
\end{equation}
where $\frac{1}{p_\beta} = \frac 12 - \frac{\beta}{d}$. Note that $p_\beta > 2$ and let 
\begin{equation*}
      q_\beta = 4 - 4/{p_\beta} \in (2, 4) \cap (2, p_\beta) \,.
\end{equation*} 
Then, using \eqref{ineq:Sobolev-rho} and H\"{o}lder inequality,
\begin{equation} \label{ineq:U-k-1}
\begin{aligned}
     \int_{t_k}^{T_0} \int_{\R^d}  
             \rho_{\lambda_k, N}^{q_\beta} \dx\dtau
&\leq          
     \int_{t_k}^{T_0}  
             \left(\int_{\R^d}\rho_{\lambda_k, N}^{p_\beta} \dx \right)^{\frac{2}{p_\beta}}
             \left(\int_{\R^d}\rho_{\lambda_k, N}^2 \dx \right)^{1-\frac{2}{p_\beta}} \dtau
\\
& \hspace{-1.5cm}
\leq 
    \left( \int_{t_k}^{T_0}  
             \left(\int_{\R^d}\rho_{\lambda_k, N}^{p_\beta} \dx \right)^{\frac{2}{p_\beta}} \dtau \right)
      \left(\sup_{t_k \leq t \leq T_0}  
                \int_{\R^d}\rho_{\lambda_k}^2 \dx \right)^{1-\frac{2}{p_\beta}} 
\\
& \hspace{-1.5cm}
\leq 
    \left( \int_{t_k}^{T_0}  
             \left(\int_{\R^d}\rho_{\lambda_k, N}^{p_\beta} \dx \right)^{\frac{2}{p_\beta}} \dtau \right)
      \left(\sup_{t_k \leq t \leq T_0}  
                \int_{\R^d} \int_{\Ss^{d-1}} u_{\lambda_k}^2 \dtheta\dx \right)^{1-\frac{2}{p_\beta}} 
\\
& \hspace{-1.5cm}
\leq 
    \vpran{c_0 + c_{9,0} (T_0 - t_0) N^{2(d-1)} \frac{2^{2k}}{M^2}} \, U_k^{\frac{q_\beta}{2}} \,.
\end{aligned}             
\end{equation}
Next, using \eqref{def:U-k} and the Sobolev imbedding over $\theta \in \Ss^{d-1}$
\begin{equation} \label{ineq:Sobolev-u}
     \int_{t_k}^{T_0} \int_{\R^d} 
       \left(\int_{\Ss^{d-1}} u_{\lambda_k}^{p_s} \dtheta\right)^{2/{p_s}} \dx\dtau
\leq 
      c_{9,1} \, U_k \,,
\end{equation}
where $c_{9,1}$ only depends on $d,s$.  Recall that $\tfrac{1}{p_s} = \tfrac 12 - \tfrac{s}{d-1}$ and $p_s > 2$.  Let 
\begin{equation} \label{def:parameters-1-1}
            \tilde\alpha_1 = \frac{2q_\beta - 2}{p_s q_\beta - 2} \in (0, 1) \,, 
\qquad \tilde\alpha_2 = \frac{p_s}{2} \tilde\alpha_1 \in (0, 1) \,,
\qquad  \tilde r = p_s \tilde\alpha_1 + (1 - \tilde\alpha_1) > 2 \,,
\end{equation} 
and observe that
\begin{equation} \nn
    \frac{\tilde\alpha_1}{\tilde\alpha_2} = \frac{2}{p_s} \,,
\qquad 
    \frac{1-\tilde\alpha_1}{1-\tilde\alpha_2} = q_\beta \,. 
\end{equation}
Therefore,
\begin{equation} \nn
     \tilde r 
     = 2 + \frac{(p_s - 2)(q_\beta - 2)}{p_s q_\beta - 2} 
     = 3 - \frac{2(p_s + q_\beta - 2)}{p_s q_\beta -2}
     \in (2, 3) \,.
\end{equation}
Using \eqref{ineq:Sobolev-u}, H\"{o}lder inequality, and the definitions of $\tilde\alpha_1, \tilde\alpha_2, r$ in~\eqref{def:parameters-1-1} it follows then
\begin{equation} \label{est:u-lambda-N-r}
\begin{aligned}
    \int_{t_k}^{T_0} \int_{\R^d} \int_{\Ss^{d-1}} u_{\lambda_k, N}^{\tilde r} \dtheta\dx\dtau
&\leq 
   \int_{t_k}^{T_0} \int_{\R^d} 
      \left(\int_{\Ss^{d-1}} u_{\lambda_k}^{p_s} \dtheta\right)^{\tilde\alpha_1} 
      \rho_{\lambda_k, N}^{1-\tilde\alpha_1}
      \dx\dtau
\\
&\hspace{-3.2cm}
\leq 
   \left(\int_{t_k}^{T_0} \int_{\R^d} 
      \left(\int_{\Ss^{d-1}} u_{\lambda_k}^{p_s} \dtheta\right)^{\frac{\tilde\alpha_1}{\tilde\alpha_2}} 
      \dx\dtau\right)^{\tilde\alpha_2}
      \vpran{\int_{t_k}^{T_0} \int_{\R^d} \rho_{\lambda_k, N}^{\frac{1-\tilde\alpha_1}{1-\tilde\alpha_2}} \dx\dtau}^{1-\tilde\alpha_2}
\\
& \hspace{-3.2cm} 
  = \left(\int_{t_k}^{T_0} \int_{\R^d} 
      \left(\int_{\Ss^{d-1}} u_{\lambda_k}^{p_s} \dtheta\right)^{2/p_s} 
      \dx\dtau\right)^{\tilde\alpha_2}
   \left(\int_{t_k}^{T_0} \int_{\R^d} \rho_{\lambda_k, N}^{q_\beta} \dx\dtau \right)^{1-\tilde\alpha_2}
\\
& \hspace{-3.2cm}
   \leq c_{9,2} \vpran{1 + (T_0 - t_0) N^{2(d-1)} \frac{2^{2k}}{M^2}}^{1-\tilde\alpha_2} U_k^{\tilde r/2} \,,
\end{aligned}
\end{equation}
where $c_{9,2}$ only depends on $\beta, d, s, b, t_0$, and $\norm{u_o}_{L^1_{x, \theta}}$.  Consequently,
\begin{align*}
U_k \leq  
     \Big(\frac{2^{k+2}}{T_0 - t_0} &+ \frac{2D_1}{D_0} + 1 \Big) 
     \int_{t_{k-1}}^{T_0} \int_{\R^d}\int_{\Ss^{d-1}} u_{\lambda_k}^2 \big(\eta_N\circ\mathcal{S}\big) 
     \dtheta\dx\dtau
\\
&\hspace{-1cm}
    + \left(\frac{2^{k+2}}{T_0 - t_0} +\frac{2D_1}{D_0} + 1\right)
       \int_{t_{k-1}}^{T_0} \int_{\R^d}\int_{\{|v| \geq N/2\}} \big[u_{\lambda_k}\big]^{2}_{\mathcal{J}}    
           \vint{v}^{-2(d-1)} \dv\dx\dtau \,.
\end{align*}
Furthermore, using estimate \eqref{ineq:Sobolev-u}
\begin{align*}
\int_{t_{k-1}}^{T_0} &\int_{\R^d}\int_{\{|v| \geq N/2\}} \big[u_{\lambda_k}\big]^{2}_{\CalJ}    
           \vint{v}^{-2(d-1)} \dv\dx\dtau
\\
&\leq
    \vpran{\int_{t_{k-1}}^{T_0} \int_{\R^d}
    \vpran{\int_{\R^{d-1}} 
       \big[u_{\lambda_k}\big]^{p_{s}}_{\CalJ} \dv}^{2/p_s}\dx\dtau}
    \vpran{\int_{\{|v| \geq N/2\}}  \vint{v}^{-2(d-1)} \dv}
\\
&\quad \quad \leq c_{9,1} N^{-(d-1)} U_{k}
   \leq c_{9,1} N^{-(d-1)} U_{k-1} \,.
\end{align*}
Hence, recalling inequality \eqref{bound:u-lambda-k-1-One}
\begin{align*}
U_k
& \leq  
     \left(\frac{2^{k+2}}{T_0 - t_0} + \frac{2D_1}{D_0} + 1\right) 
     \vpran{\int_{t_{k-1}}^{T_0} \int_{\R^d}\int_{\Ss^{d-1}} u_{\lambda_k}^2 \big(\eta_N\circ\mathcal{S}\big) 
     \dtheta\dx\dtau
    + c_{9,1} N^{-(d-1)} U_{k-1}}
\\
\leq &
     \left(\frac{2^{k+2}}{T_0 - t_0} + \frac{2D_1}{D_0} + 1\right)
     \vpran{\frac{2^{k(\tilde r-2)}}{M^{\tilde r-2}}
     \int_{t_{k-1}}^{T_0} \int_{\R^d}\int_{\Ss^{d-1}} 
         u_{\lambda_{k-1}}^{\tilde r} \big(\eta_N\circ\mathcal{S}\big) \dtheta\dx\dtau 
    + c_{9,1} N^{-(d-1)} U_{k-1}},
\end{align*}
where $\tilde r > 2$ was defined in~\eqref{def:parameters-1-1}. Using estimate \eqref{est:u-lambda-N-r} and taking $k$ sufficiently large
\begin{align*}
   U_k
\leq
   \frac{2^{k+3} (1 + c_{9,1})}{T_0 - t_0}
    \vpran{\frac{2^{k(\tilde r-2)}}{M^{\tilde r-2}} 
                \vpran{(T_0 - t_0) N^{2(d-1)} \frac{2^{2k}}{M^2}}^{1-\tilde\alpha_2} U_{k-1}^{\tilde r/2}
                + N^{-(d-1)} U_{k-1}}\,.
\end{align*}
Choosing $N>0$ such that we minimize the right side of previous estimate
\begin{align*}
     \frac{2^{k(\tilde r-2)}}{M^{\tilde r-2}} 
     \vpran{(T_0 - t_0) N^{2(d-1)} \frac{2^{2k}}{M^2}}^{1-\tilde\alpha_2} 
     U_{k-1}^{\tilde r/2}
  = N^{-(d-1)} U_{k-1} \,,
\end{align*}
one concludes that
\begin{align*}
     U_k 
\leq 
   \frac{4(1 + c_{9,1})  \, 2^{k (1+\nu_1)}}{(T_0 - t_0)^{1-(1 - \tilde \alpha_2)\nu_0} M^{\nu_1}}\,
   U_{k-1}^{1 + (\tilde r/2 - 1) \nu_0} \,, \qquad \tilde r > 2 \,,
\end{align*}
where 
\begin{equation*} \label{def:nu-0-1}
    \nu_0 = \frac{1}{1 + 2(1-\tilde\alpha_2)} \in (0, 1) \,,
\qquad 
   \nu_1  = (\tilde r -2 + 2(1-\tilde\alpha_2))\nu_0 \in (0, 1) \,.
\end{equation*}
Since $\tilde r > 2$, it follows that $1 + (\tilde r/2 - 1) \nu_0 > 1$. Thus, for any fixed $t_0, T_0$ one chooses $M$ as
\begin{equation} \label{def:M-1}
     M = C_1(\tilde r) (T_0 - t_0)^{-\frac{1-(1-\tilde\alpha_2)\nu_0}{\nu_1}} U_0^{(\tilde r/2 - 1)\nu_0/\nu_1} \,,
\end{equation}
for some constant $C_1(\tilde r)$ sufficiently large.  Then, it can be shown that with this choice $U_k \to 0$ as $k \to \infty$ which proves that the solution $u$ is bounded for any positive time.  Let us now study the dependence of $M$ with time.  Denoting
\begin{equation*} \label{def:mu-1-2}
    \mu_1 = \frac{1-(1- \tilde\alpha_2)\nu_0}{\nu_1} \,,
\qquad
    \mu_2 = \big(\tfrac{\tilde r}{2}-1\big)\frac{\nu_0}{\nu_1} \,,
\end{equation*}
it follows that \eqref{def:M-1} simply writes as
\begin{equation} \label{def:M-1.1}
     M = C_1(\tilde r) (T_0 - t_0)^{-\mu_1} U_0^{\mu_2} \,.
\end{equation}
Since $\tilde\alpha_2 \in (0, 1)$ and $\tilde r \in (2, 3)$, we have
\begin{align}
    -\mu_1 + &\mu_2
  = - \frac{1}{\nu_1} \vpran{1 - \nu_0 \vpran{\frac{\tilde r}{2} - \tilde \alpha_2}} \nn
\\
 & = -\frac{1}{\nu_1\vpran{1 + 2(1 - \tilde \alpha_2)}}
     \vpran{1 + 2(1 - \tilde \alpha_2) - \frac{\tilde r}{2} + \tilde \alpha_2}
  \label{sign:exponent}
\\
 &\quad\quad = -\frac{1}{\nu_1 \vpran{1 + 2(1 - \tilde \alpha_2)}}
     \vpran{(1 - \tilde \alpha_2) + (2 - \frac{\tilde r}{2})} 
  < -\frac{1}{2 \nu_1}
  < -\frac{1}{2} \,. \nn
\end{align}
It also follows by \eqref{sign:exponent} that $\mu_1 > \mu_2 > 0$, and  
\begin{align} \label{bound:mu-2}
    \mu_2 = \frac{\frac{\tilde r}{2} - 1}{\tilde r - 2 + 2(1 - \tilde \alpha_2)}
               = \frac{\frac{1}{2} (\tilde r -2)}{\tilde r - 2 + 2(1 - \tilde \alpha_2)} 
               < \frac{1}{2} \,.
\end{align}
Now, by the definition of $U_0$
\begin{equation} \label{bound:U-0-1}
U_0 \leq \big(3 + \tfrac{2D_1}{D_0}\big)\sup_{[t_0, T_0]} \norm{u}_{L^2_{x,\theta}} (T_0 - t_0 + 1) \,,
\qquad T_0 - t_0 > 1 \,.
\end{equation}
In this way, for any $T > 1$ choose $t_0 = T$ and $T_0 = 2T$.  Then, formula \eqref{def:M-1.1} gives
\begin{align} \label{def:M-T}
    M = C_1(\tilde r) \,  T^{-\mu_1} \, U_0^{\mu_2}
\end{align}
and, using the non-increasing property of $\norm{u(t)}_{L^2_{x, \theta}}$, \eqref{bound:U-0-1} and \eqref{def:M-T} it follows that
\begin{align*}
     U_0 
&\leq \big(6 + \tfrac{4D_1}{D_0}\big)\, T \, \norm{u(T, \cdot, \cdot)}_{L^2_{x,\theta}}^2 \nn
\\
&\quad \leq \big(6 + \tfrac{4D_1}{D_0}\big)\, \norm{u_o}_{L^1_{x,\theta}} \, M \, T 
\\
& \quad\quad\leq \big(6 + \tfrac{4D_1}{D_0}\big) \, C_1(\tilde r)\norm{u_o}_{L^1_{x,\theta}}
           U_0^{\mu_2} T^{1-\mu_1} \,.\nn
\end{align*}
As a result,
\begin{align} \label{bound:U-0-3}
   U_0 
\leq 
   \vpran{\big(6 + \tfrac{4D_1}{D_0}\big) \, C_1(\tilde r)\norm{u_o}_{L^1_{x,\theta}}}^{\frac{1}{1-\mu_2}}
   T^{\frac{1-\mu_2}{1-\mu_1}} \,.
\end{align}
Using~\eqref{bound:U-0-3} in~\eqref{def:M-T}, we derive that
\begin{align*}
   u(T, x, \theta) 
\leq M
\leq c_9 \, T^{-\mu_1 + \frac{1-\mu_1}{1-\mu_2} \mu_2}
= c_9 \, T^{-\frac{\mu_1 - \mu_2}{1 - \mu_2}} \,,
\end{align*}
where the coefficient $c_9$ is given by 
\begin{align*}
   c_9 = C_1(\tilde r) \vpran{(6 + \tfrac{4D_1}{D_0}) \, C_1(\tilde r)\norm{u_o}_{L^1_{x,\theta}}}^{\frac{\mu_2}{1-\mu_2}} \,.
\end{align*} 
Note that by \eqref{bound:mu-2} it follows in particular that
\begin{align*}
    u(T, x, \theta) \leq c_9\, T^{-\frac{1}{2}} \,.
\end{align*}
This proves the algebraic decay of $\norm{u(T, \cdot, \cdot)}_{L^\infty_{x, \theta}}$ as $T \to \infty$.
\end{proof}

\begin{appendices}
\section{Appendix}
\subsection{$L^{2}$ estimate of the scattering operator}
\begin{Lemma}\label{app:coercive}
Let $\theta' \rightarrow b(\theta\cdot\theta')\in L^{1}(\mathbb{S}^{d-1})$.  Then
\begin{equation*}
 \|\mathcal{I}_{b}(u)\|_{L^{2}(\mathbb{S}^{d-1})} \leq 2\|b\|_{L^{1}(\mathbb{S}^{d-1})}\, \|u\|_{L^{2}(\mathbb{S}^{d-1})}\,.
\end{equation*}
\end{Lemma}
\begin{proof}
This follows by using Cauchy-Schwarz inequality
\begin{align*}
& \quad \,
       \int_{\Ss^{d-1}} 
         \abs{\int_{\Ss^{d-1}}u(\theta')\,b(\theta\cdot\theta')\,\dtheta'}^{2}\,\dtheta
\\[2pt]
&\leq
   \int_{\mathbb{S}^{d-1}}\Big(\int_{\mathbb{S}^{d-1}}\big|u(\theta')\big|^{2}\,\big|b(\theta\cdot\theta')\big| \,\dtheta' \Big)\,\Big(\int_{\mathbb{S}^{d-1}}\,\big|b(\theta\cdot\theta')\big|\,\dtheta' \Big)\,\dtheta
\\[2pt]
&= \|b\|^{2}_{L^{1}(\mathbb{S}^{d-1})}\, \|u\|^{2}_{L^{2}(\mathbb{S}^{d-1})}\,.
\end{align*}
\end{proof}
\subsection{Fractional Laplacian}
Let $f \in L^{1}(\mathbb{R}^{d-1})$ be a sufficiently smooth function.  Then, the $s$-Fractional Laplacian operator $(-\Delta_v)^{s}$ is defined through the relation in the Fourier space
\begin{equation}\label{FLF}
      \mathcal{F}\big\{(-\Delta_v)^{s}f\big\}(\xi) 
   = |\xi|^{2s}\,\mathcal{F}\{f\}(\xi)\,, 
\quad s \in (0,1)\,.
\end{equation}
It is not difficult to prove that this definition is equivalent to the singular integral relation
\begin{equation}\label{FL}
     (-\Delta_v)^{s}f(v) 
  = c_{d,s} \int_{\R^{d-1}}\frac{f(v) - f(v+z)}{|z|^{d-1+2s}}\,\dz\,,
\end{equation}
where the constant is given by
\begin{equation}\label{Ac}
\frac{1}{c_{d,s}} = \int_{\mathbb{R}^{d}}\frac{1 - e^{-i\hat{\xi}\cdot z}}{|z|^{d+2s}}\,\dz>0 \,.
\end{equation}
\subsection{Bessel Potentials} \label{app:Bessel-potential}
\begin{Lemma}\label{app:convex}
The following  relation holds for any $s\in(0,1)$
\begin{equation}\label{app:bessel2}
     (-\Delta_v)^{s}\frac{1}{\langle v \rangle^{d-1-2s}}
=  \frac{c_{d,s}}{\langle v \rangle^{d-1+2s}}\,,
\end{equation}
with constant
\begin{equation*}
c_{d,s}=2^{2s}\,\frac{\Gamma\big(\frac{d-1}{2}+s\big)}{\Gamma\big(\frac{d-1}{2}-s\big)}\,.
\end{equation*}
\end{Lemma}
\begin{proof}
The Bessel Potential $B_{\alpha}(v)=\frac{1}{(2\pi)^{d-1}}\frac{1}{\langle v\rangle^{\alpha}}$, with $v\in\mathbb{R}^{d-1}$, has Fourier transform
\begin{equation}\label{app:bessel1}
    \mathcal{F}\{ B_{\alpha} \}(\xi) 
    = \frac{1}{ 2^{\frac{d+\alpha-3}{2} } \pi^{(d-1)/2} \Gamma(\frac{\alpha}{2}) } \,
       K_{ \frac{d-1-\alpha}{2} }\big(|\xi|\big) \, |\xi|^{\frac{\alpha-d+1}{2}}\,,
\end{equation}
where $K_{\nu}(z)$ is the modified Bessel function of the third kind of order $\nu\in\mathbb{R}$, refer to \cite[Section 3 and 4]{AS} for a short discussion on Bessel potentials and their Fourier transform.  We compute the left side of \eqref{app:bessel2} using \eqref{app:bessel1} with $\alpha=(d-1) - 2s$
\begin{equation}\label{app:e1bessel}
\mathcal{F}  \Big\{ (-\Delta_v)^{s}\frac{1}{\langle v \rangle^{d-1-2s}} \Big\}(\xi)  = |\xi|^{2s}\,\mathcal{F}\{ B_{d-1-2s} \}(\xi)
 = c^{1}_{d,s}\,|\xi|^{s}\,K_{s}(|\xi|). 
 \end{equation}
Furthermore, the right side of \eqref{app:bessel2} can be computed also with the same formula and $\alpha=(d-1) + 2s$
\begin{equation}\label{app:e2bessel}
\mathcal{F}  \Big\{\frac{1}{\langle v \rangle^{d-1+ 2s}} \Big\}(\xi)  =  c^{2}_{d,s}\,|\xi|^{s}\,K_{-s}(|\xi|) = c^{2}_{d,s}\,|\xi|^{s}\,K_{s}(|\xi|)\,, 
\end{equation}
where the last equality follows from the fact that $K_{\nu} = K_{-\nu}$ for any $\nu\in\mathbb{R}$.  Thus, from \eqref{app:e1bessel} and \eqref{app:e2bessel} it follows that
\begin{equation*}
\mathcal{F}  \Big\{ (-\Delta_v)^{s}\frac{1}{\langle v \rangle^{d-1-2s}} \Big\}(\xi)  = \frac{c^{1}_{d,s}}{c^{2}_{d,s}}\, \mathcal{F}  \Big\{\frac{1}{\langle v \rangle^{d-1+ 2s}} \Big\}(\xi)\,,
\end{equation*}
and the result follows after taking the inverse Fourier transform and calculating the constants.
\end{proof}
\subsection{Convergence of the operator $(-\Delta_{v})^{s}_{g}$}
\begin{Lemma}\label{app:Dapproxconv}
Let $\psi\in \mathcal{C}^{2}(\mathbb{R}^{d-1})$.  Then, for $g\lesssim1$ there exists an explicit $\beta_{s}>0$ only dependent on $s\in(0,1)$ such that for any $\epsilon\in(0,1)$
\begin{equation}\label{app:Dapprox:est1}
\Big|(-\Delta_{v})^{s}_{g}\psi - (-\Delta_{v})^{s}\psi\Big|\leq C_{d,s}\|\psi\|_{\mathcal{C}^{2}}(g-1)^{\epsilon\beta_{s}}\langle v \rangle^{\epsilon}\,,
\end{equation}
with the constant $C_{d,s}$ independent of $g$.  Moreover, if $\psi\in \mathcal{C}^{2}_{c}(\mathbb{R}^{d-1})$ previous estimate upgrades to
\begin{equation}\label{app:Dapprox:est2}
\Big|(-\Delta_{v})^{s}_{g}\psi - (-\Delta_{v})^{s}\psi\Big|\leq C_{d,s}\big|\text{supp}(\psi)\big|\|\psi\|_{\mathcal{C}^{2}}(g-1)^{\beta_{s}}\,,
\end{equation}
with explicit decay
\begin{equation}\label{app:Dapprox:est3}
\Big|(-\Delta_{v})^{s}_{g}\psi\Big| \leq C_{d,s}\frac{\big|\text{supp}(\psi)\big|\|\psi\|_{\mathcal{C}^{2}}}{\langle v \rangle^{d-1+2s}}\,.
\end{equation}
\end{Lemma}
\begin{proof}
Fix any $g\in[\frac{1}{2},1)$ and recall that we introduced the notation
\begin{equation*}
\delta_{g}(v,v') = (1+g)^{-1}\big((g-1)^{2}\langle v \rangle^{2} \langle v' \rangle^{2} + 4g|v'-v|^{2} \big)^{\frac{d-1}{2}+s} \,.
\end{equation*}
Thus,
\begin{align}\label{e1:app:Dapprox}
(-\Delta_{v})^{s}_{g}\psi(v) &= \int_{\mathbb{R}^{d-1}}\frac{\psi(v+z) - \psi(v)}{\delta_{g}(v,v+z)}\dz\nonumber\\
&\hspace{-1cm}=\tfrac{1}{2}\int_{\mathbb{R}^{d-1}}\frac{\psi(v+z)+\psi(v-z) - 2\psi(v)}{\delta_{g}(v,v+z)}\dz
\nonumber\\
&+\tfrac{1}{2}\int_{\mathbb{R}^{d-1}}\big(\psi(v-z) - \psi(v+z)\big)\frac{\delta_{g}(v,v+z) - \delta_{g}(v,v-z)}{\delta_{g}(v,v+z)\,\delta_{g}(v,v-z)}\dz\,.
\end{align}
We first prove that the second integral on the right side of \eqref{e1:app:Dapprox} goes to zero uniformly.  Indeed, using the inequality
\begin{equation}\label{e2.9:app:Dapprox}
\big|x^{\alpha} - y^{\alpha} \big|\leq \max\{1,\alpha/2\}\big(x^{\alpha-1} + y^{\alpha-1}\big)\,|x-y|\,,\quad \alpha\geq1\,,\;x,y\geq0\,,
\end{equation}
with $\alpha=\frac{d-1}{2}+s$, it follows that
\begin{align*}
\big|\delta_{g}(v,v+z) &- \delta_{g}(v,v-z)\big|\\
&\leq\frac{2\alpha}{1+g}(1-g)^{2}\langle v \rangle^{2}\big(\delta_{g}(v,v+z)^{1-1/\alpha} + \delta_{g}(v,v-z)^{1-1/\alpha}\big)\big|\langle v+z \rangle^{2} - \langle v-z \rangle^{2} \big|\\
&=\frac{4\alpha}{1+g}(1-g)^{2}\langle v \rangle^{2}\big(\delta_{g}(v,v+z)^{1-1/\alpha} + \delta_{g}(v,v-z)^{1-1/\alpha}\big)\big|v\cdot z \big| \,.
\end{align*}
As a consequence, 
\begin{align}\label{e3:app:Dapprox}
\frac{\delta_{g}(v,v+z) - \delta(v,v-z)}{\delta_{g}(v,v+z)\,\delta_{g}(v,v-z)}\leq &4\alpha(1+g)(1-g)^{2}\langle v \rangle^{2}\Bigg(\frac{1}{\delta_{g}(v,v+z)^{1/\alpha}\delta_{g}(v,v-z)}\nonumber\\
&\hspace{1cm} + \frac{1}{\delta_{g}(v,v-z)^{1/\alpha}\delta_{g}(v,v+z)}\Bigg)\big|v\cdot z \big|\,.
\end{align}
Furthermore, for any $v,\,z\in\mathbb{R}^{d-1}$ it follows that
\begin{align}\label{e3.1:app:Dapprox}
\delta_{g}(v,v\pm z)^{1/\alpha} &= \delta_{g}(v,v\pm z)^{(1-(s'-s))/\alpha}\delta_{g}(v,v\pm z)^{(s'-s)/\alpha}\nonumber\\
&\geq(g-1)^{2-2(s'-s)} \langle v \rangle^{2-2(s'-s)}\,|z|^{2(s'-s)}\,,\quad s'\in[s,1)\,.
\end{align}
The parameter $s'$ will be chosen in the sequel.  Using \eqref{e3.1:app:Dapprox} in \eqref{e3:app:Dapprox} it follows that
\begin{equation}\label{e3.2:app:Dapprox}
\frac{\delta_{g}(v,v+z) - \delta(v,v-z)}{\delta_{g}(v,v+z)\,\delta_{g}(v,v-z)}\leq C_{d,s}(1-g)^{2(s'-s)}\langle v \rangle^{2(s'-s)+1}\big|z\big|^{-(d-1)-2s'+1}\,,
\end{equation}
valid for any $v,\,z\in\mathbb{R}^{d-1}$.  Introduce a large radius $R\gg1$, then in the set $\{|v| \leq R\}$ one has directly from \eqref{e3.2:app:Dapprox}
\begin{align}\label{e3.3:app:Dapprox}
\Bigg|\int_{\mathbb{R}^{d-1}}\big(\psi(v-z) - \psi(v+z)\big)&\frac{\delta_{g}(v,v+z) - \delta_{g}(v,v-z)}{\delta_{g}(v,v+z)\,\delta_{g}(v,v-z)}\dz\Bigg| \nonumber\\
&\hspace{-3cm}\leq C_{d,s}(1-g)^{2(s'-s)}\langle v \rangle^{2(s'-s)+1}\int_{\mathbb{R}^{d-1}}\frac{\psi(v-z) - \psi(v+z)}{|z|^{d-1+2s'-1}}dz\nonumber\\
&\hspace{-4cm}\leq C_{d,s,s'}\,(1-g)^{2(s'-s)}\,R^{2(s'-s)+1}\,\big(\|\nabla_{v}\psi\|_{\infty} + \|\psi\|_{L^{\infty}}\big)\,,\quad \max\{s,\tfrac{1}{2}\}<s'<1\,.
\end{align}
For the last inequality, simply break the integral in the sets $\{|z|<1\}$ and $\{|z|\geq1\}$ and directly compute
\begin{align*}
\int_{\{|z|<1\}}\frac{\psi(v-z) - \psi(v+z)}{|z|^{d-1+2s'-1}}&dz \leq \|\nabla_{v}\psi\|_{\infty} \int_{\{|z|<1\}}\frac{1}{|z|^{d-1-2(1-s')}}dz = C_{d,s'}\|\nabla_{v}\psi\|_{\infty}\,,\\
\int_{\{|z|\geq1\}}\frac{\psi(v-z) - \psi(v+z)}{|z|^{d-1+2s'-1}}dz &\leq \|\psi\|_{\infty} \int_{\{|z|\geq1\}}\frac{1}{|z|^{d-1 + 2s' -1}}dz = C_{d,s'}\|\psi\|_{\infty}\,, \;\; s'>1/2\,.
\end{align*}
In the set $\{|v| > R\}$ break the integral in $\{|z| < |v|/2\}$ and $\{|z|\geq|v|/2\}$.  For the former, note that
\begin{equation*}
\delta_{g}(v,v\pm z)^{1/\alpha}\geq(g-1)^{2}\;\langle v \rangle^{2}\langle v \pm z  \rangle^{2}\geq \tfrac{1}{4}(g-1)^{2}\;\langle v \rangle^{4} \quad\text{whenever}\quad |z|\leq|v|/2\,,
\end{equation*}
therefore, using \eqref{e3:app:Dapprox} it holds that
\begin{align}\label{e3.4:app:Dapprox}
\Bigg|&\int_{\{|z|<|v|/2\}}\big(\psi(v-z) - \psi(v+z)\big)\frac{\delta_{g}(v,v+z) - \delta_{g}(v,v-z)}{\delta_{g}(v,v+z)\,\delta_{g}(v,v-z)}\dz\Bigg|\nonumber\\
&\leq 4\alpha\frac{(1+g)}{\langle v \rangle}\int_{\{|z|<|v|/2\}}\big(\psi(v-z) - \psi(v+z)\big)\frac{1}{|z|^{d-1+2s-1}}\dz\nonumber\\
&\leq C_{d,s}\frac{(1+g)}{\langle v \rangle}
\Bigg\{
\begin{array}{lc}
\|\nabla_{v}\psi\|_{\infty} + \max\{1,\langle v \rangle^{1-2s}\}\|\psi\|_{\infty} & \text{if}\;s\neq1/2\\
\|\nabla_{v}\psi\|_{\infty} + \ln\langle v \rangle\|\psi\|_{\infty}  & \text{if}\;s=1/2
\end{array}\nonumber\\
&\leq C_{d,s}\big(\|\nabla_{v}\psi\|_{\infty} + \|\psi\|_{\infty}\big)\frac{(1+g)}{R^{s}}\,.
\end{align}
As before, for the last inequality one breaks the integral in the sets $\{|z|<1\}$ and $\{1\leq|z|<|v|/2\}$ and directly computes the integrals.  For the latter one simply estimate
\begin{align}\label{e3.5:app:Dapprox}
\Bigg|\int_{\{|z|\geq|v|/2\}}\big(\psi(v-z) &- \psi(v+z)\big)\frac{\delta_{g}(v,v+z) - \delta_{g}(v,v-z)}{\delta_{g}(v,v+z)\,\delta_{g}(v,v-z)}\dz\Bigg|\nonumber\\
&\leq 2\|\psi\|_{\infty}\int_{|z|\geq|v|/2}\frac{1}{|z|^{d-1+2s}}\dz = C_{s}\,\frac{\|\psi\|_{\infty}}{\langle v \rangle^{2s}}\leq C_{s}\,\frac{\|\psi\|_{\infty}}{R^{2s}}\,.
\end{align}
Gathering \eqref{e3.3:app:Dapprox}, \eqref{e3.4:app:Dapprox} and \eqref{e3.5:app:Dapprox}
\begin{align}\label{e3.6:app:Dapprox}
\Bigg|\int_{\{|z|\geq|v|/2\}}\big(\psi(v-z) &- \psi(v+z)\big)\frac{\delta_{g}(v,v+z) - \delta_{g}(v,v-z)}{\delta_{g}(v,v+z)\,\delta_{g}(v,v-z)}\dz\Bigg|\nonumber\\
&\hspace{-1.5cm}\leq C_{d,s,s'}\big(\|\nabla_{v}\psi\|_{\infty} + \|\psi\|_{\infty}\big)\times\nonumber\\
&\Big((g-1)^{2(s'-s)}R^{2(s'-s)+1}+\frac{1}{R^{s}}\Big)\,,\quad \max\{s,\tfrac{1}{2}\}<s'<1\nonumber\\
&\leq C_{d,s,s'}\|\psi\|_{\mathcal{C}^{1}} (g-1)^{\beta_{1}}\,, \quad \beta_{1}=\frac{2(s'-s)s}{2(s'-s)+1+s}\,.
\end{align}
The rate $\beta_{1}$ is  obtained by choosing $R>0$ (large) such that $(g-1)^{2(s'-s)}R^{2(s'-s)+1}=1/R^{s}$.\\

\noindent
Let us now focus in the first integral on the right side of \eqref{e1:app:Dapprox}.  The procedure is similar as before, first introducing a radius $R>0$ (not necessarily large this time) and considering the region $\{|v|\leq R\}$.  Then, inequality \eqref{e2.9:app:Dapprox} and some direct computations leads to (recall that $\alpha=\frac{d-1}{2}+s$)
\begin{align}\label{e3.7:app:Dapprox}
\tfrac{1}{2}\Bigg|&\int_{\mathbb{R}^{d-1}}\big(\psi(v+z)+\psi(v-z) - 2\psi(v)\big)\Bigg(\frac{1}{\delta_{g}(v,v+z)}-\frac{1+g}{(4g|z|^{2})^{\alpha}}\Bigg)\dz\Bigg|\nonumber\\
&\leq\big\|\partial^{2}_{v}\psi\|_{\infty}\big((g-1)R^{2}\big)^{2(1-s')}\int_{\{|z|\leq1\}}\frac{1}{|z|^{d-1-2(s'-s)}}\dz\nonumber\\
&\hspace{.5cm}+\|\psi\|_{\infty}\big((g-1)R^{2}\big)^{2(1-s')}\int_{\{|z|\geq1\}}\frac{1}{|z|^{d-1+2s}}\dz\nonumber\\
&\hspace{1cm}\leq C_{d,s,s'}\|\psi\|_{\mathcal{C}^{2}}\big((g-1)R^{2}\big)^{2(1-s')}\,,\quad \max\{s,\tfrac{1}{2}\}<s'<1\,.
\end{align}
In the set $\{|v|>R\}$ simply use the rough estimate for any $\epsilon\in(0,1)$
\begin{equation*}
\Bigg|\frac{1}{\delta_{g}(v,v+z)}-\frac{1+g}{(4g|z|^{2})^{\alpha}}\Bigg|\leq \frac{2(1+g)}{(4g|z|^{2})^{\alpha}}\leq \frac{\langle v \rangle^{\epsilon}}{R^{\epsilon}}\frac{2(1+g)}{(4g|z|^{2})^{\alpha}}\,,
\end{equation*}
to conclude that
\begin{align}\label{e3.8:app:Dapprox}
\tfrac{1}{2}\Bigg|\int_{\mathbb{R}^{d-1}}&\big(\psi(v+z)+\psi(v-z) - 2\psi(v)\big)\Bigg(\frac{1}{\delta_{g}(v,v+z)}-\frac{1+g}{(4g|z|^{2})^{\alpha}}\Bigg)\dz\Bigg| \nonumber\\
&\leq C\, \frac{\langle v \rangle^{\epsilon}}{R^{\epsilon}}\Bigg(\|\partial^{2}_{v}\psi\|_{\infty}\int_{\{|z|<1\}}\frac{1}{|z|^{d-1- 2(1-s)}}\dz + \|\psi\|_{\infty}\int_{\{|z|\geq1\}}\frac{1}{|z|^{d-1+2s}}\dz\Bigg)\nonumber\\
&\leq C_{d,s}\|\psi\|_{\mathcal{C}^{2}}\frac{\langle v \rangle^{\epsilon}}{R^{\epsilon}}\,.
\end{align}
Thus, gathering \eqref{e3.7:app:Dapprox} and \eqref{e3.8:app:Dapprox} it follows that
\begin{align}\label{e3.9:app:Dapprox}
\tfrac{1}{2}\Bigg|\int_{\mathbb{R}^{d-1}}&\big(\psi(v+z)+\psi(v-z) - 2\psi(v)\big)\Bigg(\frac{1}{\delta_{g}(v,v+z)}-\frac{1+g}{(4g|z|^{2})^{\alpha}}\Bigg)\dz\Bigg|\nonumber\\
&\leq C_{d,s,s'}\|\psi\|_{\mathcal{C}^{2}}\Big(\big((g-1)R^{2}\big)^{2(1-s')}+ \frac{\langle v \rangle^{\epsilon}}{R^{\epsilon}}\Big)\nonumber\\
&\hspace{1.5cm}\leq  C_{d,s,s'}\|\psi\|_{\mathcal{C}^{2}}(g-1)^{\epsilon\beta_{2}}\langle v \rangle^{2\epsilon\beta_{2}}\,,\quad \beta_{2}=\frac{2(1-s')}{4(1-s')+\epsilon}\,.
\end{align}
The last inequality follows choosing $R>0$ such that $\big((g-1)R^{2}\big)^{2(1-s')}=\frac{\langle v \rangle^{\epsilon}}{R^{\epsilon}}$.  Since $\frac{2(1-s')}{4(1-s')+1}\leq\beta_{2}\leq \frac{1}{2}$ estimate \eqref{app:Dapprox:est1} follows from \eqref{e3.6:app:Dapprox} and \eqref{e3.9:app:Dapprox} after choosing $\max\{s,\tfrac{1}{2}\}<s'<1$.\\

\noindent
Finally, estimates \eqref{app:Dapprox:est2} and \eqref{app:Dapprox:est3} follow from the fact that the all integrals vanish in the regions $\{|z|\leq |v|/2\}$ whenever $|v|\geq 2\,\text{diam}\big(\text{supp}(\psi)\big)$.  For instance
\begin{align*}
\Big|(-\Delta_{v})^{s}_{g}\psi(v)\Big| &= \Big|\int_{\mathbb{R}^{d-1}}\frac{\psi(v+z) - \psi(v)}{\delta_{g}(v,v+z)}\dz\Big| \nonumber\\
&=  \Big|\int_{|z|\geq|v|/2}\frac{\psi(v+z)}{\delta_{g}(v,v+z)}\dz\Big|\leq C_{d,s}\frac{\big|\text{supp}(\psi)\big|\|\psi\|_{\infty}}{|v|^{d-1+2s}}\,.
\end{align*}
Meanwhile, in the region $|v|\leq 2\,\text{diam}\big(\text{supp}(\psi)\big)$ one clearly has
\begin{equation*}
\big|(-\Delta_{v})^{s}_{g}\psi(v)\big|\leq C_{d,s}\big|\text{supp}(\psi)\big|\|\psi\|_{\mathcal{C}^{2}}\,.
\end{equation*}
\end{proof}

\end{appendices}

\subsection*{Acknowledgements.} R. Alonso thanks the Conselho Nacional de Desenvolvimento Cient\'{i}fico e Tecnol\'{o}gico (CNPq) - Brazil.  He also thanks the hospitality of W. Sun and the Department of Mathematics at Simon Fraser University. The research of W.Sun was supported in part by the
  Simon Fraser University President's Research Start-up Grant PRSG-877723   and NSERC Discovery Individual Grant \#611626.

\end{document}